\documentclass{article}
\usepackage{amsthm}
\usepackage{graphicx} 
\usepackage{amsmath}
\usepackage{amssymb}
\usepackage{amsfonts}
\usepackage[x11names,dvipsnames]{xcolor}

\usepackage[doi=false,url=false,backref=true,,sorting=nyt,sortcites]{biblatex}
\usepackage[colorlinks=True]{hyperref}
\definecolor{cadmiumgreen}{rgb}{0.0, 0.42, 0.24}
\definecolor{hotpink}{rgb}{1.0, 0.41, 0.71}
\definecolor{cgreen}{RGB}{0, 180, 100}
\hypersetup{
    colorlinks,
    linkcolor=cgreen,
    citecolor=cgreen,
    urlcolor=blue,
}
\usepackage[nameinlink,capitalize]{cleveref}
\usepackage[margin=1in]{geometry}
\usepackage{csquotes}
\usepackage{todonotes}
\usepackage{subcaption}
\usepackage{pgfplots}
\usepackage{booktabs}
\usepackage{multirow}
\usepackage{tabularx}
\usepackage{array}
\usepackage{authblk}
\usepackage{bm}
\usepackage{enumitem}

\crefname{algocf}{Algorithm}{Algorithms}
\usepackage{algorithm}
\usepackage{algpseudocode}

\usepackage{tikz}
\usetikzlibrary{spy}
\tikzset{annot/.style={draw=black,fill=white,text=black}}
\pgfplotsset{
  compat=newest, 
  tinyticks/.style={
every axis/.append style={font=\scriptsize},
  }
}
\usetikzlibrary{matrix}

\newcommand{\R}{\mathbb{R}}
\newcommand{\N}{\mathbb{N}}
\newcommand{\X}{\mathcal{X}}
\newcommand{\LS}{\mathcal{S}}
\newcommand{\budget}{\varepsilon}
\newcommand{\xadv}{\bar{x}}
\newcommand{\pert}{\Delta}
\newcommand{\dx}{\Delta}
\newcommand{\sign}{\operatorname{sign}}
\newcommand{\cidx}{\kappa}
\newcommand{\net}{h}
\newcommand{\obj}{f}
\newcommand{\denseNES}{\pi}
\newcommand{\eps}{\varepsilon}

\newcommand{\nc}{\color{black}}
\newcommand{\pnabla}{\overset{\text{\tiny{ES}}}{\nabla}}
\newcommand{\cnabla}{\overset{\text{\tiny{CH}}}{\nabla}}

\newcommand{\abs}[1]{\left|#1\right|}
\newcommand{\con}{\mathbf{c}}
\newcommand{\en}[1]{\uppercase{#1}}
\DeclareMathOperator*{\argmax}{arg\ max}

\DeclareMathOperator*{\argmin}{arg\,min}

\newcommand{\st}{\,:\,}
\newcommand{\de}{\,\mathrm{d}}

\definecolor{mediumelectricblue}{rgb}{0.21, 0.61, 0.99}

\makeatletter
\define@key{todonotes}{TR}[]{%
    \setkeys{todonotes}{author=TR,color=mediumelectricblue,size=scriptsize}}%
\makeatother 

\makeatletter
\define@key{todonotes}{LB}[]{%
    \setkeys{todonotes}{author=LB,color=pink,size=scriptsize}}%
\makeatother 

\newcommand{\rev}{\color{blue!80!black}}

\def\subfmult{.75}
\title{Consensus-based optimization for closed-box adversarial attacks and a connection to evolution strategies}

\author[1]{Tim Roith\thanks{Corresponding author: \href{mailto:tim.roith@tum.de}{tim.roith@tum.de}}}
\author[2]{Leon Bungert}
\author[3]{Philipp Wacker}

\affil[1]{TUM School of Computation, Information and Technology and Munich Center for Machine Learning (MCML), Boltzmannstr. 3, 85748 Garching bei München, Germany}
\affil[2]{Institute of Mathematics, Center for Artificial Intelligence and Data Science (CAIDAS), University of Würzburg, Emil-Fischer-Str. 40, 97074 Würzburg, Germany}
\affil[3]{School of Mathematics and Statistics, University of Canterbury}
\date{\today}

\renewcommand{\rev}{}
\newtheorem{theorem}{Theorem}[section]
\newtheorem{proposition}[theorem]{Proposition}

\newtheorem{remark}[theorem]{Remark}
\newtheorem{corollary}[theorem]{Corollary}

\begin{document}

\maketitle
\begin{abstract}
	
	Consensus-based optimization (CBO) has established itself as an efficient gradient-free optimization scheme, with attractive mathematical properties, such as mean-field convergence results for non-convex loss functions. 
	In this work, we study CBO in the context of closed-box adversarial attacks, which are imperceptible input perturbations that aim to fool a classifier, without accessing its gradient. Our contribution is to establish a connection between the so-called consensus hopping as introduced by \citeauthor{riedl2026gradient} and natural evolution strategies (NES) commonly applied in the context of adversarial attacks and to rigorously relate both methods to gradient-based optimization schemes. 
	Beyond that, we provide a comprehensive experimental study that shows that despite the conceptual similarities, CBO can outperform NES and other evolutionary strategies in certain scenarios. 
	
	\par\vskip\baselineskip\noindent
	\textbf{Keywords}: adversarial attacks, closed-box attacks, consensus-based optimization, gradient-free optimization\\
	\textbf{AMS Subject Classification}: 65K10, 68Q32, 65K15, 90C26
\end{abstract}
\section{Introduction}


\subsection{Adversarial attacks and defenses}

We consider a classification task of assigning a label $\cidx\in\{1,\ldots,K\}, K\in\N$ to inputs $x\in \mathcal{X}$. We employ a classifier $\net:\mathcal{X}\to\Delta^K$, where $\Delta^K$ denotes the $K$-dimensional unit simplex, and obtain the classification via the maximum likelihood estimator, $\net^{\text{MLE}}(x) := \argmax_{\cidx=1,\ldots,K} \net(x)_\cidx$. Given an input $x\in\mathcal{X}$ and its true label $\cidx\in\{1,\ldots,K\}$, the input $\xadv\in\X$ is called an \textit{adversarial example} if
\begin{align*}
\net^{\text{MLE}}(\xadv) \neq \cidx\qquad\text{and}\qquad \xadv \text{ is \enquote{similar} to } x.
\end{align*}
The requirement of $\xadv$ being similar to $x$ merely serves as an intuition, where originally the authors in \cite{szegedy2013intriguing,goodfellow2014explaining} demand that the difference between $\xadv$ and $x$ is imperceptible to the human eye. 
In practice, one cannot rely on human perception and instead chooses a function $\phi:\X\times\X\to[0,\infty)$ that indicates the similarity between two images.  In the following, we focus on cases where $\phi$ is induced by some $\ell^p$-norm, i.e., $\phi=\abs{\cdot-\cdot}_p$ with $p\geq 1$. 
Letting $\budget>0$ denote the so-called adversarial budget, one formulates the following attack problem,
\begin{align}\label{eq:adv}\tag{AdvAtt}
\rev{}\argmax_{\xadv \st \phi(\xadv, x)\leq \varepsilon} 
\nc{}
\ell(\net(\xadv), \cidx),
\end{align}
\rev{}where the function $\ell:\Delta^K\times \{1,\ldots,K\}\to\R$ measures the disagreement between $\net(\xadv)$ and $\cidx$.\nc{} 
This procedure is called an \textit{untargeted} attack, which means that we are not interested in the specific output of $\net^\text{MLE}_\theta(\xadv)$ as long as it is different from $\cidx$. This is opposed to targeted attacks, where we pick a label $\bar{\cidx} \neq \cidx$ and want to find $\xadv$ such that $\net^\text{MLE}(\xadv) = \bar{\cidx}$. The resulting optimization problem is similar to \eqref{eq:adv}, only exchanging $\ell(\net(\xadv),\kappa_x)$ with $-\ell(\net(\xadv), \bar \kappa)$.

With the goal of making classifiers robust against adversarial attacks, many approaches have been considered in the literature. 
A conceptually straightforward approach is to limit the Lipschitz constant of neural networks with the goal of achieving certifiable robustness guarantees, see, e.g.,~\cite{finlay2019lipschitz,fazlyab2023certified,huang2021training,bungert2021clip,zhang2022rethinking,anil2019sorting}.
However, restricting the Lipschitz constant often leads to a high degree of conservatism, and consequently to poor accuracy on unperturbed data. The most prominent approach to achieve adversarial robustness is called adversarial training \cite{goodfellow2014explaining,madry2017towards} and\rev{} consists augmenting the training data with adversarial attacks of the current classifier adaptively during training\nc{}.
While the algorithm itself is now a decade old, the theoretical \rev{}a first theoretical understanding of adversarial training has been obtained much more recently\nc{}, e.g., with existence proofs~\cite{awasthi2023existenceadversarialbayesclassifier,bungert2021clip}, relations to optimal transport~\cite{trillos2023multimarginal,trillos2024optimal}, and the study of its asymptotic regularization effect~\cite{bungert2024gamma,bungert2024mean}.

%

\subsection{Finding adversarial examples}

The first studies concerning adversarial attacks \cite{goodfellow2014explaining,szegedy2013intriguing} proposed the so-called fast gradient sign method that involves a single optimization step,
\begin{align}\label{eq: FGSM}\tag{FGSM}
\xadv= x+\budget\, \sign(\nabla_x \ell(\net(x), \cidx)),
\end{align}
\rev{}where the gradient is just taken with respect to the $x$-dependency in $h(x)$, not in the label $\cidx$.\nc{}
This one-step attack respects the budget constraint $\phi(\xadv,x)\leq \budget$ if $\phi=\abs{\cdot-\cdot}_\infty$ equals the $\ell^\infty$ metric, however, as pointed out in \cite{weigand2024adversarialflowsgradientflow} other gradient normalizations $N:\R^d\to B^p_1$ can be employed instead of the $\sign$ operation. The set $B^p_r(x)$ denotes the closed $\ell^p$ ball around $x$ with radius $r$, and we define $B^p_r := B^p_r(0)$. One-step attacks are especially attractive for adversarial training, since every training datum is attacked in every epoch. However, when we are only interested in the attack problem, it is also feasible to consider more than one iteration. We then have to ensure that the iterates $x^{(k)}$ still fulfill the constraint $x^{(k)}\in B_\budget^p(x)\cap\X$ for every $k\in \N$. \rev{}To enforce this, we consider \textit{projected} gradient descent with step size $\tau>0$\nc{}
\begin{align*}
x_{k+1} = \Pi_{B_\budget^p(x)}\left(x_{k} + \tau\, N(\nabla_x \ell(\net(x^{(k)}), \cidx))\right),
\end{align*}
where $\Pi_{B_\budget^p(x)}$ denotes the orthogonal projection onto the set $B_\budget^p(x)$. For $\ell^\infty$ attacks, one can alternatively employ an explicit parameterization, as in \cite{carlini2017towards}, using the fact that $B_\budget^\infty(x) = \{x + \budget \tanh(w): w\in\R^d\}$, 
\rev{}
with a componentwise application of $\tanh$ on $\R^d$.\nc

%
%

\paragraph{Closed-box attacks}

So far, we described methods to construct adversarial examples that employ the gradient $\nabla_x \ell(f(x),y)$. Such methods are also referred to as open- or white-box attacks\footnote{In the literature, the terminology white- vs. black-box optimization is more common. In this paper, we use the more appropriate terms open- and closed-box.}, since the computation of the mentioned gradient requires full knowledge on the network's architecture and weights. Often, however, this knowledge cannot be assumed, which then yields the scenario of \emph{closed-box} attacks, where we refer to \cite{papernot2017practical} for one of the first studies in this direction. %
In our setting, we assume only access to the full output $\net(x)$ for any input $x$. Note that we do not consider the more restrictive cases, where only $\net^\text{MLE}(x)$ is known, which are called \emph{decision-based} attacks.  Moreover, our work falls into the category of \textit{query-based} attacks, i.e., we employ the output of the original neural network to solve \eqref{eq:adv}. This is opposed to \textit{transfer-based} attacks that produce attacks against a given model (typically in an open-box fashion) which should then also be a valid attack against a closed-box model. %
We refer to \cite{suya2024sok} for a more detailed taxonomy of different attack scenarios.
In the following, we only review the most important references for the present work. For exhaustive literature reviews and surveys, we refer to \cite{Zheng2023,mahmood2021back,Wang2022,bhambri2019survey}. Furthermore, in \cref{tab:lowresmore} we comment on more works in this direction which are not directly related to our setup.
\paragraph{Attacks based on evolution strategies} 
Evolution strategies (ES) employ a population of particles exploring the state space and use function evaluations for the update of the population, with the goal of contracting the population in regions of optimality. One of the first studies in this direction was given in \cite{rechenberg1978evolutionsstrategien}, which led to the development of the so-called \emph{(1 + 1)-ES}, where in each iteration a new candidate solution is produced by randomly perturbing the current iterate. This candidate is then made the new iterate if it improves with respect to some objective function. Further important developments in this line of optimization literature involve the \emph{Covariance matrix adaptation evolution strategy} (CMA-ES) \cite{hansen2001completely} and natural evolution strategies \cite{wierstra2014natural,glasmachers2010exponential,salimans2017evolution}. We refer to \cite{slowik2020evolutionary} for an overview of this topic\rev{} and also to \cite{auger2005convergence,akimoto2022global,diouane2015globally} for convergence results on evolutionary strategies. In particular, a recent work, \cite{fornasier2026consensus}, establishes global convergence for different classes of ES through the lens of CBO.\nc{} %
The application of ES to adversarial attacks was first studied in \cite{ilyas2018black} for the so-called \emph{Natural evolution strategy} (NES), and further extended to other schemes in \cite{meunier2019yet,ilie2021evoba,qiu2021black}. Regarding adversarial NES, we remark that variations of the original formulation in \cite{ilyas2018black} were given in \cite{li2019nattack}, but we do not consider these here. It should also be noted that before \cite{ilyas2018black}, the authors in \cite{sharif2016accessorize} employed particle swarm optimization (PSO), another ensemble method, in their setup to obtain adversarial attacks. The use of PSO for adversarial attacks was further studied in \cite{bhagoji2017exploring}.
\rev{}
Beyond evolution strategies and particle-based optimizers, there is substantial literature on model-based derivative-free optimization (DFO), also in the context of adversarial attacks, see \cite{ughi2022empirical}. For our work here, we mostly concentrate on the comparison between different particle-based approaches, but refer to \cite{larson2019derivative,conn2009introduction} for an introduction to DFO. 
\nc{}

\subsection{Contribution and outline}

The goal of this paper is to analyze the connection between the natural evolution strategy (NES) and consensus-based optimization (CBO). 
The main insight here is that so-called Consensus Hopping (CH) as introduced in \cite{riedl2026gradient}, which is a derivative of the original CBO algorithm, is very similar to a class of evolution strategies. 
Beyond the connection of CH and NES, we also examine the performance of the original CBO algorithm for the adversarial attack problem. 
Here, our main insight is that CBO outperforms comparable algorithms in \enquote{easier} attack scenarios, including targeted attacks on the CIFAR-10 dataset \cite{krizhevsky2009learning} and untargeted attacks on the ImageNet dataset \cite{deng2009imagenet}. 
However, for harder settings, e.g., targeted attacks on ImageNet, CBO performs the worst out of the considered schemes while NES and CH are on par. This observation is underlined by experiments on different attack types and scenarios. \rev{}The different schemes are formally introduced in \cref{sec:CBOvsNES}, however, here we briefly outline the intuition and key differences:\nc{}
\paragraph{From CBO to CH} CBO takes an ensemble view, where the evolution of particles towards a minimizer of an objective function is considered. In contrast to that, CH can be expressed as the evolution of a single point, around which new particles are sampled in every step. Based on the derivation of \cite{riedl2026gradient} we show in \cref{sec:CBOvsNES} how CBO can be modified to obtain CH:
\begin{enumerate}
\item CBO allows the particles to drift towards the \rev{}consensus point\nc{} with a certain speed. 
By choosing it to be infinity, the particles directly \enquote{hop} onto the consensus point.
\item The noise scale of each particle in CBO depends on its distance to the consensus point. CH replaces this with a fixed scale for all particles.
\end{enumerate}
\paragraph{From CH to NES}
In NES the gradient is approximated by taking a weighted mean of samples around the current estimate. In CH, this is done very similarly, but the objective function $f$ is replaced by an exponential rescaling and the mean is replaced by a normalized sum. However, the numerical examples in \cref{sec:num} show that this only leads to minor differences in the performance of NES and CH.
We underline this empirical observation with a theoretical result, proving that in the regime of infinitely many Gaussian samples, both schemes approximate the gradient of the objective with comparable error terms.
\vspace{1em}

\section{Consensus-based optimization and evolution strategies}
In this section, we first introduce consensus-based optimization. Further, using the concept of consensus hopping as introduced in \cite{riedl2026gradient}, we establish a connection between evolution strategies and CBO.

\subsection{Consensus-based optimization}
Consensus-based optimization (CBO) was introduced in \cite{Pinnau2017} as a derivative-free scheme for global optimization. 
With this aim of minimizing an objective function $\obj:\R^d\to\R$ one solves the following \rev{}system of\nc{} stochastic differential equation (SDE) for an ensemble of particles $X=(x^{(1)},\ldots, x^{(N)})$:
\begin{align}\label{eq:SDE}
\de x^{(n)} = -\rev{}\lambda\nc{}\left(x^{(n)} - \con(X)\right)\de t + \sigma\,
\rev{} \mathsf{D}(x^{(n)} - \con(X)) \nc{}\de W^{(n)}\rev{},
\quad 
\text{for all}\quad n\in\{1,\ldots,N\}.
\end{align}
Here, \rev{}$\lambda>0$ denotes the drift coefficient,\nc{} $(W^{(1)},\dots,W^{(N)})$ are independent Brownian motions, \rev{}and 
$\mathsf{D}:\R^d\to\R^{d\times d}$ determines the noise model, where the most prominent choices are the isotropic noise $\mathsf{D}(z) = |z|\mathbb{I}_{d\times d}$ or the anisotropic one from \cite{carrillo2021consensus}, with $\mathsf{D}(z) :=
\operatorname{diag}(z)$. The consensus point $\con(X)$ with respect to the ensemble $X$ is defined as 
%
%
\begin{align}\label{eq:con}
\con(X) := \frac{\sum_{n=1}^N x^{(n)} \exp(-\alpha \obj(x^{(n)}))}{\sum_{n=1}^N \exp(-\alpha \obj(x^{(n)}))}.
\end{align}
This quantity is an ensemble average, weighted by the particles' respective objective function, with lower-objective particles carrying a larger weight. 
\rev{}The parameter $\alpha\geq 0$ influences the strength of the weighting. In the limit $\alpha\to 0$, $\con(X)$ corresponds to the standard Euclidean mean. For CBO the limit $\alpha\to\infty$ is more relevant, where in the discrete setting $\con(X)$ then corresponds to $\argmin_{n=1,\ldots,N} f(x^{(n)})$, given that all particles yield distinct objective values. In the mean-field setting, \cite{fornasier2024consensus} derive a quantitative version of the Laplace principle which shows that under certain assumptions $\con(X)$ converges to the minimizer of $f$ for $\alpha\to\infty$.
\nc{}
We can then see that the dynamics of \eqref{eq:SDE} consist of deterministic contraction of all particles to their joint weighted mean, and additional distance-weighted explorative random perturbations. The mean-field version of the SDE in \eqref{eq:SDE} allows to study the convergence behavior under mild assumption on the objective function $\obj$, see e.g. \cite{carrillo2018analyticalframeworkconsensusbasedglobal,fornasier2024consensus,fornasier2022convergence,grassi2023mean,riedl2024leveraging,riedl2024mathematical,wang2025mathematical,bae2022constrained,fornasier2025regularity}. For a convergence study in the finite-particle regime, we refer to \cite{ko2022convergence,byeon2024discrete,bellavia2025discrete}. 
The time-discrete Euler--Maruyama scheme for the \rev{}$n$-th particle in\nc{} \eqref{eq:SDE} with time step size $\tau>0$ is given by 
\begin{align}\label{eq:discr_CBO}\tag{CBO}
x^{(n)}_{k+1} &= 
x^{(n)}_k
-\tau\lambda(x_k^{(n)} - \con(X_k)) + \sqrt{\tau}\sigma\,  \rev{}\mathsf{D}(x^{(n)}_k-\con(X_k))\nc{}\xi_k^{(n)},\qquad \xi_k^{(n)} \sim \mathcal{N}(0,\mathbb I_{d\times d}),
\end{align}
where by $X_k:=(x_k^{(1)},\dots,x_k^{(N)})$ we denote the particle ensemble in the $k$-th iteration of the scheme.
Many extensions of this scheme for different application scenarios have been proposed, including constrained optimization \cite{beddrich2024constrained,CBO_reg_giacomo,CBO_sphere2,Fornasier2020ConsensusbasedOO,CBO_reg_jose,carrillo2024interacting,herty2025micro}, sparse optimization and mirror descent \cite{bungert2025mirrorcbo}, bilevel optimization \cite{Garcia_Trillos2025-ix,trillos2024cb}, stochastic optimization \cite{bonandin2024consensus}, saddle point problems \cite{borghi2024particle,huang2024consensus} and problems with multiple global minima \cite{bungert2024polarized,fornasier2024pde}. Furthermore, it was also extended to the sampling of uni- and multimodal distributions in \cite{carrillo2022consensus,bungert2024polarized}, respectively. 
Notably, so far, CBO has been used in the context of adversarial attacks and defenses only once \cite{Garcia_Trillos2025-ix}, for solving an adversarially robust federated learning task.
Part of the popularity of CBO-based methods is grounded on the fact that they are backed by a rich mathematical convergence theory for non-convex optimization, based on mean-field limits and numerical methods for stochastic differential equations, see~\cite{fornasier2024consensus,carrillo2018analyticalframeworkconsensusbasedglobal} and many of the references above.

\subsection{Natural evolution strategies}

A popular method to solve \eqref{eq:adv} in the closed-box scenario is the so-called natural evolution strategy (NES) \cite{wierstra2014natural,ilyas2018black}. One aims to maximize a fitness function $\obj:\R^d\to\R$. However, instead of directly optimizing $\obj$, NES instead maximizes the expected value with respect to a parametrized probability density $\denseNES(\cdot; \theta):\R^d\to\R$:
\begin{align*}
\max_\theta J(\theta) := \max_\theta \int \obj(z) \denseNES(z; \theta)\de z.
\end{align*}
Employing the \enquote{log-likelihood trick} the authors in \cite{wierstra2014natural} first calculate
\begin{align*}
\nabla_\theta J(\theta) 
=
\int \obj(z) \nabla_\theta \log \denseNES(z;\theta)\, \denseNES(z;\theta)\de z
\end{align*}
and propose the gradient approximation
\begin{align}\label{eq:NESGrad}
\nabla_\theta J(\theta) \approx \pnabla J(\theta; Z) := 
\frac{1}{N} \sum_{n=1}^N \obj(z^{(n)}) \nabla_\theta \log \denseNES(z^{(n)};\theta), 
\end{align}
where each point in the ensemble $Z=(z^{(1)},\ldots, z^{(N)})$ is distributed according to $\denseNES(\cdot, \theta)$. This is related to Stein's lemma, and is also used in the context of Stein variational gradient descent \cite{liu2016stein}. For invariance with respect to the concrete parameterization, one considers the \textit{natural gradient} \cite{amari1998natural},
%
$
\tilde{\nabla}_\theta J = F(\theta)^{-1} \nabla_\theta J,
$
%
where 
$F(\theta) := \int\nabla_\theta \log \denseNES(z,\theta)\nabla_\theta \log \denseNES(z,\theta)^\top\denseNES(z;\theta)\de z$
denotes the Fisher information matrix. In \cite{wierstra2014natural}, this is then used to formulate the NES update,
\begin{align}\tag{NES}\label{eq:NES}
\left.
\begin{aligned}
Z&=(z^{(1)},\ldots,z^{(N)})\quad\text{with}\quad z^{(n)} \sim \denseNES(\cdot, \theta_k),\rev{}\text{ for all } n\in\{1,\ldots,N\},\\
\theta_{k+1} &= \theta_k + \eta F(\theta_k)^{-1} \pnabla J(Z;\theta)\rev{},\nc{}
\end{aligned}\, \right\}
\end{align}
\rev{}for a step size $\eta>0$.\nc{} 
If the Fisher information matrix is not available analytically, it can be approximated similarly \rev{}using samples $Z=(z^{(1)},\dots,z^{(N)})$ distributed according to $\pi(\cdot;\theta)$ via\nc{}
\begin{align*}
    F(\theta) \approx 
    {F}(\theta;Z) := \frac{1}{N}
    \sum_{n=1}^N
    \nabla_\theta \log \denseNES(z^{(n)};\theta)
    \nabla_\theta \log \denseNES(z^{(n)};\theta)^\top.
\end{align*}
For adversarial attacks, it was proposed in \cite{ilyas2018black} to choose a Gaussian parameterization $\pi(\cdot,\theta)\equiv\mathcal{N}(\mu,\sigma^2\mathbb{I}_{d\times d})$ with fixed covariance $\sigma^2\mathbb{I}_{d\times d}$, and only optimize over the means, i.e., $\theta=\mu\in\R^d$. In this case, one calculates $\nabla_\mu \log \denseNES(z;\mu) = \frac{1}{\sigma^2}(z - \mu),
F= \frac{1}{\sigma^2} \mathbb{I}_{d\times d}$ and thus 
\begin{gather*}
\tilde{\nabla}_\mu J(\mu) \approx
\frac{1}{N} \sum_{n=1}^N \obj(z^{(n)}) (z^{(n)} - \mu) = 
\frac{\sigma}{N} \sum_{n=1}^N \obj(\mu + \sigma \xi^{(n)})\, \xi^{(n)},\qquad \xi^{(n)}\sim\mathcal{N}(0,\mathbb{I}_{d\times d}) 
\end{gather*}
which can be directly plugged into \labelcref{eq:NES}, to obtain \rev{}Gaussian NES\nc{}
\begin{align}\tag{GNES}\label{eq:GNES}
\mu_{k+1} &= \mu_k + \frac{\eta \sigma}{N} 
\sum_{n=1}^N \obj(\mu_k + \sigma \xi_k^{(n)})\, \xi_k^{(n)}\quad \xi_k^{(n)} \sim \mathcal{N}(0, \mathbb{I}_{d\times d}).
\end{align}
Note that in \cite{ilyas2018black} the authors did not use the natural gradient and hence arrive at a different scaling with $\frac{\eta}{\sigma N}$ instead of $\frac{\eta\sigma}{N}$ in front of the sum, which only changes the step size $\eta$.

In the following, we relate Gaussian NES to consensus-based optimization via consensus hopping introduced in \cite{riedl2026gradient}. We remark that \cite{wierstra2014natural} already established a connection between NES and the Covariance Matrix Adaptation Evolution Strategy (CMA-ES)~\cite{hansen2001completely}.

\subsection{Consensus hopping and NES interpretation}\label{sec:CBOvsNES}

In order to relate \labelcref{eq:discr_CBO} to \labelcref{eq:GNES}, we consider the so-called the \enquote{infinite drift speed limit} of CBO, which is obtained by choosing $\lambda=\frac{1}{\tau}$, yielding the iteration
\begin{align*}
x^{(n)}_{k+1} 
&= 
\con(X_k) + \sqrt{\tau}\sigma\, \rev{}\mathsf{D}( x_k^{(n)}-\con(X_k))\nc{} \xi_k^{(n)},\qquad \xi_k^{(n)} \sim \mathcal{N}(0,\mathbb{I}_{d\times d}),
\end{align*}
or equivalently
\begin{align}\label{eq:discr_CBO_tau=1}
x_{k+1}^{(n)} &\sim \mathcal{N}\left(\con(X_k), \tau\sigma^2\, \rev{}\mathsf{D}( x_k^{(n)}-\con(X_k))\mathsf{D}( x_k^{(n)}-\con(X_k))^\top\nc{}\right).
\end{align}
The iteration \cref{eq:discr_CBO_tau=1} is similar to consensus hopping studied in \cite{riedl2026gradient} which takes the form 
\begin{align}\label{eq:gradient_hopping}
x_{k+1}^{(n)} &\sim \mathcal{N}\left(\con(X_k), \tilde\sigma^2\mathbb{I}_{d\times d}\right)
\quad
\iff
\quad x_{k+1}^{(n)} = \con(X_k) + \tilde{\sigma} \xi^{(n)}_k,\quad \xi_k^{(n)} \sim \mathcal{N}(0,\mathbb{I}_{d\times d}),
\end{align}
where the natural scaling of the noise variance is $\tilde\sigma=\sqrt{\tau}\sigma$. 
In particular, the only difference between \cref{eq:discr_CBO_tau=1,eq:gradient_hopping} is the scaling of the noise which is \rev{}always isotropic\nc{} and independent of $n$ in \cref{eq:gradient_hopping}. We can rewrite the consensus hopping scheme in terms of the consensus point as follows
\begin{align}
\con(X_{k+1})
&=\notag
\frac{\sum_{n=1}^N \exp(-\alpha \obj(x_{k+1}^{(n)}))\, x_{k+1}^{(n)}}{\sum_{n=1}^N \exp(-\alpha \obj(x_{k+1}^{(n)}))}
=\notag
\frac{\sum_{n=1}^N \exp(-\alpha \obj(\con(X_k)+\tilde\sigma\xi_k^{(n)}))\,(\con(X_k)+\tilde\sigma\xi_k^{(n)})}{\sum_{n=1}^N \exp(-\alpha \obj(\con(X_k)+\tilde\sigma\xi_k^{(n)}))}
\\
&=\tag{CH}
\label{eq:CHop}
\con(X_k)
+
\tilde\sigma 
\frac{\sum_{n=1}^N \exp(-\alpha \obj(\con(X_k)+\tilde\sigma\xi_k^{(n)}))\xi_k^{(n)}}{\sum_{n=1}^N \exp(-\alpha \obj(\con(X_k)+\tilde\sigma\xi_k^{(n)}))}
\rev{}=:
\con(X_k) + \tilde\sigma\cnabla \obj(\con(X_k)).\nc{}
\end{align}
%
%
%
This iteration exhibits strong similarities to the\rev{} \labelcref{eq:GNES} up to the following differences:\nc{}
\begin{itemize}
    \item The objective function $f$ in Gaussian NES, which is a maximization method, is replaced by $\exp(-\alpha f)$ for consensus hopping, which is a minimization method. \cref{prop:comparison} shows how this leads to a similar time-stepping scheme based on an empirical gradient approximation.
    \rev{}For $\alpha>0$,\nc{} minimizers of $\exp(-\alpha f)$ are exactly the maximizers of $f$. 
    \item Gaussian NES approximates the gradient at a point by the \emph{mean value} of weighted samples around that point.
    Consensus hopping approximates it by a \emph{convex combination} of such weighted samples.
    \item Gaussian NES depends on learning rate $\eta$ and a sample variance $\sigma^2$ whereas consensus hopping depends on an inverse temperature parameter $\alpha$ and a sample variance $\tilde\sigma^2$.
\end{itemize}
In the mean-field regime where $N\to \infty$ and for smooth objective functions $f$, we can prove that the two methods are approximations of gradient ascent / descent, respectively,  if the free parameters $\eta,\sigma^2,\alpha$ and $\tilde\sigma^2$ are scaled appropriately.
In \labelcref{eq:GNES}\rev{} it holds that
\begin{align}\label{eq:mfNES}
    \mathbb E\left[
    \frac{1}{N} 
\sum_{n=1}^N \obj(\mu + \sigma \xi^{(n)})\, \xi^{(n)}
    \right]
    =
    \mathbb E\left[
    \obj(\mu + \sigma \xi)\,\xi
    \right]
\end{align}
and so a\rev{} Taylor approximation of $f$ around $f(\mu)$ shows that this expectation approximates a multiple of $\nabla f(\mu)$ for small values of  $\sigma$.\nc{}
For \labelcref{eq:CHop}, the update is a quotient of random variables and we cannot compute its expected values directly. 
Instead, \rev{}letting $\pi$ be a standard normal distribution\nc{}, by the law of large numbers the limit of the expected value as $N\to\infty$ is the expression
\begin{align}\label{eq:mfCHop}
    \frac{\mathbb E\left[\exp\left(-\alpha f(c+\tilde\sigma\xi)\right)\xi\right]}{\mathbb E\left[\exp\left(-\alpha f(c+\tilde\sigma\xi)\right)\right]}
    =
    \frac{\int\exp\left(-\alpha f(c+\tilde\sigma\xi)\right)\xi\de\pi(\xi)}{\int\exp\left(-\alpha f(c+\tilde\sigma\xi)\right)\de\pi(\xi)}
\end{align}
for which we can also use Taylor approximations to relate it to\rev{} $-\nabla f(c)$ for small values of $\tilde\sigma^2$\nc{}.

\begin{proposition}\label{prop:comparison}
If \rev{}$f\in C^{1,1}(\R^d)$ is $L$-smooth for some $L>0$\nc{}, and if $\pi$ is \rev{}the standard normal distribution \nc{} $\mathcal{N}(0,\mathbb{I}_{d\times d})$, then if we scale $\sigma^2:=\frac{\tau}{\eta}$ for \rev{}
$0<\tau\ll 1$, it holds\nc{}
\begin{align*}
\eta\sigma\int f(\mu + \sigma\xi)\xi \de\pi(\xi) 
=
\tau\nabla f(\mu) + O\left(\sqrt{{\tau^3}/{\eta}}\right).
\end{align*}
If $f$, in addition, has a bounded first derivative and if we scale $\tilde\sigma^2:=\frac{\tau}{\alpha}$, then it holds
    \begin{align*} 
        \tilde\sigma
        \frac{\int\exp\left(-\alpha f(c+\tilde\sigma\xi)\right)\xi\de\pi(\xi)}{\int\exp\left(-\alpha f(c+\tilde\sigma\xi)\right)\de\pi(\xi)}
        =
        -\tau\nabla f(c) + O\left(\sqrt{\tau^3\alpha}\rev{}+\sqrt{\frac{\tau^3}{\alpha}}\right),
        \qquad
        \rev{}
        \text{for }0<\tau\ll 1.
    \end{align*}
\end{proposition}
\begin{proof}
    To prove the first statement, we compute
    \begin{align*}
        \int f(\mu + \sigma\xi)\xi \de\pi(\xi) 
        &=
        \int \left(f(\mu)+\sigma\langle\nabla f(\mu),\xi\rangle + O(\sigma^2\abs{\xi}^2)\right)\xi \de\pi(\xi) 
        \\
        &=
        \sigma
        \int 
        \xi\xi^\top 
        \de\pi(\xi)
        \nabla f(\mu)
        +
        O(\sigma^2)
        =
        \sigma\nabla f(\mu) + O(\sigma^2)
    \end{align*}
    which, upon multiplication with $\eta\sigma$ and using $\sigma^2=\frac{\tau}{\eta}$, proves the first claim.
    The second statement is proved in a similar fashion, using the Taylor approximation
    \begin{align*}
        \exp(-\alpha f(c+\tilde\sigma\xi))
        =
        \exp(-\alpha f(c))
        \left[
        1
        -
        \alpha
        \tilde\sigma
        \langle
        \nabla f(c),
        \xi\rangle  
        +
        O\left(\tilde\sigma^2\abs{\xi}^2\left(\alpha^2\abs{\nabla f(c)}^2+\alpha\right)\right)
        \right].
    \end{align*}
    Using this \rev{}and the uniform boundedness of $\abs{\nabla f}$\nc{} we obtain
    \begin{gather*}
        \frac{\int\exp\left(-\alpha f(c+\tilde\sigma\xi)\right)\xi\de\pi(\xi)}{\int\exp\left(-\alpha f(c+\tilde\sigma\xi)\right)\de\pi(\xi)}
        =
        \frac{\int\left(
        1
        -
        \alpha
        \tilde\sigma
        \langle
        \nabla f(c),
        \xi\rangle  
        +
        O\left(\tilde\sigma^2\abs{\xi}^2\left(\alpha^2\abs{\nabla f(c)}^2+\alpha\right)\right)
        \right)\xi\de\pi(\xi)}{\int\left(1
        -
        \alpha
        \tilde\sigma
        \langle
        \nabla f(c),
        \xi\rangle  
        +
        O\left(\tilde\sigma^2\abs{\xi}^2\left(\alpha^2\abs{\nabla f(c)}^2+\alpha\right)\right)\right)\de\pi(\xi)}
        \\
        =
        \frac{-\alpha\tilde\sigma\nabla f(c)+O\left(\tilde\sigma^2\left(\alpha^2\abs{\nabla f(c)}^2+\alpha\right)\right)}{1+O\left(\tilde\sigma^2\left(\alpha^2\abs{\nabla f(c)}^2+\alpha\right)\right)}
        =
        \rev{}
        -\alpha\tilde\sigma\nabla f(c)
        +
        O\left(\tilde\sigma^2\left(\alpha^2+\alpha\right)\right),
    \end{gather*}
\rev{}provided $\tau(\alpha+1)=\tilde\sigma^2(\alpha^2+\alpha)\ll 1$, where we used $\tilde\sigma^2=\frac{\tau}{\alpha}$.\nc{} Multiplying by $\tilde\sigma$ and using the scaling $\tilde\sigma^2=\frac{\tau}{\alpha}$, we simplify the error \rev{} $\tilde\sigma^3(\alpha^2+\alpha)=\sqrt{\tau^3\alpha}+\sqrt{\tfrac{\tau^3}{\alpha}}$ \nc{} to arrive at the conclusion.
\end{proof}
The preceding proposition shows that, when suitably scaled \rev{}, in the mean-field setting, and for sufficiently regular objective functions\nc{}, both \labelcref{eq:GNES,eq:CHop} perform a inexact gradient ascent or \rev{}descent\nc{}, respectively, of form
\[ z_{k+1} = z_k \pm \tau \widehat{\nabla f(z_k)}= z_k \pm \tau \nabla f(z_k) + O\left(\tau^\frac{3}{2}\right), \]
where $\widehat{\nabla f(z_k)}$ are method-dependent estimators of the true gradient based on sampling-based perturbations of either $f$ or its exponentiated form. 
\rev{}
A consequence of \cref{prop:comparison} is that mean-field Gaussian NES and mean-field consensus hopping both subsequentially converge to stationarity under standard assumptions on the function $f$ and for decaying step sizes.
This follows from the next abstract result on gradient descent dynamics with inexact gradients.
\begin{proposition}
    Let $f\in C^{1,1}(\R^d)$ be a $L$-smooth function which is bounded from below.
    If the step sizes $(\tau_k)_{\in\N}$ are chosen such that
    \begin{align}\label{eq:ss_bound}
        \tau_k \leq 
        \max\left\lbrace
        \frac{1}{4L}, 1
        \right\rbrace,
    \end{align}
    then there exists a constant $C=C(L)>0$ such that for given $z_1\in\R^d$ the iteration
    \begin{align*}
        z_{k+1} = z_k - \tau_k \nabla f(z_k) + \operatorname{err}_k, \qquad k\in\N,
    \end{align*}
    where $\abs{\operatorname{err}_k}\lesssim L\tau_k^\frac32$ uniformly in $k\in\N$, satisfies
    \begin{align*}
        \min_{k=1,\dots,K}\abs{\nabla f(z_k)}^2
    \leq 4\frac{f(z_1)-\inf f + C \sum_{k=1}^K\tau_k^2}{\sum_{k=1}^K \tau_k}.
    \end{align*}
\end{proposition}
\begin{remark}
    For convergence to stationarity we need decaying step sizes. 
    For instance, the choice $\tau_k \sim \frac{1}{k^{\frac{1}{2}+\eps}}$ for $k\in\N$ and $\eps\in(0,\tfrac{1}{2})$ yields the rate
    \begin{align*}
        \min_{k=1,\dots,K}\abs{\nabla f(z_k)}^2\lesssim\frac{1}{K^{\frac{1}{2}+\eps}}
    \end{align*}
    that is arbitrarily close to the optimal one of $\frac{1}{\sqrt{K}}$ for stochastic gradient descent \cite{liu2022almost}.
    For constant step sizes $\tau_k=\tau\leq\max\left\{\frac{1}{4L},1\right\}$ one gets the biased bound
    \begin{align*}
        \min_{k=1,\dots,K}\abs{\nabla f(z_k)}^2
    \lesssim
    \frac{1}{K\tau} + \tau.
    \end{align*}
\end{remark}
\begin{proof}
    By the standard descent lemma for $L$-smooth functions, the triangle inequality, Young's inequality, the bound on $\abs{\operatorname{err}_k}$, as well as \eqref{eq:ss_bound} we have
\begin{align*}
    f(z_{k+1}) 
    &\leq 
    f(z_k)
    +
    \langle\nabla f(z_k),z_{k+1}-z_k\rangle
    +
    \frac{L}{2}\abs{z_{k+1}-z_k}^2
    \\
    &\leq 
    f(z_k)    
    -\tau_k\abs{\nabla f(z_k)}^2
    +
    \langle\nabla f(z_k),\operatorname{err}_k\rangle
    +
    L\tau_k^2\abs{\nabla f(z_k)}^2
    +
    L\abs{\operatorname{err}_k}^2
    \\
    &\leq 
    f(z_k)    
    -\tau_k\abs{\nabla f(z_k)}^2
    +
    \frac{\tau_k}{2}
    \abs{\nabla f(z_k)}^2
    +
    \frac{1}{2\tau_k}
    \abs{\operatorname{err}_k}^2
    +
    L\tau_k^2\abs{\nabla f(z_k)}^2
    +
    L\abs{\operatorname{err}_k}^2
    \\
    &\leq 
    f(z_k)
    -
    \tau_k
    \left(
    \frac{1}{2} - L\tau_k
    \right)
    \abs{\nabla f(z_k)}^2
    +
    C \tau_k^{2}
    \leq 
    f(z_k)
    -
    \frac{\tau_k}{4}
    \abs{\nabla f(z_k)}^2
    +
    C \tau_k^{2}
\end{align*}
where $C>0$ is a constant that depends only on $L$.
Reordering and summing over the iterations $k=1,\dots,K$ for $K\in\N$ yields
\begin{align*}
    \min_{k=1,\dots,K}\abs{\nabla f(z_k)}^2
    \leq 4\frac{f(z_1)-\inf f + C \sum_{k=1}^K\tau_k^2}{\sum_{k=1}^K \tau_k}.
\end{align*}
\end{proof}
\begin{corollary}
    Under the conditions of \cref{prop:comparison}, mean-field \labelcref{eq:GNES,eq:CHop} using the gradient estimators \eqref{eq:mfNES}, \eqref{eq:mfCHop}, respectively, satisfy
    \begin{align*}
         \min_{k=1,\dots,K}\abs{\nabla f(z_k)}^2
    \lesssim \frac{1+\sum_{k=1}^K\tau_k^2}{\sum_{k=1}^K \tau_k},        
    \end{align*}
    where the constant in $\lesssim$ depends solely on $f(z_1)$, $L$, and $\inf f$ or $\sup f$, respectively.
    Here the variables $z_k$ can be either $\mu_k$ for Gaussian NES or $\mathbf{c}_k$ for Consensus Hopping.
\end{corollary}
\begin{remark}
The question as to whether this convergence result can be extended to the finite sample version, i.e., algorithms \labelcref{eq:GNES,eq:CHop} for finite sample size $N\in\N$ is challenging and beyond the scope of this work since it requires good bounds for the variances of the gradient estimators, possibly combined with variance reduction techniques.
\end{remark}
The previous results show \nc{}the similarity in spirit of the two methods and might give rise to future theoretical investigations of their convergence properties \rev{}for inexact sample-based estimators of the gradient.\nc{}
While for \labelcref{eq:CHop} a convergence analysis with pretty general conditions is provided in \cite{riedl2026gradient}, for \labelcref{eq:GNES} existing results are relatively scarce. In \cite{schaul2012natural}, convergence of NES is proven for a very small class of objective functions. Furthermore, \cite{ollivier2017information} studies the infinite population limit of a unified class of evolution strategies.
\rev{}
\begin{remark}
    We also present another perspective on how to interpret the two (mean-field) gradient approximations (using Stein's identity $\mathbb E[g(\xi)\xi] = \mathbb E[\nabla g(\xi)]$ for $\xi\sim \mathcal N(0,\mathbb{I}_{d\times d})$): We define $\pi = \mathcal N(0,\mathbb{I}_{d\times d})$, the standard normal distribution, and a reweighted measure $\pi_{x,\alpha,\sigma}$ via \begin{align*}
    \frac{\mathrm{d}\pi_{x,\alpha,\sigma}}{\mathrm{d}\pi}(\xi) = \frac{\exp(-\alpha f(x+\sigma \xi))}{\mathbb E\exp(-\alpha f(x+\sigma \xi))}.
\end{align*}
Then the two gradient approximations can be expressed as
\begin{align*}
\sigma^{-1}\pnabla f(x) \, &\approx\,\sigma^{-1}\mathbb E[f(x + \sigma \xi)\xi]\, =\,  \mathbb E [\nabla f(x+\sigma \xi)],\\
-\alpha^{-1}\sigma^{-1}\cnabla f(x) \,&\approx\, -\alpha^{-1}\sigma^{-1}\frac{\mathbb E \left[\exp(-\alpha f(x+\sigma\xi))\xi\right]}{\mathbb E\left[\exp(-\alpha f(x+\sigma\xi))\right]}= \frac{\mathbb E\left[ \nabla f(x+\sigma\xi) \exp(-\alpha f(x+\sigma\xi))\right]}{\mathbb E\left[\exp(-\alpha f(x+\sigma\xi))\right]}\\
&= \mathbb E\left[\nabla f(x+\sigma\xi) \cdot\frac{\mathrm{d}\pi_{x,\alpha,\sigma}}{\mathrm{d}\pi}(\xi)  \right] = \mathbb E\left[\nabla f(x+\sigma\tilde{\xi})   \right],
\end{align*}
where $\tilde{\xi}\sim\pi_{x,\alpha,\sigma}$. This means that both methods take a statistical average of gradients around $x$, but \labelcref{eq:CHop} does so in a biased way, i.e., weighting lower cost points more. 
From this we also observe, that the gradient approximations can be seen as taking the gradient of smoothed functionals. Namely, from the calculations above, we see
\begin{align*}
\mathbb E [\nabla f(x+\sigma \xi)] = \nabla \Phi(x),\qquad \Phi(x) &:= f \ast \mathcal N(0,\sigma^2\mathbb{I}_{d\times d}),\\
\mathbb E [\nabla f(x+\sigma \tilde{\xi})] = \nabla \tilde{\Phi}(x),\qquad \tilde{\Phi}(x) &:=-\alpha^{-1}\log\left(\exp(-\alpha f) \ast \mathcal N(0,\sigma^2\mathbb{I}_{d\times d})\right).
\end{align*}
The second expression reveals the similarity to yet another particle based method, the so-called zeroth-order proximal point method, see \cite{osher2023hamilton,naldi2026convergence}.\footnote{We thank Cesare Molinari for pointing out this connection.}
\end{remark}
%
%
\begin{figure}
\centering
\includegraphics[width=0.6\linewidth]{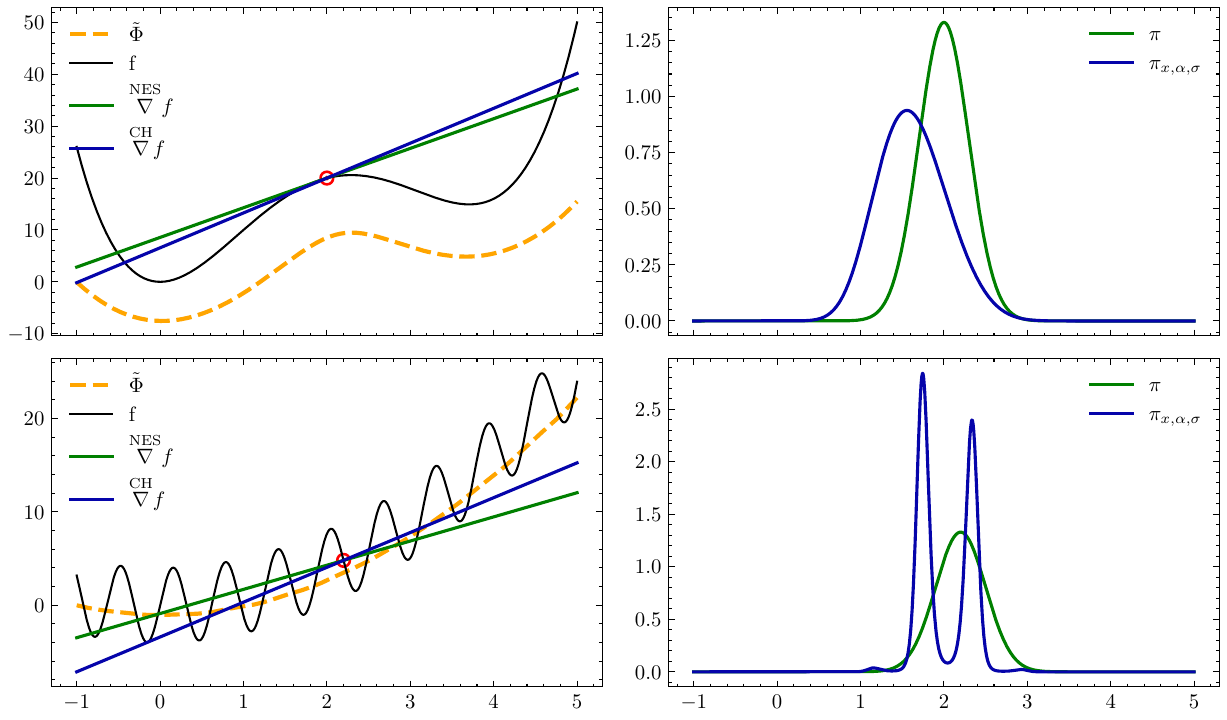}
\caption{
\rev{}Visualization of NES and CH gradient estimates. Note that in the second example shown, we obtain approximations pointing towards the global minimum. The dashed curve is obtained by numerically integrating the CH gradient estimate at grid points in the displayed interval. The effect displayed can be interpreted as a local convexification of the objective function. For CBO, \cite[Thm.~4.1]{riedl2026gradient} and \cite[Thm.~2]{fornasier2022convergence} provide a rigorous analysis of that fact.}
\label{fig:gradest}
\end{figure}
\nc{}
\section{Numerical Examples}\label{sec:num}
In this section, we examine the numerical performance of CBO in different attack scenarios. In addition to the norm-budget $\budget$, in the following we also consider a query budget $Q\in\N$, which determines how often a single optimizer is allowed to evaluate the network $\net$. The metrics we consider consist of success rates within the given query budget, as well as average and median number of queries required to achieve success.

All the following experiments show that NES and CH perform very similar, indicating that the difference in gradient estimation is not significant. In order to compare CH and NES, we slightly deviate from the formulation in \cite{riedl2026gradient} and allow for an additional step size, as described in \cref{sec:CBOvsNES}. The concrete implementation is based on the projected gradient ascent/descent scheme from \cite{ilyas2018black}, which we detail in \cref{sec:details}.
We only consider image classification tasks, where the input space is given as $\X=[0,1]^{C\times H\times W}$. Here, $H$ and $W$ denote the height and width of the images and $C$ the number of channels. Already for the ImageNet dataset, with $H=W=224$ this space is relatively high-dimensional, which is an unfavorable setting for zero-order optimization strategies. A technique to circumvent this problem, is to instead search for a perturbation in a lower-dimensional latent space $\LS$. An application map $T:\LS\times\X\to\X$ applies the perturbation $s$ to the original input $x$, i.e., the adversarial example is constructed as $\bar x=T(s; x)$. The problem in \eqref{eq:adv} transforms to
\begin{align}\tag{LatentAdv}\label{eq:LatentAdv}
\max_{s\in\mathcal{S}\st \phi(T(s;x), x)\leq \budget} \ell(\net(T(s;x)), \cidx).
\end{align}
Apart from the budget constraint, we also need to enforce that $T$ maps into $\X$. By a slight misuse of notation, in the following we employ the map $R:\R^\diamondsuit\to\X$, where the dimension $\diamondsuit$ changes given the concrete application. In most cases $R$ denotes a point-wise clipping operation, i.e., $R(z)_i = \min\{1, \max\{0, z_i\}\}$.
The implementation is based on the \texttt{CBXPy} package~\cite{bailo2024cbx}, and the source code is available in our repository \href{https://github.com/TimRoith/AdversarialCBO}{github.com/TimRoith/AdversarialCBO}.
\paragraph{Enforcing the norm budget constraint} We remark on the different possibilities to enforce the norm constraint into the considered schemes.
\begin{enumerate}[label=BC \roman*), leftmargin=*, labelsep=.5em]
\item\label{it:eca} \emph{Definition of $T$}: An easy way to enforce the norm budget is to define the application in such a way that it only maps to $B^p_\budget(x)$, which can be realized by a projection. I.e., if $T(\cdot; x)$ is the attack-specific application map, we define
%
$
\hat{T}(\cdot; x) = \Pi_{B_\budget^p}\left(T(\cdot, x)\right).
$
%
With this choice, one equivalently can transform \labelcref{eq:LatentAdv} into an unconstrained problem. A disadvantage is that the particles are allowed to explore the whole space $\LS$ and the prior information, about the norm constraint, is only enforced weakly. This creates flat regions of inputs with equal loss, which slightly worsens the performance for particle-based methods, see \cref{tab:cifar}. In our practical implementation, we always use the definition of $\hat{T}$ as a post-processing to ensure that the inputs given to the network fulfill the desired constraint. However, in order to improve the optimization, we employ the method described below.
\item\label{it:ecb} \emph{Projection step}: in the open box setting, one typically employs a projected gradient descent scheme. This means that after each gradient step, the iterate is projected back to the ball $B_\varepsilon^p(x)$. For NES, CH and CBO, this projection step can be similarly employed. Projection schemes for CBO have for example been explored in \cite{bungert2025mirrorcbo,beddrich2024constrained,bae2022constrained}.
\item\label{it:ecc} \emph{Reparametrization}: In the case of an $\ell^\infty$ constraint, one can reparametrize the constraint via $B^\infty_\budget(x) = \{x + \budget \tanh(s): s\in\R^d\}$, which has been done in \cite{carlini2017towards}.
\end{enumerate}
In the following, we mostly employ option \labelcref{it:ecb} for CH, NES and CBO and therefore the algorithms have an additional projection step after each iteration, see \cref{alg:CBO,alg:CHNES}. In the case of $\ell^\infty$ attacks, we noticed only minor differences between option \labelcref{it:ecb} and \labelcref{it:ecc} and therefore do not consider it separately in the following. However, when we compare to other schemes like CMA in the $\ell^\infty$ setting, we use option \labelcref{it:ecc} since it allows us to directly use the implementation provided by the \texttt{Nevergrad}  package \cite{nevergrad}.

An interesting alternative to \labelcref{it:ecb} would be to instead consider the mirror or dual averaging view as in \cite{bungert2025mirrorcbo}. There, one considers a \textit{dual} ensemble, which is allowed to move freely in some space. For the computation of the consensus point the ensemble would then be mapped back to $B_\budget^p(x)$ via a projection. The subtle, but significant difference to variant \labelcref{it:eca} is that here all particles used in \eqref{eq:con} are in $B_\budget^p(x)$, while in \labelcref{it:eca} only the points queried to $\net$ are. We leave this direction for future research.
\paragraph{The objective functions} Here we describe the concrete choice of loss functions $\ell:\R^K\times\{1,\ldots,K\}\to\R$ employed in the following. Note, that we now do not assume that the first input of $\ell$ is already normalized to be in $\Delta^K$, but allow any vector in $\R^K$. When dealing with untargeted attacks, with true label $\cidx$, we employ the margin-based loss \cite{carlini2017towards}, \rev{}$\ell(y, \cidx) := -y_\cidx + \max_{\tilde{\cidx}\neq \cidx} y_{\tilde{\cidx}}$\nc{}.
When maximizing over $y$, this function penalizes the entry corresponding to the true label $\cidx$ and enforces the output at a different (wrong) label to be big. In particular, whenever $\ell(f(\bar{x}), \cidx) < 0$, then $\bar{x}\in B_\budget(x)$ is an adversarial example.
For targeted attacks with target label $\kappa$, we employ a cross-entropy type loss, which in this case reduces to
%
$\ell(y, \cidx) := y_\cidx - \log\left(
\sum_{\tilde{\cidx}=1}^K \exp(y_{\tilde{\cidx}})\right).$
%
For a single attack the objective function is then defined as $\obj(s) := \ell(\net(\tilde{T}(s, x)), \kappa)$.
\rev{}%
\paragraph{Hyperparameters} For CBO we use the algorithm displayed in \cref{alg:CBO}. If not specified otherwise, we always choose the hyperparameters
$\tau = 1.3, \sigma=1$ and $\alpha$ is chosen adaptively with the effective ensemble size scheduler from \cite{carrillo2022consensus}. Due to the high dimensionality of the problems, we choose the anisotropic noise model from \cite{carrillo2021consensus} and refer to \cite{bonandin2026exploiting,fiedler2026consensus} for recent discussions on the efficacy of this choice. For NES and CH we choose $\sigma=0.001$ with an effective initial step size of $\tilde{\eta}=\eta\sigma=0.1$ in \labelcref{eq:GNES}, which is furthermore scheduled with the scheme proposed in \cite{ilyas2018black}. If not specified otherwise, all methods use $N=50$ particles, where CBO additionally uses the batching scheme from \cite{carrillo2021consensus}, where per step only $b=10$ particles re-evaluate the objective function.
\paragraph{Organization}
Our numerical study considers five different attack spaces, consisting of direct attacks in the image space \cref{sec:cifar}, low-resolution attacks \cref{sec:lowres}, $P$-pixel attacks \cref{sec:ppixel}, spectral attacks \cref{sec:dct} and finally square attacks \cref{sec:square}. Furthermore, we conduct experiments on adversarially trained networks in \cref{sec:advtrain}. The finding across these scenarios is, that CBO performs especially well in easier settings, while for harder settings, e.g., targeted attacks in the low-resolution scenario, is not competitive with other algorithms.
\nc{}%
\subsection{Direct attacks on CIFAR-10}\label{sec:cifar} 

We first consider a simple example on the CIFAR-10 dataset \cite{slowik2020evolutionary}. In this case, the image space is comparably low-dimensional with $H=W=32$ and $C=3$. Therefore, a direct attack with $\mathcal{S}=\mathcal{X}$ together with the application map $T(s; x) := R(x + s)$ is feasible. In \cref{tab:cifar} we compare the results of different optimizers in the $\ell^\infty$ setting, with a norm budget of $\budget=0.05$ and a query budget of $Q=10,000$. For NES we used the algorithm proposed in \cite{ilyas2018black}, with the same hyperparameters. For CH, we use exactly the same implementation, only exchanging the gradient estimation, more details are provided in \cref{alg:CHNES}. For CBO we use the hyperparameters detailed in \cref{sec:details}. For comparison, we also consider the performance of the diagonal CMA algorithm, where we employ the implementation provided in the \texttt{Nevergrad} package, \cite{nevergrad}, using the hyperparameters specified there. Here we compare \labelcref{it:eca} and \labelcref{it:ecc} for enforcing the budget constraints, where, as expected, the latter leads to better results. The network to be attacked is a ResNet50 architecture \cite{he2015delving}, which has been trained to a test accuracy of $95.3\%$.

We observe that NES and CH, perform very similar, supporting the previous connection made between the schemes. The difference in gradient estimation, does not have significant impact on the results. \rev{}The CBO\nc{} algorithm requires fewer queries, than the other algorithms.%
\begin{table}
\scriptsize%
\centering%
{\bfseries Targeted attacks on CIFAR-10}\\[.5em]
\begin{tabularx}{\textwidth}{p{3.1cm}*{3}{>{\centering\arraybackslash}X}}
\toprule
\multirow{1}{*}{\bfseries Attack} &
\multicolumn{1}{c}{\bfseries Failure Rate $\downarrow$} &
\multicolumn{1}{c}{\bfseries Average Queries $\downarrow$} &
\multicolumn{1}{c}{\bfseries Median Queries $\downarrow$} \\
%
\midrule
NES   & \bfseries 0.0$ \%$ & 1822.5 \tiny (1822.5) & 1786\\
CH   & \bfseries 0.0$ \%$ & 1833.1 \tiny (1833.1) & 1837\\
CBO  & \bfseries 0.0$ \%$ & \bfseries 494.2 \tiny (494.2) & \bfseries 310\\

\midrule
DiagonalCMA \labelcref{it:eca}  & 0.5$ \%$ & 983.6 \tiny (1019.9) & 729\\
DiagonalCMA \labelcref{it:ecc} & \bfseries 0.0$ \%$ & 840.1 \tiny (840.1) & 645\\
%
%
\bottomrule
\end{tabularx}
\caption{Performance of different optimizers for the attack problem on CIFAR-10. Here and in the following, the query statistics are computed only on the successful runs. The average number of queries computed on all runs is reported in brackets. We print the best success rate and lowest query count on the successful runs in bold, see also \cref{rem:querycount}.}\label{tab:cifar}
\end{table}%

\subsection{Low resolution attacks}\label{sec:lowres}
In this section we consider attacks on the ImageNet dataset, \cite{deng2009imagenet}, where now $H=W=224$. A simple approach to reduce the dimensionality of the image space, is to consider a low-resolution representation of the perturbation $\dx$. This means, we choose a resolution $H_{\text{low}}\leq H, W_{\text{low}}\leq W$ and consider the search space $\LS:=[0,1]^{C\times H_{\text{low}} \times W_{\text{low}}}$. Together with an interpolation map $I:\LS \to \X$ this yields the application
%
$
T(s; x) := R(x + I(s)).
$
%
Choosing $I$ as a nearest neighbor interpolation corresponds to so-called pixel-tiling, which is a well-known technique in the context of closed-box adversarial attacks, see, e.g., \cite{meunier2019yet,ilyas2018prior}. In \cref{fig:lowres-ex} we display an example of an attack produced with this strategy and \cref{tab:lowres} shows a quantitative comparison of different optimizers. As in other works on this topic (see e.g. \cite{ACFH2020square,shukla2021simple}) we attack three different architectures, InceptionV3 \cite{szegedy2016rethinking} (I), ResNet50 \cite{he2015delving} (R) and VGG-16-BN \cite{simonyan2014very} (V), where we use the weights provided by the \texttt{PyTorch} package \cite{paszke2019pytorch}.

Again, we observe that NES and CH perform very similar in this setting, underlining the insight from \cref{sec:CBOvsNES}. Beyond that, \rev{}CBO\nc{} outperforms NES by a large margin across different architectures. In particular, it also works better than CMA~\cite{hansen2006cma}.

%
%
%
%
\begin{figure}
\begin{subfigure}[t]{.3\textwidth}\centering%
\includegraphics[width=\subfmult\textwidth,trim={1.cm 1cm 1cm 1.cm}, clip]{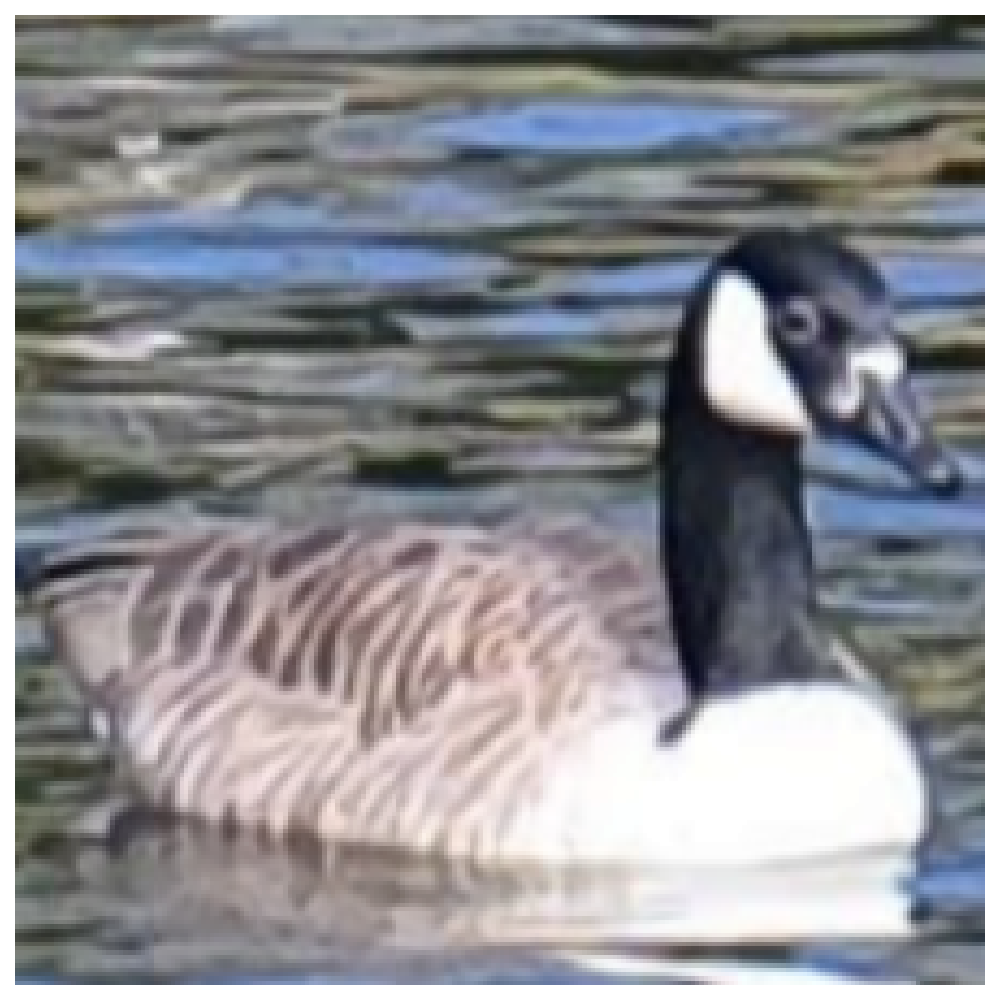}
\caption{Classified as \enquote{Goose}, with probability $99.9\%$ on resolution $H=W=224$.}
\end{subfigure}\hfill%
\begin{subfigure}[t]{.3\textwidth}\centering%
\includegraphics[width=\subfmult\textwidth,trim={1.cm 1cm 1cm 1.cm}, clip]{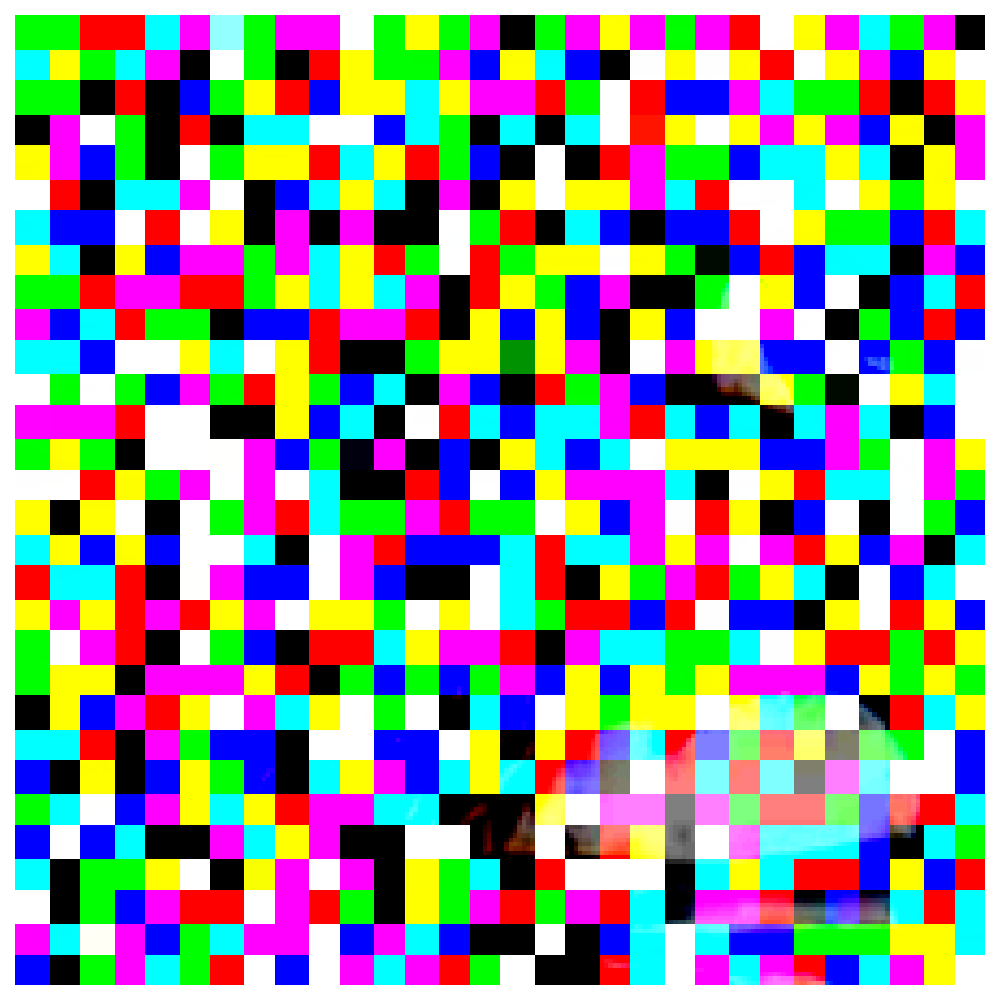}
\caption{Perturbation with $H_{\text{low}}=W_{\text{low}}=50$, $\abs{\dx}_\infty \leq 0.05$\rev{} and scaled pixel values for display.}
\end{subfigure}\hfill%
\begin{subfigure}[t]{.3\textwidth}\centering%
\includegraphics[width=\subfmult\textwidth,trim={1.cm 1cm 1cm 1.cm}, clip]{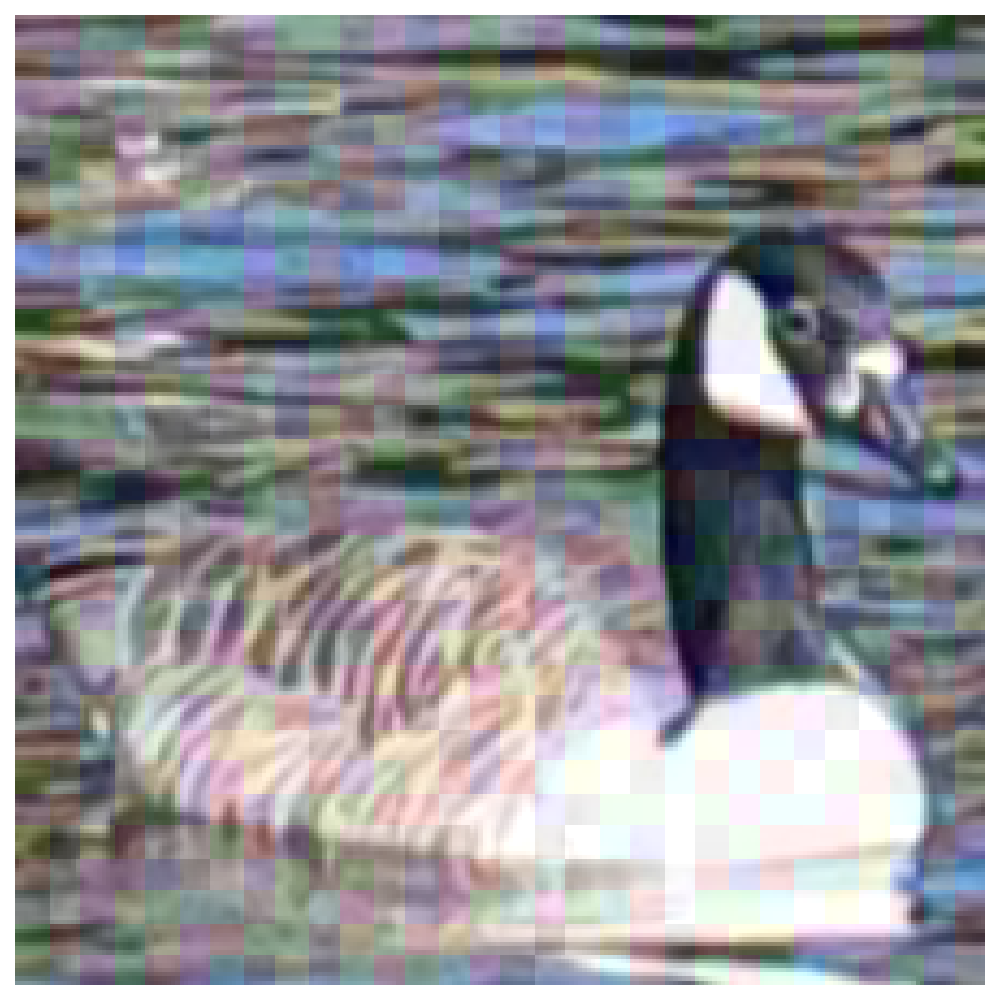}
\caption{Classified as \enquote{Park bench}, with probability $18.9\%$.}
\end{subfigure}\\[-.5em]%
\caption{Example of a low-resolution attack on the InceptionV3 architecture.
}\label{fig:lowres-ex}
\end{figure}%

\begin{table}
\scriptsize%
\centering%
{\bfseries Untargeted low-resolution attacks on ImageNet}\\[.5em]
\begin{tabularx}{\textwidth}{p{3.5cm}*{9}{>{\centering\arraybackslash}X}}
\toprule
\multirow{2}{*}{\bfseries Attack} &
\multicolumn{3}{c}{\bfseries Failure Rate $\downarrow$} &
\multicolumn{3}{c}{\bfseries Average Queries $\downarrow$} &
\multicolumn{3}{c}{\bfseries Median Queries $\downarrow$} \\
\cmidrule(lr){2-4}\cmidrule(lr){5-7}\cmidrule(lr){8-10}
& I & R & V & I & R & V & I & R & V \\
\midrule
NES        & 1.6$\%$ & 0.2$\%$ & 0.1$\%$ &  1388.4 \tiny (1509.5) & 1044.4 \tiny (1053.8) & 728.9 \tiny (732.7) & 1072 & 715 & 154 \\
CH & 1.6$\%$ &  0.2 $\% $ & 0.1$\%$ & 1389.5 \tiny (1517.2) & 1044.3 \tiny (1053.4) & 729.6 \tiny (734.2) &  1123 & 715 & 154 \\
CBO        & 1.5$\%$  & 0.1$\%$ & \bfseries  0.01$\%$ &  \bfseries  416.7 \tiny (560.5)  & \bfseries  250.3 \tiny (259.0) &  \bfseries 
 139.6 \tiny (143.6) & \bfseries  120  & \bfseries  70 & \bfseries  10 \\
\midrule
DFO$_c$ -- DiagonalCMA \cite{meunier2019yet} & 2.8$\%$ & 1.0$\%$ & 0.1$\%$ & 533 & 263 & 174 & 189 & 95 & 55 \\
DFO$_c$ -- CMA \cite{meunier2019yet}         & \bfseries 0.8$\%$ & \bfseries 0.0$\%$ & 0.1$\%$ & 630 & 270 & 219 & 259 & 143 & 107 \\
\bottomrule
\end{tabularx}
\caption{Performance of different optimizers for the untargeted attack problem using the low-resolution attack space in \cref{sec:lowres}. For the results produced with our implementation, we attacked $10,000$ randomly sampled images from the test set of ImageNet.}\label{tab:lowres}
\end{table}
\begin{remark}
Some remarks on the results in \cref{tab:lowres} are in order. The metrics from \cite{meunier2019yet} were directly taken from their paper. The goal of this table is to compare the performance between similar zero-order optimizers, therefore we only used the best performing methods with a strong relation to NES. A more exhaustive comparison is provided in \cref{tab:lowresmore}. Moreover, in \cite{meunier2019yet} the authors explore the idea of solving a discrete combinatorial problem, as in \cite{moon2019parsimonious}, by using the search space 
%
$
\mathcal{S} = \{-\budget, \budget\}^{C\times H_\text{low} \times W_\text{low}},
$
%
which leads to better results. In the current work, we only employ the \enquote{continuous} formulation, which is denoted by the DFO$_c$ in \cite{meunier2019yet} and respectively in $\cref{tab:lowres}$. Furthermore, we note that \cite[Tab. 2]{qiu2021black} obtains different results for adversarial attacks with NES and CMA. It is not possible to directly pin the setup differences, and a comparison between our values and the ones in \cite{qiu2021black} is not meaningful.
\end{remark}
\paragraph{Limitations for targeted attacks}  In \cref{tab:low_res_tar} we show that the performance of CBO is unsatisfactory for targeted attacks on ImageNet. We use exactly the same hyperparameters as before. We see that even NES and CH outperform CBO in this setting. Therefore, we can conclude that the advantages of CBO are mainly visible in easier attack settings. The authors in \cite{meunier2019yet} observed a similar phenomenon, where optimizers that perform well in the untargeted setting are worse for targeted attacks. While we do not have a full explanation for this in the CBO setting, \cref{sec:tarfail} offers some additional remarks and figures on this result.

\begin{table}[h!]
\scriptsize%
\centering%
{\bfseries Targeted low resolution attacks on ImageNet}\\[.5em]
\begin{tabularx}{\textwidth}{p{3.5cm}*{9}{>{\centering\arraybackslash}X}}
\toprule
\multirow{2}{*}{\bfseries Attack} &
\multicolumn{3}{c}{\bfseries Failure Rate $\downarrow$} &
\multicolumn{3}{c}{\bfseries Average Queries $\downarrow$} &
\multicolumn{3}{c}{\bfseries Median Queries $\downarrow$} \\
\cmidrule(lr){2-4}\cmidrule(lr){5-7}\cmidrule(lr){8-10}
& I & R & V & I & R & V & I & R & V \\
\midrule
NES & 1.7$ \%$ & \bfseries 0.8$ \%$ & \bfseries 0.1$ \%$ & 10163.6 \tiny (10452.1) & \bfseries 5931.2 \tiny (5971.1) & 4662.2 \tiny (4668.2) & 7651.0 & 5305.0 & 4234.0\\
CH &  2.1$ \%$ & 1.4$ \%$ & \bfseries 0.1$ \%$ & 9977.9 \tiny (10296.2) & 6017.9 \tiny (6058.7) & 4631.7 \tiny (4632.0) & 7651.0 & 5356.0 & 4234.0\\
CBO         & 3.5$ \%$ & 0.9$ \%$ & 0.2$ \%$ & 14478.2 \tiny (16184.6) & 6320.8 \tiny (6329.6) & \bfseries 3590.7 \tiny (3594.9) & 7800 & \bfseries 3730 & \bfseries 2239\\
\midrule
DFO$_c$ -- DiagonalCMA \cite{meunier2019yet} & 6.0 $\%$ & --- & --- & 6768 & --- & --- & \bfseries 3797 & --- & --- \\
DFO$_c$ -- CMA \cite{meunier2019yet}         & \bfseries 0.0$\%$ & --- & --- & \bfseries 6662 & --- & --- &  4692 & --- & --- \\
\bottomrule
\end{tabularx}
\caption{Performance for low-resolution attacks with a budget of $\budget=0.05$ in the $\ell^\infty$ distance and a query budget of $Q=100,000$. Note that \cite{meunier2019yet} do not report results for \rev{}R and V\nc{}.}
\label{tab:low_res_tar}
\end{table}
\subsection{$P$-pixel attacks}\label{sec:ppixel}

%
%
%
%
\begin{figure}[b!]
\hfill
\begin{subfigure}[t]{.3\textwidth}\centering%
\includegraphics[width=\subfmult\textwidth,trim={1.cm 1cm 1cm 1.cm}, clip]{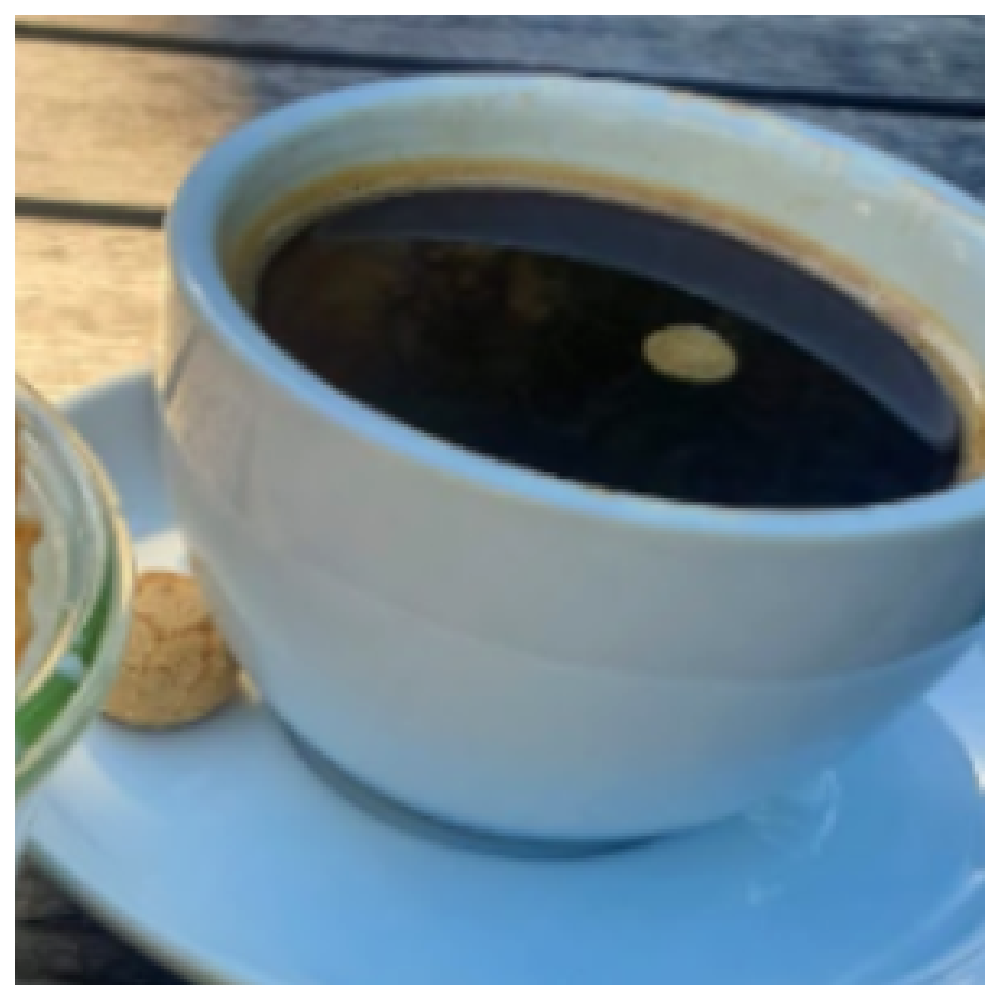}
\caption{Classified as \enquote{Cup}, with probability $27.1\%$.}
\end{subfigure}\hspace{15em}%
\begin{subfigure}[t]{.3\textwidth}\centering%
\begin{tikzpicture}
\hspace{-.77\textwidth}%
\begin{scope}[spy using outlines={circle,magnification=5.5,connect spies,size=3cm}]
    \node[inner sep=0,outer sep=0,anchor=south west] (image) at (0,0) 
      {\includegraphics[width=\subfmult\textwidth]{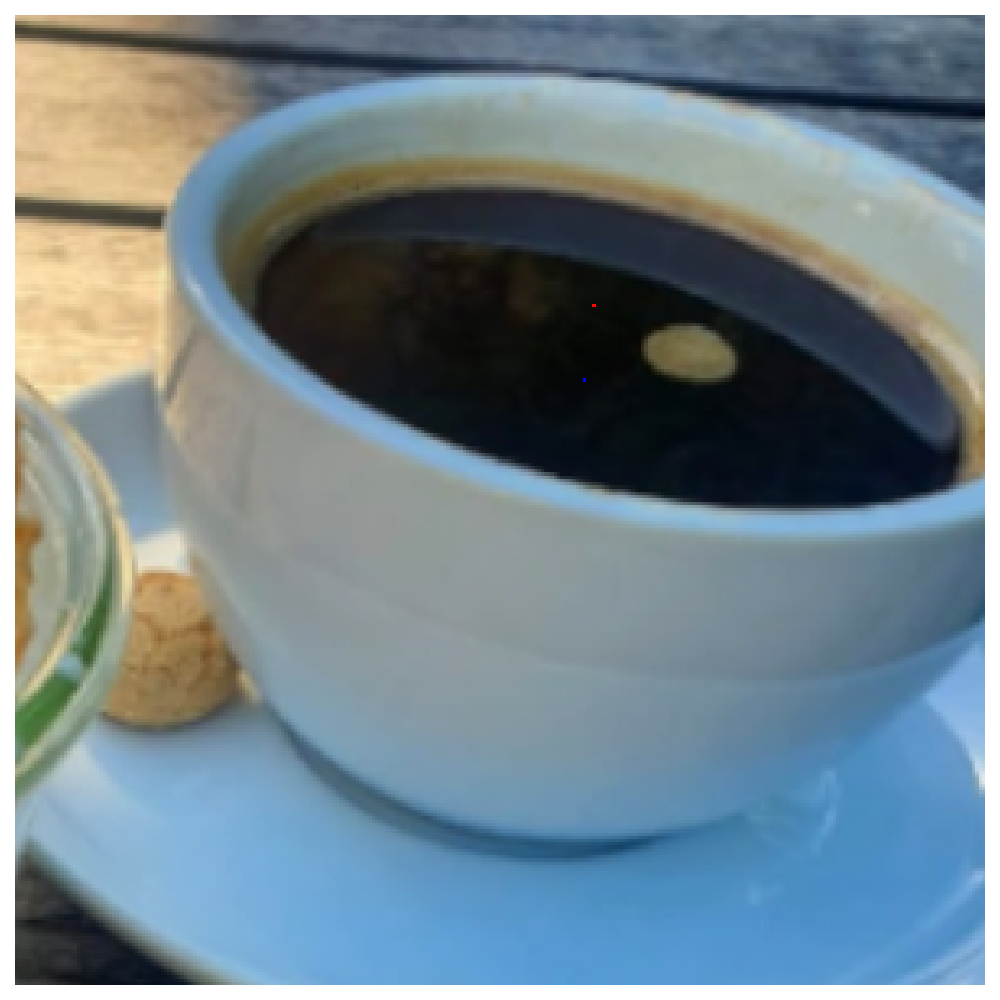}};
    \spy[orange, every spy on node/.append style={ultra thick}, spy connection path={\draw[ultra thick, orange] (tikzspyonnode) -- (tikzspyinnode);}] on (.44\textwidth,.49\textwidth) in node (zoom) at (-.55\textwidth,.5\textwidth);
\end{scope}
\end{tikzpicture}
\caption{Classified as \enquote{Consommé}, with probability $37.3\%$.}
\end{subfigure}\hfill\phantom{.}\\[-.5em]%
\caption{Example of a $2$-pixel attack on the picture depicted on the left. The two pixels added to the image on the right successfully change the label of the classification. Note that, semantically, the attack shown here, does not constitute an extreme misclassification. However, in the performance metrics any run where the label is changed counts as successful.
}\label{fig:pix}
\end{figure}

So-called \emph{one-pixel} attacks were introduced in \cite{su2019one}, see \cref{fig:pix}. The idea here is that an adversary is allowed to only change one-pixel to produce the perturbed image. As a generalization, the authors then further allow an arbitrary number $P$ of pixels to be changed, which we refer to as $P$-pixel attacks. In this case, the attack space can be modeled as
%
$
\left([0,1]^C \times \{1,\ldots, H\} \times \{1,\ldots, W\}\right)^P,
$
%
encoding $C$ variables for the color of each pixel and the coordinate it has on the image. For our experiments, we choose a continuous relaxation of this attack space, namely,
\begin{align*}
\mathcal{S} = \big(\underbrace{[0,1]^C}_{\text{color value}} \times \underbrace{[0,1]^2}_{\text{position}}\big)^P,
\end{align*}
where the position can now freely vary in the domain $[0,1]^2$. We can extract the pixel index via the map $\gamma: [0,1]^2 \to \{1,\ldots, H\} \times \{1,\ldots, W\}$,
%
$
\gamma(\pi):= (\lfloor \pi_1\cdot H \rfloor, \lfloor \pi_2\cdot W\rfloor).
$
%
%
For an element $s\in\LS$, we employ the notation $s=((\zeta^{(1)}, \pi^{(1)}), \ldots, (\zeta^{(P)}, \pi^{(P)}))$, denoting $P$ pairs (color value, positions) and define the mapping
\begin{align*}
\tilde{T}(s; x)_{c, i, j} = 
x_{c,i,j} + \sum_{p=1}^P \boldsymbol{\delta}_{(i,j), \gamma(\pi^{(p)})}\cdot \zeta^{(p)}_c
\end{align*}
where $\boldsymbol{\delta}$ denotes the Kronecker delta, we obtain the application map as $T = R \circ \tilde{T}$.
In \cite{su2019one}, the authors employ so-called \textit{differential evolution} (DE) to optimize the attack problem in \labelcref{eq:LatentAdv} with the above attack space. This is again a particle based zero-order optimizer, and therefore CBO can be similarly employed here. In \cref{tab:pixel} we evaluate the performance of different optimizers in the $1$-pixel attack setting. In order to compare to the results obtained in \cite{su2019one} we attack the AlexNet (A) \cite{krizhevsky2012imagenet,krizhevsky2014one} architecture instead of (R). For (I) and (V) we allow a maximum budget of $10,000$ evaluations, while for (A) we allow 40,000 in order to compare with the results in \cite{su2019one}. We attack $1000$ randomly chosen images from the ImageNet dataset.\rev{} Note that $1$-pixel attacks are considerably harder than the previously considered low-resolution attacks. In \cref{tab:pixel} we thus report success instead of failure rates.\nc{}
\begin{table}[h!]
\scriptsize%
\centering%
{\bfseries Untargeted $1$-pixel attacks on ImageNet}\\[.5em]
\begin{tabularx}{\textwidth}{p{2cm}*{9}{>{\centering\arraybackslash}X}}
\toprule
\multirow{2}{*}{\bfseries Attack} &
\multicolumn{3}{c}{\bfseries Success Rate $\uparrow$} &
\multicolumn{3}{c}{\bfseries Average Queries $\downarrow$} &
\multicolumn{3}{c}{\bfseries Median Queries $\downarrow$} \\
\cmidrule(lr){2-4}\cmidrule(lr){5-7}\cmidrule(lr){8-10}
& I & V & A & I & V & A & I & V & A \\
\midrule
NES & 2.5$ \%$ & 3.1$ \%$ & 1.7$ \%$ & 1543.2 \tiny (9835.4) & 1461.9 \tiny (9781.8) & 1417.0 \tiny (39329.3) & 1225 & 1021 & 52\\

CH & 2.4$ \%$ & 3.8$ \%$ & 1.8$ \%$ & 1375.9 \tiny (9839.9) & 1693.4 \tiny (9730.5) & 2236.5 \tiny (39305.5) & 1021 & 1251 & 1225\\

CBO & \bfseries 7.5$ \%$ & 11.0$ \%$ & 15.0$ \%$ &  1131.2 \tiny (9325.0) &  1257.1 \tiny (9038.3) & 3475.2 \tiny (34512.8) &  260 &  285 & 310\\
\midrule
DE \cite{su2019one} & --- & --- & \bfseries 16.04$\%$ & --- & --- & 25600 & --- & --- & ---\\
DE (Re-run)       & 7.1$ \%$ & \bfseries 11.1$ \%$ & 15.4$ \%$ & 1367.3 \tiny (9378.1) & 1555.1 \tiny (9062.6) & 3871.0 \tiny (34420.2) & 503.0 & 386.0 & 785.0\\
\midrule%
Cauchy (1 + 1)      & 4.3$ \%$ & 4.8$ \%$ & 5.9$ \%$ & \bfseries 31.8 \tiny (9571.4) & \bfseries 29.6 \tiny (9501.8) & \bfseries 561.0 \tiny (37673.1) & \bfseries 28 & \bfseries 24 & \bfseries 18\\
\bottomrule
\end{tabularx}
\caption{Performance of different optimizers for the untargeted attack problem using the $1$-pixel attack space in \cref{sec:lowres}. The re-run of the differential evolution strategy employed in \cite{su2019one}, uses the implementation provided by the \texttt{Nevergrad} package.}\label{tab:pixel}
\end{table}
\begin{remark}
In \cref{tab:pixel} we also report the values obtained in \cite{su2019one} for the ImageNet dataset with the AlexNet network. In order to ensure better comparability, we rerun this experiment in our setup using the DE implementation provided by the \texttt{Nevergrad} library \cite{nevergrad}. Furthermore, for the experiment on (A) we allow a budget of $40,000$ queries as it was done in \cite{su2019one}. %
As remarked in \cite{su2019one}, especially the loss landscape of the $1$-pixel attack exhibits only weak structures, see \cref{fig:onepix}. Therefore, in \cref{tab:pixel} we also test the performance of a $(1 + 1)$-strategy, with Cauchy noise as proposed in \cite{yao1996fast}, again using the \texttt{Nevergrad} library \cite{nevergrad}. While the query statistics on the successful runs are much better, we observe that success rate is far lower than CBO and DE. Therefore, the query counts on the optimizers with higher success rates are biased by \enquote{harder} cases, which is similarly observed in \cite[Tab. 2]{ACFH2020square}.
Furthermore, if one restricts the pixel values to only take $m$ different values, one can reduce the dimensionality of the search space to $m^C\times H \times W$. E.g., if we only allow $\zeta_c^{(p)}\in\{+1,-1\}$, the number of allowed queries should be below $2^3\times 224^2$, otherwise a full search would be better.
\end{remark}

\subsection{Spectral attacks}\label{sec:dct}

\begin{figure}[b!]
\begin{subfigure}[t]{.31\textwidth}\centering%
\includegraphics[width=\subfmult\textwidth,trim={1.cm 1cm 1cm 1.cm}, clip]{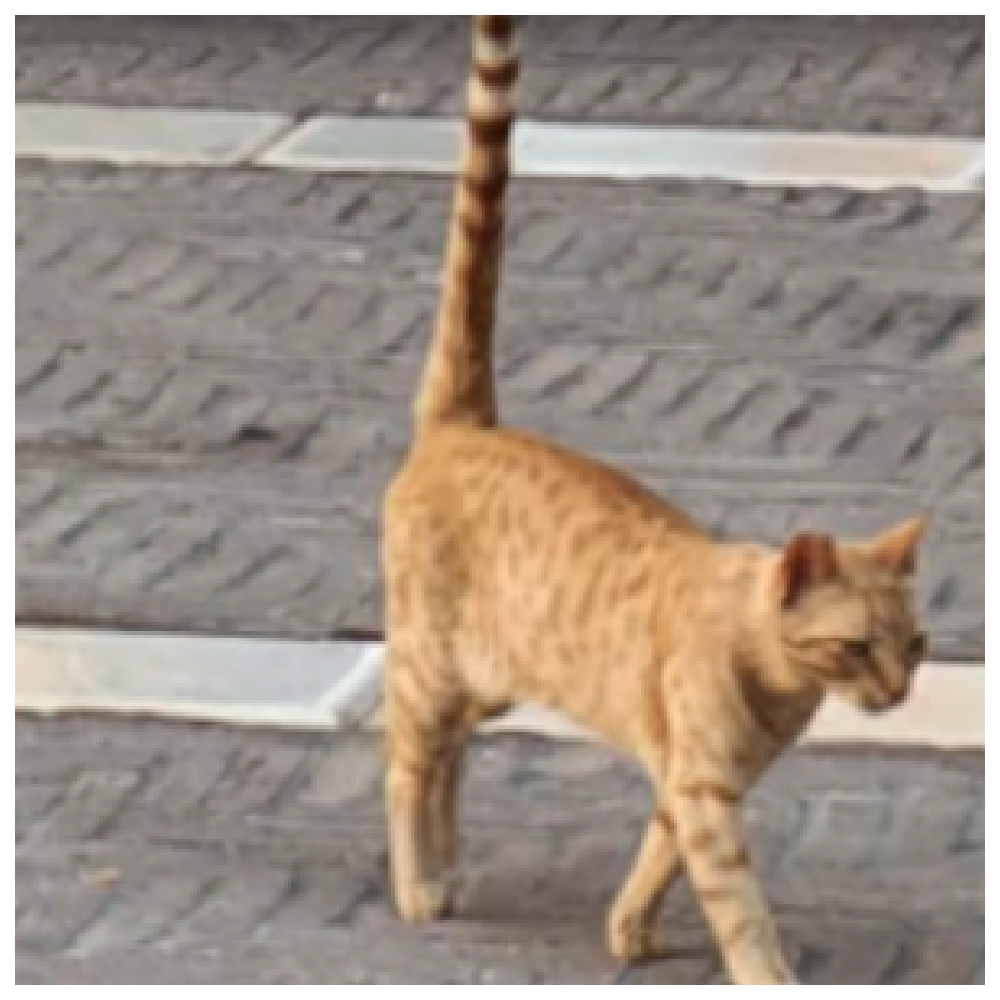}
\caption{Classified as \enquote{Egyptian Cat}, with probability $39.8\%$.}
\end{subfigure}\hfill%
\begin{subfigure}[t]{.31\textwidth}\centering%
\includegraphics[width=\subfmult\textwidth,trim={1.cm 1cm 1cm 1.cm}, clip]{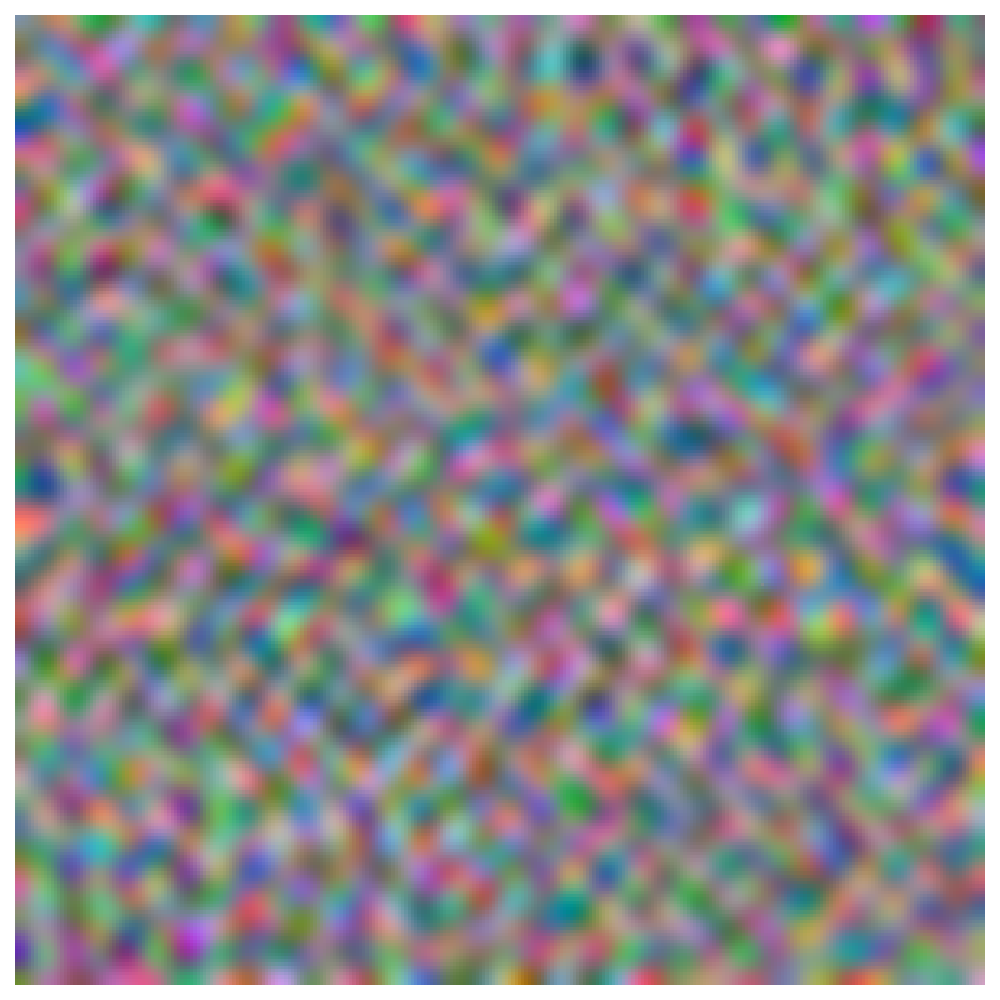}
\caption{Perturbation with $\abs{\pert{}}_2 \leq 3$\rev{} and scaled pixel values for display.}
\end{subfigure}\hfill%
\begin{subfigure}[t]{.31\textwidth}\centering%
\includegraphics[width=\subfmult\textwidth,trim={1.cm 1cm 1cm 1.cm}, clip]{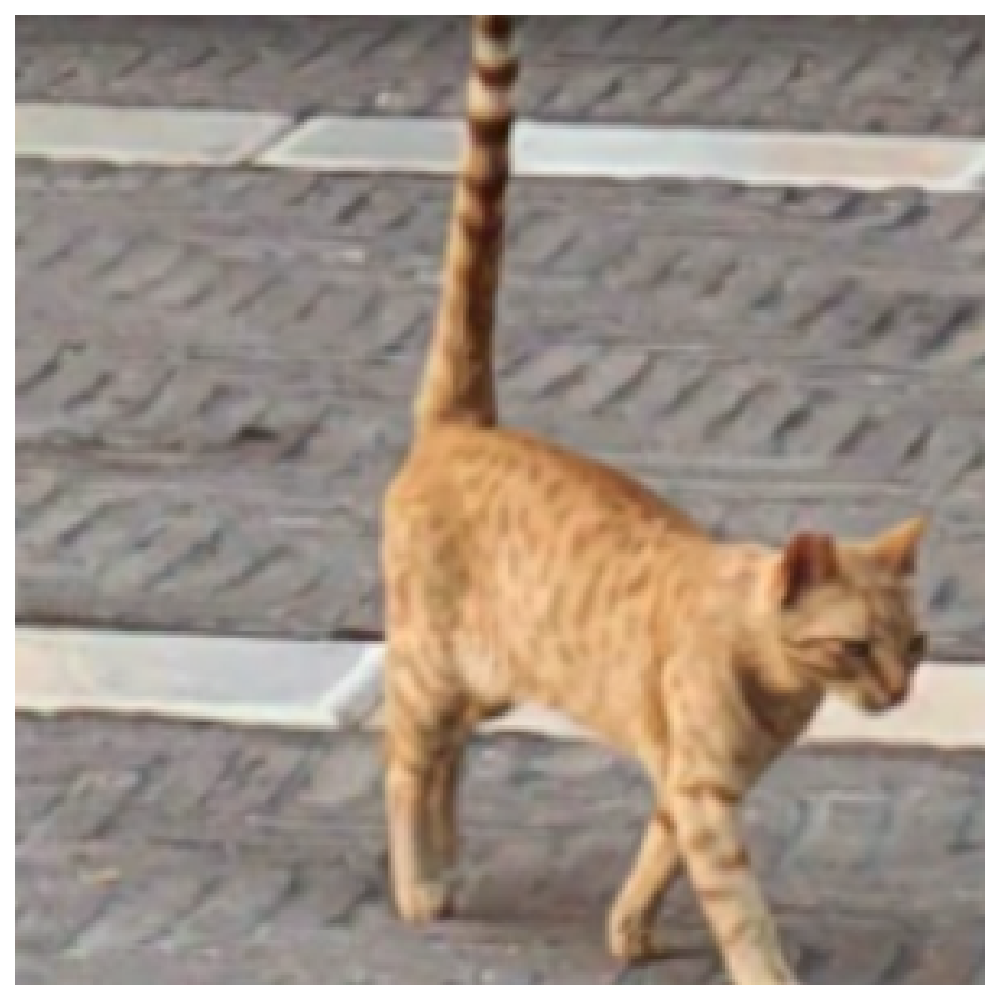}
\caption{Classified as \enquote{Chesapeake Bay Retriever}, with probability $18.3\%$.}
\end{subfigure}\\[-1em]%
\caption{Example of a discrete cosine attack. Here in total, $28$ modes are used for the attack, resulting in $3\cdot 28^2 = 2352$ parameters to optimize.
}\label{fig:dct}
\end{figure}

We now consider an attack space that is based on perturbation in some spectral representation of images, \rev{}see \cref{fig:dct}\nc{}. As in \cite{guo2019simple} we choose the discrete cosine transform (DCT) \cite{Ahmed_dct}, $D:[0,1]^{C\times H\times W}\to \R^{C\times H\times W}$, where we refer to \cite{khayam2003discrete} for a precise definition. In order to reduce the dimensionality of the attack space, we extract the first $m$ modes in each direction. This yields the attack space $\mathcal{S} = \R^{C\times m\times m}$, with the application map
\begin{align*}
T(s; x):= R(D^{-1}(P(s)) + x)
\end{align*}
where $P:\mathcal{S}\to\R^{C\times H\times W}$ places the coefficients in the top left corner and fills the missing entries with zeros. In this setting one typically considers an $\ell^2$-type constraint, and thus by choosing a normalization such that $D$ is an orthogonal operator, we obtain
\begin{align*}
\abs{x - T(s;x)}_2 \leq \abs{D^{-1}(P(s))}_2 = \abs{P(s)}_2 = \abs{s}_2.
\end{align*}
Therefore, in the following we choose the $\ell^2$-projection in the projection step of \cref{alg:CBO,alg:CHNES}.
In \cref{tab:dct} we compare the performance for an untargeted attack, between the CBO-type optimizers and the algorithm employed in \cite{guo2019simple}. For the simpler attack task on CIFAR-10 we observe that CBO requires significantly fewer queries to arrive at a successful attack. However, for the more difficult task on ImageNet, CBO is not able to achieve competitive results. While CH and NES perform better, both their success rates and the required queries are worse than the results obtained in \cite{guo2019simple}. We address the differences between the algorithms in the next paragraph \rev{}and\nc{} again conclude that CBO works well, especially in easier attack scenarios.
\begin{remark}
Apart from the DCT parameterization, other spectral attacks have been explored, for example in \cite{shukla2021simple},  where a Fourier basis was chosen. An interesting direction here would be to consider different spectral representation, based on, e.g., wavelets.
\end{remark}
\begin{table}
\scriptsize%
\centering%
{\bfseries Untargeted DCT attacks on CIFAR-10}\\[.5em]
\begin{tabularx}{\textwidth}{p{3cm}*{3}{>{\centering\arraybackslash}X}}
\toprule
\multirow{2}{*}{\bfseries Attack} &
\multicolumn{1}{c}{\bfseries Success Rate $\uparrow$} &
\multicolumn{1}{c}{\bfseries Average Queries $\downarrow$} &
\multicolumn{1}{c}{\bfseries Average $\ell^2$ $\downarrow$} \\
\cmidrule(lr){2-2}\cmidrule(lr){3-4}\cmidrule(lr){4-4}
& R & R & R \\
\midrule
NES  & \bfseries 100.$ \%$ &   314.0 \tiny (314.0) & 3.0\\
CH   & \bfseries 100.$ \%$ &   314.6 \tiny (314.6) & 2.9\\
CBO  & 99.8$ \%$ &   \bfseries 173.1 \tiny (192.7) & 2.4\\
\midrule%
SimBA-DCT \cite{guo2019simple} &  \bfseries 100.$\%$ & 353.  & \bfseries 2.2 \\
\midrule
\midrule
\end{tabularx}\\[.5em]
{\bfseries Untargeted DCT attacks on ImageNet}\\[.5em]
\begin{tabularx}{\textwidth}{p{3cm}*{9}{>{\centering\arraybackslash}X}}
\toprule
{\bfseries Attack}
& I & R & V & I & R & V & I & R & V \\
\midrule
NES   & 68.7$ \%$ & \bfseries 91.8$ \%$ & \bfseries 96.4$ \%$ & 3417.7 \tiny (5443.1) & 2748.1 \tiny (3319.5) & 2337.2 \tiny (2613.0) & 3.0 & 3.0 & 3.0\\
CH    & 68.1$ \%$ & 91.7$ \%$ & 96.1$ \%$ & 3369.2 \tiny (5451.1) & 2738.6 \tiny (3315.3) & 2352.8 \tiny (2639.6) & 3.0 & 3.0 & 3.0\\
CBO    & 19.6$ \%$ & 33.7$ \%$ & 38.8$ \%$ & \bfseries 1175.1 \tiny (8242.8) & \bfseries 960.5 \tiny (6947.1) & 1019.5 \tiny (6509.5) & \bfseries 2.9 & \bfseries 2.8 & 2.8\\

\midrule
CBO (DCT noise)   & 59.1$ \%$ & 77.0$ \%$ & 84.4$ \%$ & 1223.0 \tiny (4756.6) & 1042.9 \tiny (3100.8) & \bfseries 840.2 \tiny (2267.5) & 3.1 & 2.8 & \bfseries 2.7\\
\midrule
SimBA-DCT \cite{guo2019simple} &  \bfseries 97.8$\%$ & --- & --- & 1283 & --- & --- & 3.1 & --- & --- \\
SimBA-DCT (re-run) &  86.9$ \%$ & ---  & ---  & 2111.7 \tiny (2982.1) & ---  & ---  & 4.6 & ---  & --- \\
\bottomrule
\end{tabularx}
\caption{Performance of different optimizers for the untargeted attack problem using the DCT attack space in \cref{sec:dct}. For NES, CH and CBO we allow a budget of $\budget=3.0$. For the re-run we \rev{}use\nc{} the code provided by \cite{guo2019simple} \rev{}and do not re-run\nc{} on (R) and (V) since these models expect a different input resolution, which is not possible in the provided implementation.}\label{tab:dct}
\end{table}

\begin{remark}\label{rem:querycount}
As in the other tables, in \cref{tab:dct} we also print the lowest success query count per column in bold. However, we want to highlight that in this case where the success rates differ a lot, it is more meaningful to consider the query statistics on all runs.
\end{remark}

\paragraph{Beyond the latent space parametrization} 
The attack proposed in \cite{guo2019simple} exhibits conceptual differences from our particle based approaches here. The update step of the scheme for the perturbation in \cite{guo2019simple} falls into the category of $(1 + \lambda)$-type schemes \cite{beyer2002evolution,rechenberg1978evolutionsstrategien}. At iteration $k$ denote by $s_k$ the current perturbation, then the update reads
\begin{equation}\label{eq:latentadd}
\begin{gathered}
\text{Draw}\quad \tilde{s}^{(1)}, \ldots, \tilde{s}^{(n)} \in \mathcal{S},\quad
s_{k+1} = s_k + \argmin_s \left\{ f(s_k + s) : s \in \{\tilde{s}^{(1)}, \ldots, \tilde{s}^{(n)}\}\right\}.
\end{gathered}
\end{equation}
More precisely, in \cite{guo2019simple} it is assumed that $\mathcal{S}$ is a vector space with orthogonal basis vectors $\{b^{(1)}, \ldots, b^{(d)}\}$. In each step a single basis vector $b^{(k)}$ is drawn randomly and then in the above notation one sets $\tilde{s}^{(1)} = \tau b^{(k)}$, for some parameter $\tau>0$ and further $\tilde{s}^{(2)} = 0$ allowing the previous iterate to be carried over into the next step, if it performs better. A third point $\tilde{s}^{(3)}$ is chosen based on the performance of $\tilde{s}^{(1)}, \tilde{s}^{(2)}$ namely,
\begin{align*}
\tilde{s}^{(3)} = 
\begin{cases}
-\tilde{s}^{(1)}&\text{ if }  \obj(s_k+\tilde{s}^{(1)}) \geq \obj(s_k), \\
0 &\text{ else }.
\end{cases}
\end{align*}
This means, that only when $\tilde{s}^{(1)}$ does not lead to an improvement, additionally $-\tilde{s}^{(1)}$ is added as a candidate. With this approach, the dimension of the search space is equal to the current iteration, i.e., is small at the beginning but grows as the algorithm proceeds. The results obtained with this procedure are significantly better than the ones obtained with the approach considered in \cref{tab:dct}. In the following, we discuss, how CBO could be adapted in this direction in order to increase its performance. As described above, each mutation in the algorithm employed in \cite{guo2019simple} only modifies a single component in the basis $\{b^{(1)}, \ldots, b^{(d)}\}$, whereas in CBO the noise is Gaussian distributed. One possible modification is to change the noise model such that it only varies in a single basis component. In the practical implementation this amounts to replacing the function \textbf{Noise} in \cref{alg:CBO}. Here, we examine this possibility by employing the function specified in \cref{alg:dctnoise}. One should note that this type of noise model is not covered by the standard analysis of CBO methods and therefore looses interpretability. Furthermore, the noise here is not scaled by the distance to the consensus point, similarly to the CH scheme, \rev{} which for CBO is also explored in \cite{chen2020consensus,fornasier2026consensus}.\nc{}

In \cref{tab:dct}, we observe that this modification greatly improves the performance of the \rev{}CBO\nc{} algorithm. However, it still performs worse than the CH and NES schemes, and in particular still much worse than the $(1+\lambda)$ strategy reported in \cite{guo2019simple}. For comparison, we re-run the SimBA-DCT algorithm using the code\footnote{\url{https://github.com/cg563/simple-blackbox-attack}} provided by the authors in \cite{guo2019simple}. The results in \cref{tab:dct} cannot reproduce the ones reported in the paper \cite{guo2019simple}, which hints at a possible setup mismatches. For example, possibly we did not attack the same set of images.

We conclude that for a simple $\ell^2$ attack problem on CIFAR-10, CBO performs very well. In the more difficult ImageNet setting, CBO fails drastically. A modification in the noise function can improve this, but still does not match the performance of other schemes.

\subsection{Square attacks}\label{sec:square}

We now consider a more efficient attack space introduced in \cite{ACFH2020square}, which produces better results in terms of queries and success rate. The attack is based on the observation that convolutional networks are in particular vulnerable to square-shaped perturbations. Therefore, the authors in \cite{ACFH2020square} proposed to obtain closed-box attacks by adding random squares to the target image. The scheme they employ is in spirit very similar to the one described in \eqref{eq:latentadd}. Namely, in each step of the algorithm, a random square is drawn from some distribution. Then, based on whether this square reduces the loss, it is actually added to the perturbation. In order to employ similar concepts for \rev{}CBO\nc{}, we again first choose a fixed parametrization of this attack space. Here, we choose the space
\begin{align*}
\mathcal{S} = \big(\underbrace{[0,1]}_{\text{side length}} \times \underbrace{[0,1] \times [0,1]}_{\text{coordiantes}}\big)^P,
\end{align*}
encoding the $[0,1]$-normalized side length and position coordinates of each square, with a total number of $P$ squares. Additionally, it was noticed in \cite{ACFH2020square}, that initializing a perturbation with random vertical stripes, increases the attack performance significantly. We adapt this, by sampling a fixed image $I\in \mathcal{X}$ containing vertical stripes, see \cref{fig:square}. We define the mapping
$
\tilde{T}(s; x) := 
x +  C\bigg(\sum_{p=1}^P \beta_p(s_p) + I\bigg),
$
where $C(\cdot) = \min\{\budget, \max\{-\budget, \cdot\}\}$ enforces the $\ell^\infty$ budget constraint and $\beta_p$ are the maps that place the squares at the desired position and scale them by a factor, defined as $\beta_p:[0,1]^3\to \mathcal{X}$,
\begin{align*}
\beta_p(r)_{c,h,w} &:= \zeta_{p,c} \cdot
\begin{cases}
1  & \text{ if } (h/H, w/W)\in B^\infty_{r_1}(r_2,r_3),\\
0  & \text{else}.
\end{cases}
\end{align*}
\begin{figure}[t!]
\begin{subfigure}[t]{.31\textwidth}\centering%
\includegraphics[width=\subfmult\textwidth,trim={1.cm 1cm 1cm 1.cm}, clip]{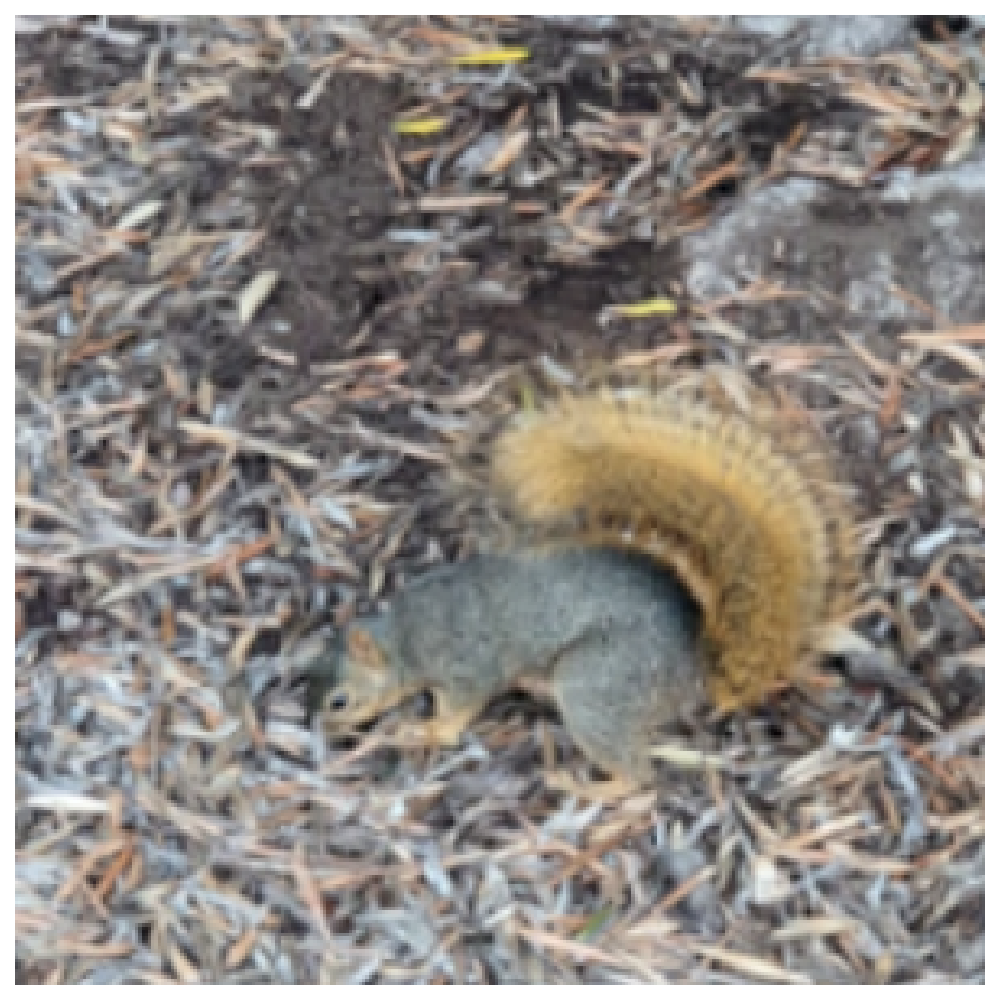}
\caption{Classified as \enquote{Fox squirrel}, with probability $82.9\%$.}
\end{subfigure}\hfill%
\begin{subfigure}[t]{.31\textwidth}\centering%
\includegraphics[width=\subfmult\textwidth,trim={1.cm 1cm 1cm 1.cm}, clip]{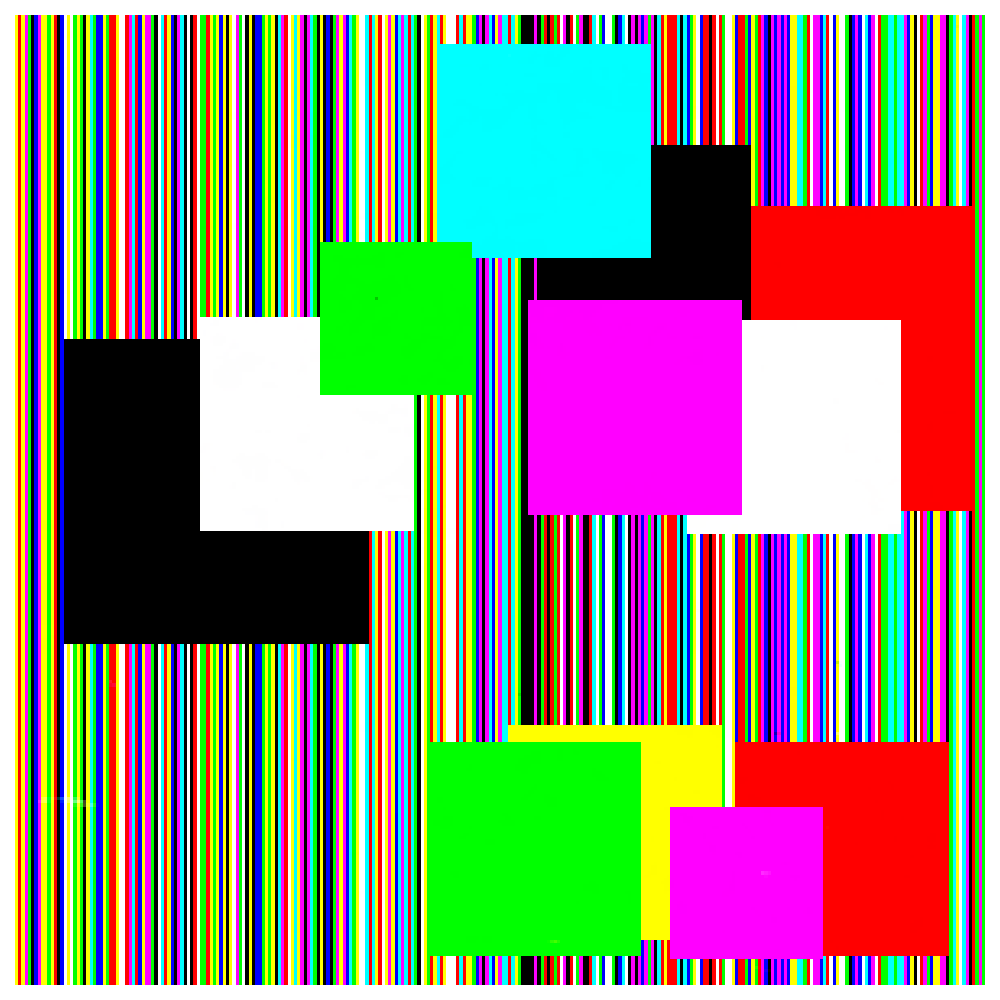}
\caption{Perturbation with $\abs{\pert{}}_\infty \leq 0.05$\rev{} and scaled pixel values.}
\end{subfigure}\hfill%
\begin{subfigure}[t]{.31\textwidth}\centering%
\includegraphics[width=\subfmult\textwidth,trim={1.cm 1cm 1cm 1.cm}, clip]{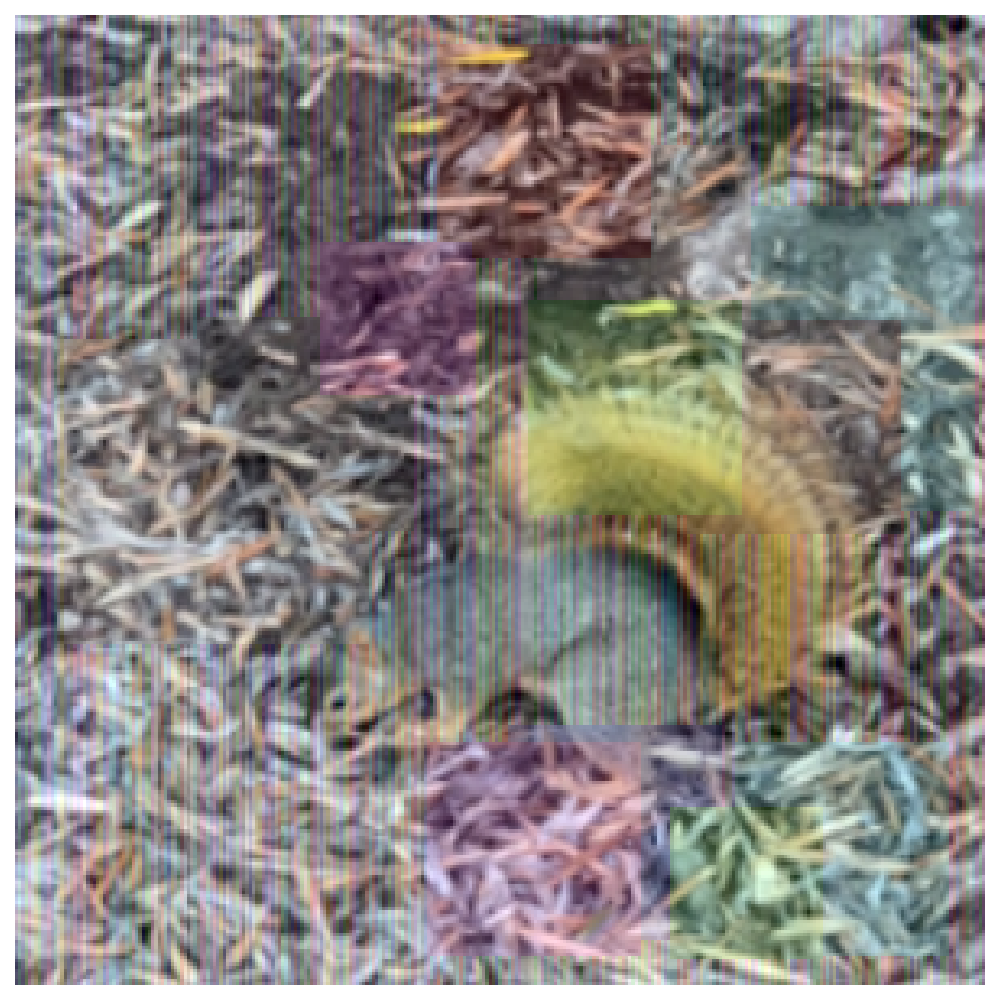}
\caption{Classified as \enquote{Centipede}, with probability $74.8\%$.}
\end{subfigure}\\[-.5em]%
\caption{Example of a square attack as proposed in \cite{ACFH2020square}.}\label{fig:square}
\end{figure}
\begin{table}[t!]%
\centering%
\scriptsize%
{\bfseries Untargeted square attacks on ImageNet}\\[.5em]
\begin{tabularx}{\textwidth}{p{3cm}*{9}{>{\centering\arraybackslash}X}}
\toprule
\multirow{2}{*}{\bfseries Attack} &
\multicolumn{3}{c}{\bfseries Failure Rate $\downarrow$} &
\multicolumn{3}{c}{\bfseries Average Queries $\downarrow$} &
\multicolumn{3}{c}{\bfseries Median Queries $\downarrow$} \\
\cmidrule(lr){2-4}\cmidrule(lr){5-7}\cmidrule(lr){8-10}
& I & R & V & I & R & V & I & R & V \\
\midrule
NES    & 29.6$ \%$ & 17.2$ \%$ & 14.9$ \%$ & 3267.3 \tiny (5256.0) & 2700.2 \tiny (3945.9) & 2035.2 \tiny (3221.5) & 3979 & 3239 & 1\\
CH     & 31.5$ \%$ & 18.1$ \%$ & 17.7$ \%$ & 3277.0 \tiny (5384.5) & 2807.8 \tiny (4087.5) & 2022.0 \tiny (3433.6) & 3979 & 3112 & 1\\
CBO   & 20.9$ \%$ & 9.5$ \%$ & 4.5$ \%$ & 379.8 \tiny (2388.3) & 292.5 \tiny (1213.7) & 182.2 \tiny (623.6) & 20 &  10 & 10\\
\midrule
%
CBO (square noise) & 1.2$\%$ & \bfseries 0.0$\%$ & \bfseries 0.0$\%$ & 226.8 \tiny (344.1) & 96.2 \tiny (96.2) & 39.9 \tiny (39.9) & 30 & \bfseries 10 & 5\\
\midrule
Square Attack \cite{ACFH2020square} &  \bfseries 0.3$\%$ & \bfseries 0.0$\%$ & \bfseries 0.0$\%$ & \bfseries 197. & \bfseries 73. & \bfseries 31. & \bfseries 24 & \bfseries 11 & \bfseries  1 \\
\bottomrule
\end{tabularx}
\caption{Performance of different optimizers for the untargeted attack problem using the square attack space in \cref{sec:square}. For NES, CH and CBO we use $P=50$ squares. For CBO with the square noise model, we use a lower number of particles, namely $N=10$ (compared to $N=50$ otherwise). Furthermore, we choose $\eta=0.01$ (compared to $\eta=0.1$ otherwise), to enforce higher values of $\alpha$, see \cref{sec:appsquare}}\label{tab:square}
\end{table}
Here, $\zeta_{p,c}\in\{-\budget, \budget\}$ encodes the channel value each square contributes to the perturbation, which we sample randomly at initialization and keep fixed. \rev{}The application map is $T=R\circ \tilde{T}$.\nc{}

In \cref{tab:square} we observe that, as before, with the DCT attack the parametrized variant does not achieve the same performance as the original proposed variant in \cite{ACFH2020square}. We again examine the possibility to exchange the noise model as in \cref{alg:dctnoise}. In this case, 
this can be realized by choosing as the attack space $\LS = \mathcal{X}$ and employing a noise model that adds random squares to each particle. The corresponding pseudocode is displayed in \cref{alg:squarenoise}. In \cref{tab:square} we see, that this greatly improves the performance of CBO, however, still not matches the performance of $(1+1)$-ES scheme in \cite{ACFH2020square}. \rev{}For\nc{} the $\ell^\infty$\rev{}-\nc{}norm budget, this still beats the performance of \rev{}CBO\nc{} in the low-resolution setting in \cref{sec:lowres}.

\subsection{Attacks on adversarially trained networks}\label{sec:advtrain}%
In this section we evaluate the performance of CBO, when attacking robust models, obtained with adversarial training \cite{madry2017towards}. We consider a challenge \cite{madrylab_mnist_challenge}\footnote{
The details on the challenge can be found here: \url{https://github.com/MadryLab/mnist_challenge}
} 
on the MNIST dataset \cite{lecun1998gradient}, where an adversarially robust classifier was trained. One considers an $\ell^\infty$ attack space with a norm budget of $\budget=0.3$, however the challenge as specified in \cite{madrylab_mnist_challenge} did not constrain the query budget $Q$. For reference, the authors in \cite{ACFH2020square} use $50$ restarts of the square attack with $20,000$ queries each, which results in a budget of $Q=1,000,000$. Given some test set $\mathcal{T}$, the robust accuracy of a model $\net$ is defined as 
\begin{align*}
R(\net) := 
\frac{1}{\abs{\mathcal{T}}} \sum_{(x,\cidx)\in\mathcal{T}}
\min_{\tilde{x}\in B_\budget(x)} 
\bm{\delta}_{\cidx, \net^{\text{MLE}}(\tilde{x})}.
\end{align*}
The robust accuracy associated with a certain attack is then given as the above value, when replacing $\min_{\tilde{x}\in B_\budget(x)} 
\bm{\delta}_{\cidx, \net^{\text{MLE}}(\tilde{x})}$
by $\bm{\delta}_{\cidx, \net^{\text{MLE}}(x^*)}$, where $x^*$ is the output of the attack. In \cref{tab:mnist} we compare the performance of CBO in this setting to other closed- and open box attacks. Like \cite{ACFH2020square} we also consider multiple runs of the algorithm. For CBO we noticed that letting the algorithm run for longer can often achieve successful attacks. In total, we use $1,000,000$ queries spread out over 5 repeats, with query budgets ranging from $10,000$ to $400,000$. We use the same setup from \cref{sec:cifar}, i.e., the direct attack on the image space.
\begin{table}[h]%
\centering%
\scriptsize%
{\bfseries Untargeted attacks on the MNIST challenge}\\[.5em]
\begin{tabularx}{.6\textwidth}{lc}
\toprule
Attack & Robust Accuracy $\downarrow$  \\
\midrule
\multicolumn{2}{c}{\bfseries Open-box attacks}\\[.5em]
Guided Local Attack                & \bfseries 88.0 $\%$  \\
Open-box challenge winner (2017)   & $89.62 \%$ \\
Projected Gradient Descent (PGD)          & $89.62\%$  \\
\midrule%
\multicolumn{2}{c}{\bfseries Closed-box attacks}\\[.5em]
Square Attack \cite{ACFH2020square} & \bfseries 88.25 $\%$\\
SignHunter \cite{weigand2024adversarialflowsgradientflow} & $91.47\%$ \\
Closed-box challenge winner (2017) \cite{xiao2018generating}  & $92.76 \%$ \\
\midrule%
CBO &  $88.53 \%$\\
\bottomrule%
\end{tabularx}
\caption{Performance of different optimizers for the untargeted attack problem on adversarially trained networks. While CBO does not match the performance of the square attack, it has a higher success rate than the original challenge winners, both for the closed-box and white-box attacks. It outperforms the standard projected gradient descent attack, which tries to solve the same underlying optimization problem, since for CBO we employed the direct attack space from \cref{sec:cifar}. This highlights, that CBO can beat gradient-based methods in this case.
}\label{tab:mnist}
\end{table}%

\section{Conclusion and outlook}

In this work, we established a connection between consensus-based optimization and evolution strategies. We showed that a variant of CBO, namely consensus hopping introduced in \cite{riedl2026gradient} is equivalent to so-called natural evolution strategies, up to second order error terms. We examined the performance of CBO for closed-box adversarial attacks, where evolutionary algorithms are proven to be an efficient strategy. The performance is measured in terms of success rate of the attacks, and number of required queries. We considered untargeted and targeted attacks, with different attack spaces, including direct attacks on the image space, low-resolution, $P$-pixel, spectral and square attacks. With all attacks, we observed that CH and NES perform very similarly, as suggested by our theory. Furthermore, especially in easier attack scenarios, CBO can outperform existing strategies, thus not only offering a mathematically attractive, but also a numerically competitive zero-order optimizer. However, for more difficult scenarios, like targeted attacks, CBO fails to achieve the same benefit. Moreover, for certain attack spaces like spectral and square attacks, it is not directly possible to match the performance of typical $(1+\lambda)$-evolutionary strategies. We explored a modification in the CBO noise term, to mitigate this effect, which lowers the mentioned gap, but does not close it completely.

An evident question for future work, is a more precise characterization of the interplay between ES and CBO. While we offer first results in this direction, the concrete connection between CBO and NES still remains not fully clear, both from a theoretic and practical point of view. Beyond that, the relation to more advanced ES like CMA-ES, is a challenging task we intend to study. 

Furthermore, in this work, we only addressed adversarial attacks on image classification tasks. An interesting direction would be to also apply CBO for attacks on \rev{}language\nc{} models, see, e.g., \cite{zou2023universal}. Due to the discrete nature of input tokens, the problem of closed-box attacks differs from the image classification setting. However, in open-box scenarios, so-called soft-prompt attacks solve an optimization task very similar to \eqref{eq:adv}, see e.g., \cite{schwinn2024soft}. Since the results in \cref{sec:advtrain} indicate that CBO can beat open-box attacks in certain scenarios, one could investigate whether CBO can improve such gradient-based attacks.
\section*{Acknowledgment}
\addcontentsline{toc}{section}{Acknowledgment}

Parts of study was carried out, while TR was affiliated with the DESY (Hamburg, Germany), a member of the Helmholtz Association HGF. This research was
supported in part through the Maxwell computational resources operated
at Deutsches Elektronen-Synchrotron DESY, Hamburg, Germany. Parts of this study was carried out, while TR was visiting the group of Franca Hoffmann at the California institute of technology, supported by the DAAD grant for project 57698811 \enquote{Bayesian Computations for Large-scale (Nonlinear) Inverse Problems in Imaging} and the host FH. TR further wants to thank Samira Kabri for many insightful discussions. 
LB and TR acknowledge funding by the German Ministry of Science and Technology (BMBF) under grant agreement No. 16IS24072 (COMFORT).
LB also acknowledges funding by the Deutsche Forschungsgemeinschaft (DFG, German Research Foundation) – project number 544579844 (GeoMAR).

\printbibliography[heading=bibintoc]

\appendix
\section{Details on the implementation}\label{sec:details}
In this section, we give details on the implementation and hyperparameters for the algorithms used in our experiments.

\subsection{Details on CBO}

For CBO we use the algorithm displayed in \cref{alg:CBO}. If not specified otherwise, we always choose the hyperparameters
$\tau = 1.3, \sigma=1$. The parameter $\alpha$ is chosen adaptively with the effective ensemble size scheduler from \cite{carrillo2022consensus}. The concrete implementation is taken from \cite[Alg. 6]{bungert2025mirrorcbo}. This algorithm introduces an additional hyperparameter $\eta\in[0,\infty)$, where $\eta\ll 1$ results in larger values of $\alpha$. If not specified otherwise, we choose $\eta=0.1$ in our experiments.

Furthermore, for the function \textbf{ComputeConsensus} in \cref{alg:CBO}, we employ the mini-batch scheme proposed in \cite{carrillo2021consensus}. This means that we do not evaluate the consensus point on all $N$ particles but in each step randomly take a subset of $b$ particles on which it is evaluated. The concrete implementation is taken from \cite[Alg. 9]{bungert2025mirrorcbo}. If not specified otherwise, we choose $N=50$ and $b=10$. For completeness, we provide the un-batched variant in \cref{alg:consensus}, which is based on the LogSumExp function.
\begin{algorithm}[H]
\caption{Computes the consensus point with a LogSumExp trick.}\label{alg:consensus}
\begin{algorithmic}[1]
\Function{ComputeConsensus}{$\en{x}, \alpha>0$}
    \State $\en{c} = \exp(- \alpha f(\en{x}) -  \textbf{LogSumExp}(-\alpha f(\en{x})))$
    \State \Return $\sum_{n=1}^N c^{(n)}\, x^{(n)}$
\EndFunction
\end{algorithmic}
\end{algorithm}
Furthermore, by default we use the anisotropic noise model as proposed in \cite{carrillo2021consensus}, with the concrete implementation of $\textbf{Noise}$ taken from \cite[Alg. 4]{bungert2025mirrorcbo}.

In the following, we use the function \textbf{Terminate} to indicate a termination criterion. In our concrete case, we terminate in two cases, namely if the query budget is exhausted, or the current best iterate is adversarial. For CBO, in the untargeted setting, the latter is fulfilled once a particle attains a negative loss value. For the targeted setting, we can use a similar criterion, by modifying the cross-entropy loss $\ell$ as follows,
\begin{align*}
\tilde{\ell}(y,\cidx) = \ell(y, \kappa) - M \cdot \bm{\delta}_{y^{\text{MLE}},\cidx}
\end{align*}
where we choose $M=10$. Note that this does not modify the inputs $y$ which are counted as adversarial, but only their loss value. For NES and CH, we use one extra query in every iteration to check whether the current iterate is adversarial.
\begin{algorithm}
\caption{Consensus-based optimization}\label{alg:CBO}
\textbf{Input:}  Initial ensemble $\en{x}_{0}$, hyperparameters $\sigma\geq 0, \alpha, \tau>0$ 
\begin{algorithmic}[1]
\While{\textbf{Not Terminate()}}
    \State $\en{c}_{k+1} = \textbf{ComputeConsensus}(\en{x}_{k}, \alpha_{k})$
    \State $\tilde{\en{x}}_{k} = \en{x}_{k} - \tau\, (\en{x}_{k} -  \en{c}_{k}) + \sigma \,\textbf{Noise}(\en{x}_{k} -  \en{c}_{k}, \tau)$
    \State $\en{x}_{k+1} = \textbf{Projection}(\tilde{\en{x}}_{k})$
\EndWhile
\end{algorithmic}
\end{algorithm}
\subsection{Details on CH and NES}

For CH and NES we use the algorithm proposed in \cite{ilyas2018black}, which is displayed in \cref{alg:CHNES}.

\begin{algorithm}
\caption{Consensus Hopping and natural evolution strategy}\label{alg:CHNES}
\textbf{Input:}  Initial guess $x_{0}$, hyperparameters $\sigma\geq 0,\eta>0$ 
\begin{algorithmic}[1]
\While{\textbf{Not Terminate()}}
\State $\en{s} = \textbf{Noise}()$
\State $g_k = \textbf{EstimateGradient}(x + \sigma \en{s})$
\State $\hat{g}_k = \textbf{NormalizeGradient}(g_k)$
\State $\tilde{x}_{k+1} = x_k - \eta\, \textbf{GradientStep}(\hat{g})$
\State $x_{k+1} = \textbf{Projection}(\tilde{\en{x}}_{k})$
\EndWhile
\end{algorithmic}
\end{algorithm}

Compared to \cref{alg:CBO} the \textbf{Noise} function employs a fixed a noise scale. Furthermore, as in \cite{ilyas2018black} we employ antithetic sampling, see \cref{alg:NoiseCH}.
\begin{algorithm}[H]%
\caption{The antithetic noise for \cref{alg:CHNES}}\label{alg:NoiseCH}
\begin{algorithmic}[1]%
\Function{AntitheticNoise}{}
    \State $z^{(1)}, \ldots, z^{(\lceil N/2 \rceil)} \sim \mathcal{N}(0, \mathbb{I}_{d\times d})$
    \State $z^{(\lceil N/2 \rceil +i)} = -z^{(i)}$ for $i=1,\ldots, z^{(\lfloor N/2 \rfloor)}$
    \State \Return $(z^{(1)}, \ldots, z^{(N)})$
\EndFunction
\end{algorithmic}
\end{algorithm}%
The function \textbf{EstimateGradient} is the only difference between CH and NES. For CH we choose this function to be \textbf{ComputeConsensus}, with a fixed parameter of $\alpha=10$. For NES, we choose the gradient estimation given in \cref{eq:NESGrad}. 

The algorithm proposed in \cite{ilyas2018black} borrows concepts from open-box attacks. These schemes often employ a gradient normalization to ensure faster convergence. E.g., in \cite{goodfellow2014explaining} for attacks with an $\ell^\infty$-norm budget, it is proposed to normalize the gradient in the $\ell^\infty$ metric. The underlying intuition is that adversarial examples are expected to be in the corners of $B_\budget(x)$. We also refer to \cite{weigand2024adversarialflowsgradientflow} to a theoretical study on the effect of this normalization. As in \cite{ilyas2018black}, we adapt this gradient normalization in \cref{alg:CHNES}. However, in all our experiments we obtained better results by employing $\ell^2$-normalization instead of $\ell^\infty$ . Therefore, throughout all experiments, $\textbf{NormalizeGradient}$ is chosen to be a $\ell^2$ normalization.

The function \textbf{GradientStep} gives the freedom, to additionally introduce momentum into the descent scheme. This is done in \cite{ilyas2018black}, which we adapt.

\rev{}
\subsection{Alternative noise models}

In this section, we specify the pseudo-code for the DCT noise and the square noise employed within the numerical experiments.

\begin{algorithm}
	\footnotesize%
	\caption{DCT-Noise}\label{alg:dctnoise}
	\textbf{Input:}  Number of particles $N$, \rev{}the DCT basis of $\mathcal{X}$, $\{b^{(1)}, \ldots, b^{(d)}\}$\nc{} 
	\begin{algorithmic}[1]
		\State Initialize index sets $I^1, \ldots, I^N$ as permutations of $\{1,\ldots, d\}$; $j \gets 0$
		\Function{DCTNoise}{$d, \tau$}
		\State Initialize $z^{(1)}, \ldots, z^{(N)}$
		\For{$n=1,\ldots, N$}
		\State $z^{(n)} = D^{-1}(b^{(I_j)})$
		\EndFor
		\State $j\gets j + 1$ 
		\State \Return $\sqrt{\tau}\cdot z$
		\EndFunction
	\end{algorithmic}
\end{algorithm}
\begin{remark}
In \cref{alg:dctnoise} we use the same procedure as proposed in \cite{guo2019simple} to create the index sets $I^{(n)}$. Note that the function does in fact not use the drift argument $d$.
\end{remark}
\begin{algorithm}
	\footnotesize%
	\caption{Square-Noise}\label{alg:squarenoise}
	\textbf{Input:}  Number of particles $N$
	\begin{algorithmic}[1]
		\Function{SquareNoise}{$d, \tau$}
		\For{$n=1,\ldots, N$}
		\State Sample a random square $s$ and set $z^{(n)} = s$
		\EndFor
		\State \Return $\sqrt{\tau}\cdot (z^{(1)}, \ldots, z^{(N)})$
		\EndFunction
	\end{algorithmic}
\end{algorithm}
\begin{remark}
	In \cref{alg:squarenoise} we use the same procedure as proposed in \cite{ACFH2020square} to sample random squares in the image space $\X$. In particular, we employ the same rule to adapt the square sizes throughout the iteration. Again, the function does not use the drift argument $d$.
\end{remark}
\nc{}

\section{Further numerical examples}
\rev{}
\subsection{Qualitative behavior of CBO, CH and NES}

In this section we illustrate the difference between CBO, CH and NES on tractable two dimensional settings. In our first example we consider two established benchmarks, namely the Rosenbrock and Rastrigin function in \cref{fig:Rast,fig:Ros}. For CBO we plot the evolution of the consensus point and observe that its evolution exhibits bigger jumps, which is due to the evolution of the underlying ensemble. Compared to that CH and NES (which again perform very similarly) exhibit an evolution that is more reminiscent of a gradient descent scheme. We can observe a similar behavior on the Rastrigin function, where CBO successfully explores the landscape to find the global minimum, whereas CH and NES converge to a suboptimal local minimizer, an effect very typical for gradient descent.

We further want to study a two-dimensional problem that is closer to the adversarial attack class. To do so, we create a synthetic dataset of $8$ classes with two features. The dataset is visualized in \cref{fig:data}. We then train a simple neural network on this classification task with the given data, using the cross-entropy loss and the Adam optimizer \cite{kingma2014adam}. In \cref{fig:data} we visualize the classification performance, by shading areas in $\R^2$ according to the class the model assigns them. On the one hand, we observe that the classification works well and the given data is indeed correctly classified by the network, which suggests that the training process worked well. We can also observe an interesting phenomenon of areas of different classifications forming in parts where data is scarce. This can serve as intuition of how adversarial examples occur. 

Given this setup, we can now perform adversarial attacks for certain given points, both in a targeted and untargeted fashion, which is displayed in \cref{fig:advtwo}. The underlying plot visualizes the objective function, which is the margin-based loss for the untargeted attack and the cross-entropy towards a chosen class (red class in \cref{fig:data}). We only compare CBO and NES, since CH performs similarly to NES. For the untargeted case, we observe that CBO converges already in the first step, meaning that already in the initial ensemble an adversarial example was present and the algorithm stops. Compared to that, NES starts at the clean point and first needs to perform some approximate gradient steps to arrive at the final state.

For the targeted case, we see that there is only a very small region at the boundary of the budget ball that has a small loss value and actually counts as adversarial. CBO here is not able to converge to this region, where a possible explanation is that the ensemble is confined to the budget ball and thus cannot explore the outside area with a potentially favorable loss. More concretely, since the weighted average is always computed between points in the inside of the ball, there is a certain drift toward the inside, which can potentially make it difficult to converge towards points on the boundary. While NES needs more steps than in the untargeted case, it is eventually able to find an adversarial example.

This visualization again underlines our finding of CBO performing well for easier loss landscapes, while struggling for harder cases.

\begin{figure}
\begin{subfigure}{.3\textwidth}
\includegraphics[width=\textwidth]{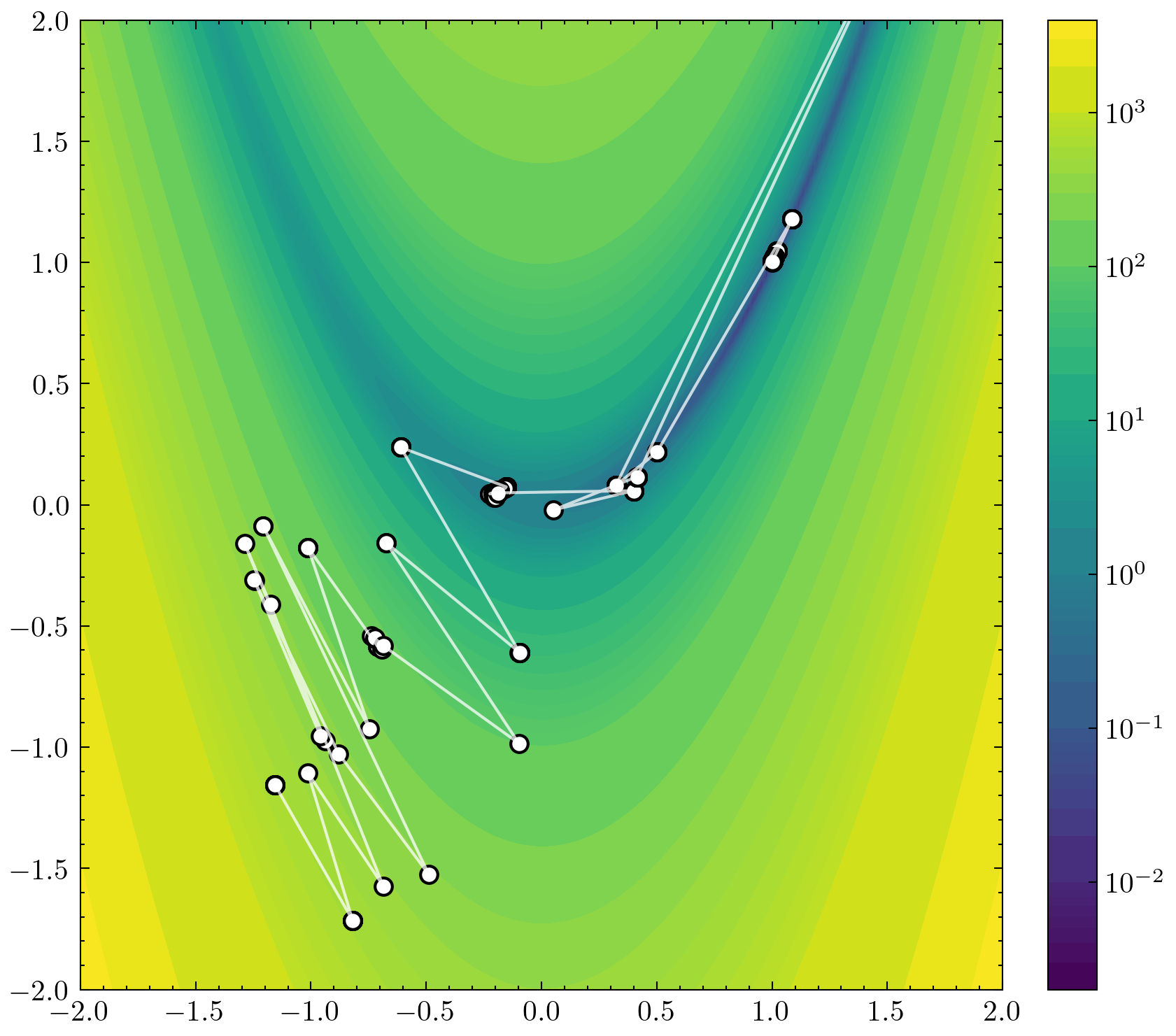}%
\caption{CBO}
\end{subfigure}\hfill%
\begin{subfigure}{.3\textwidth}
\includegraphics[width=\textwidth]{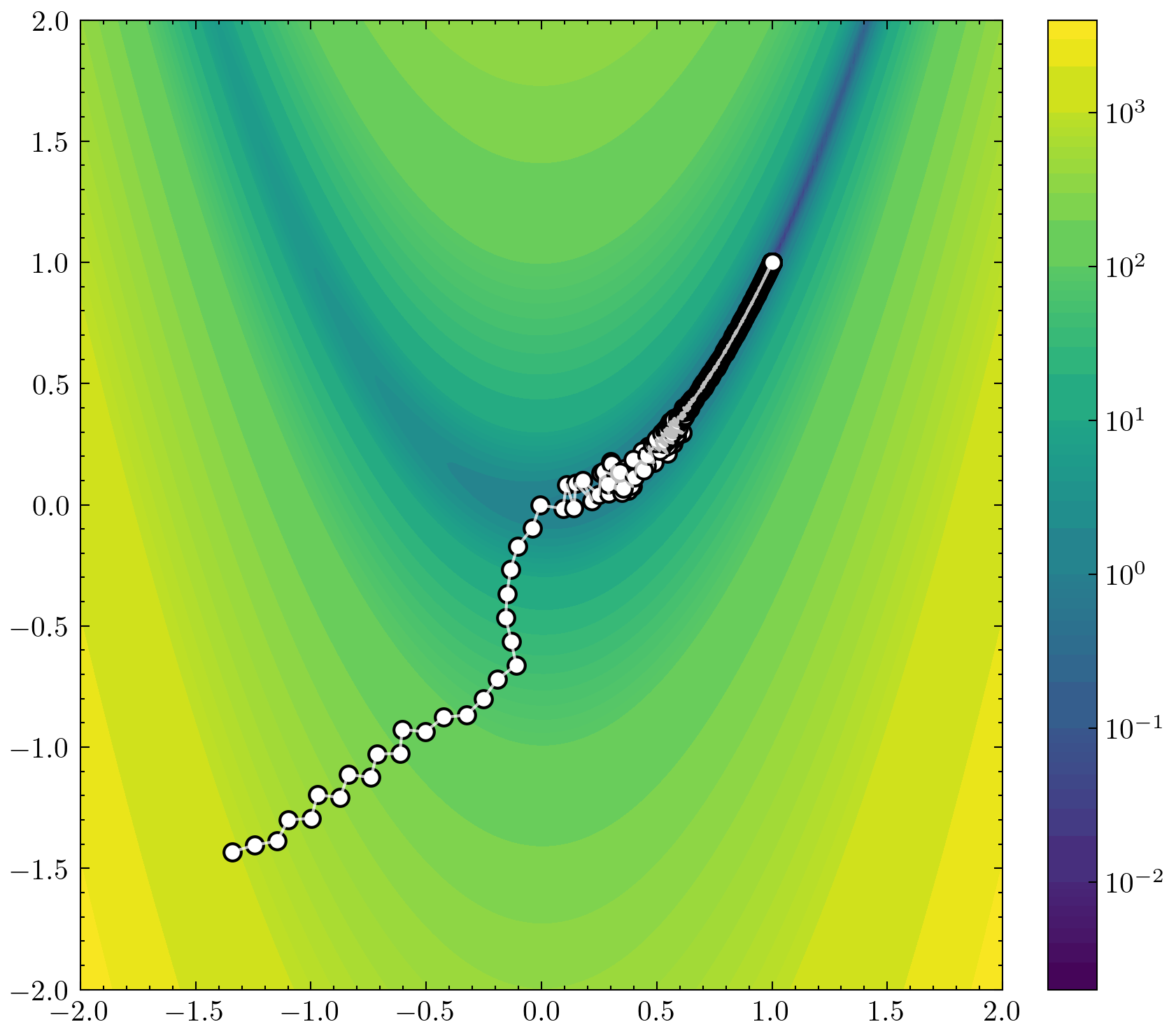}%
\caption{CH}
\end{subfigure}\hfill%
\begin{subfigure}{.3\textwidth}
\includegraphics[width=\textwidth]{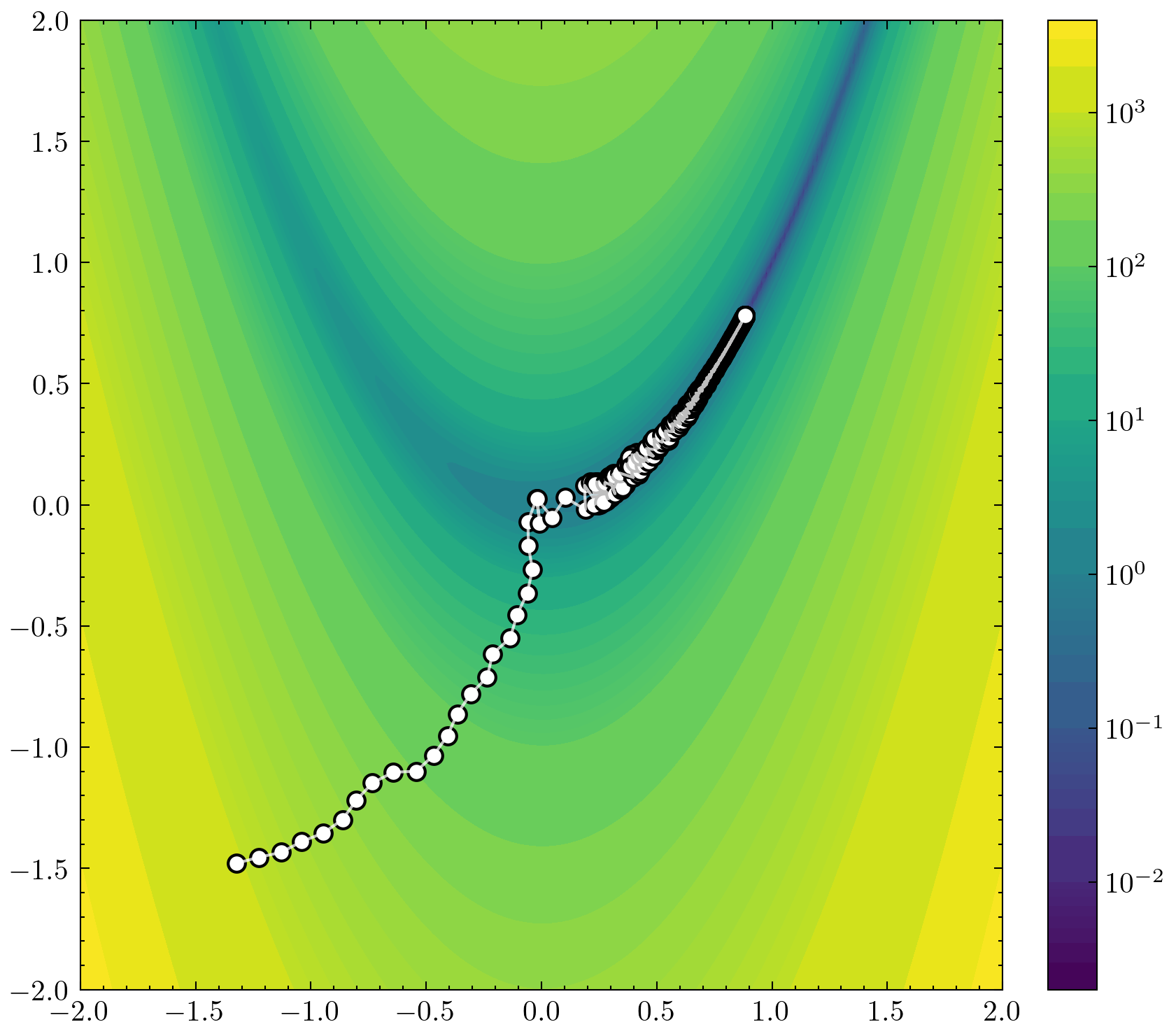}%
\caption{NES}
\end{subfigure}%
\caption{We compare the qualitative behavior of CBO, CH and NES on the Rosenbrock function.}\label{fig:Ros}
\end{figure}
\begin{figure}
\begin{subfigure}{.3\textwidth}
\includegraphics[width=\textwidth]{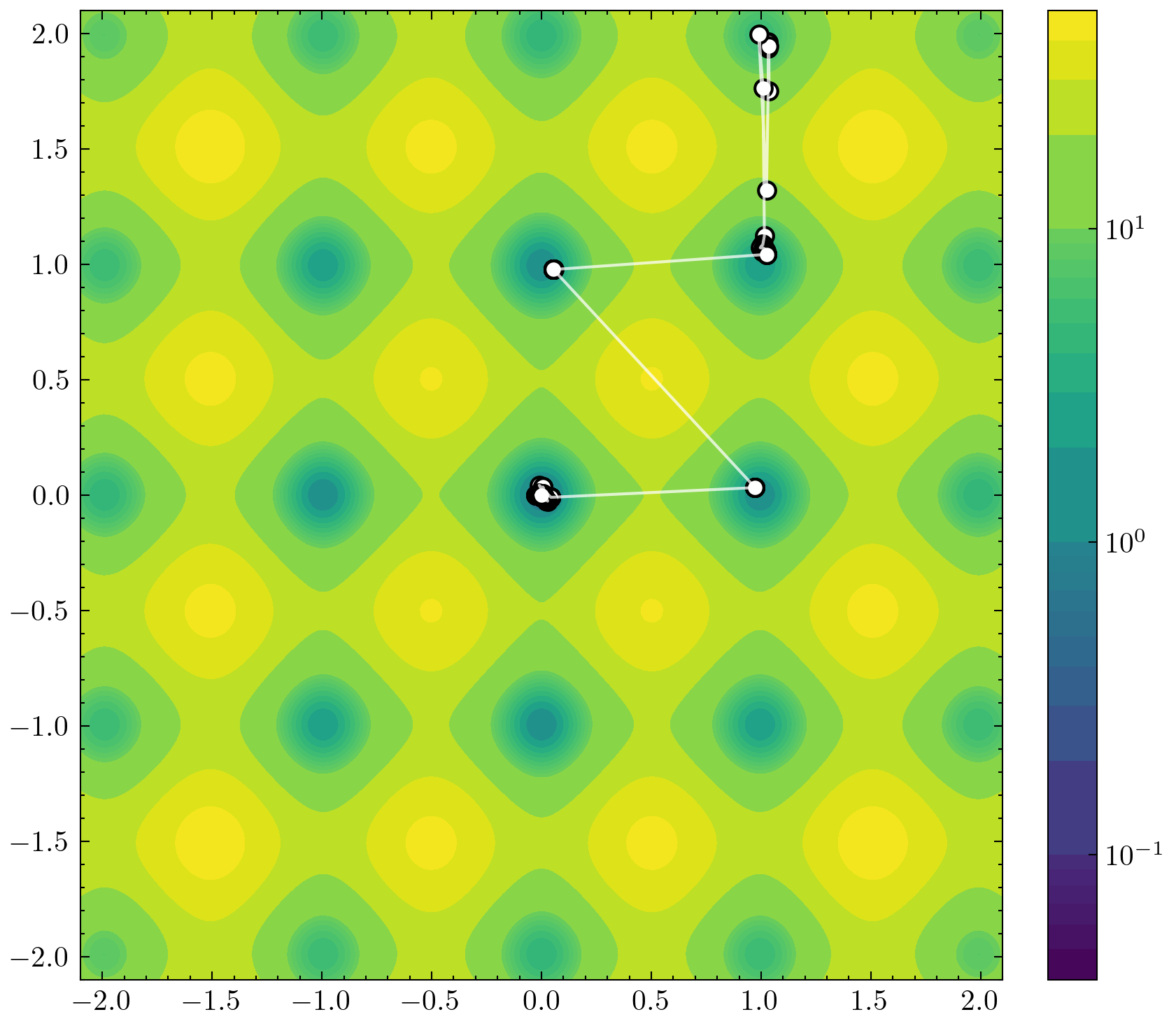}%
\caption{CBO}
\end{subfigure}\hfill%
\begin{subfigure}{.3\textwidth}
\includegraphics[width=\textwidth]{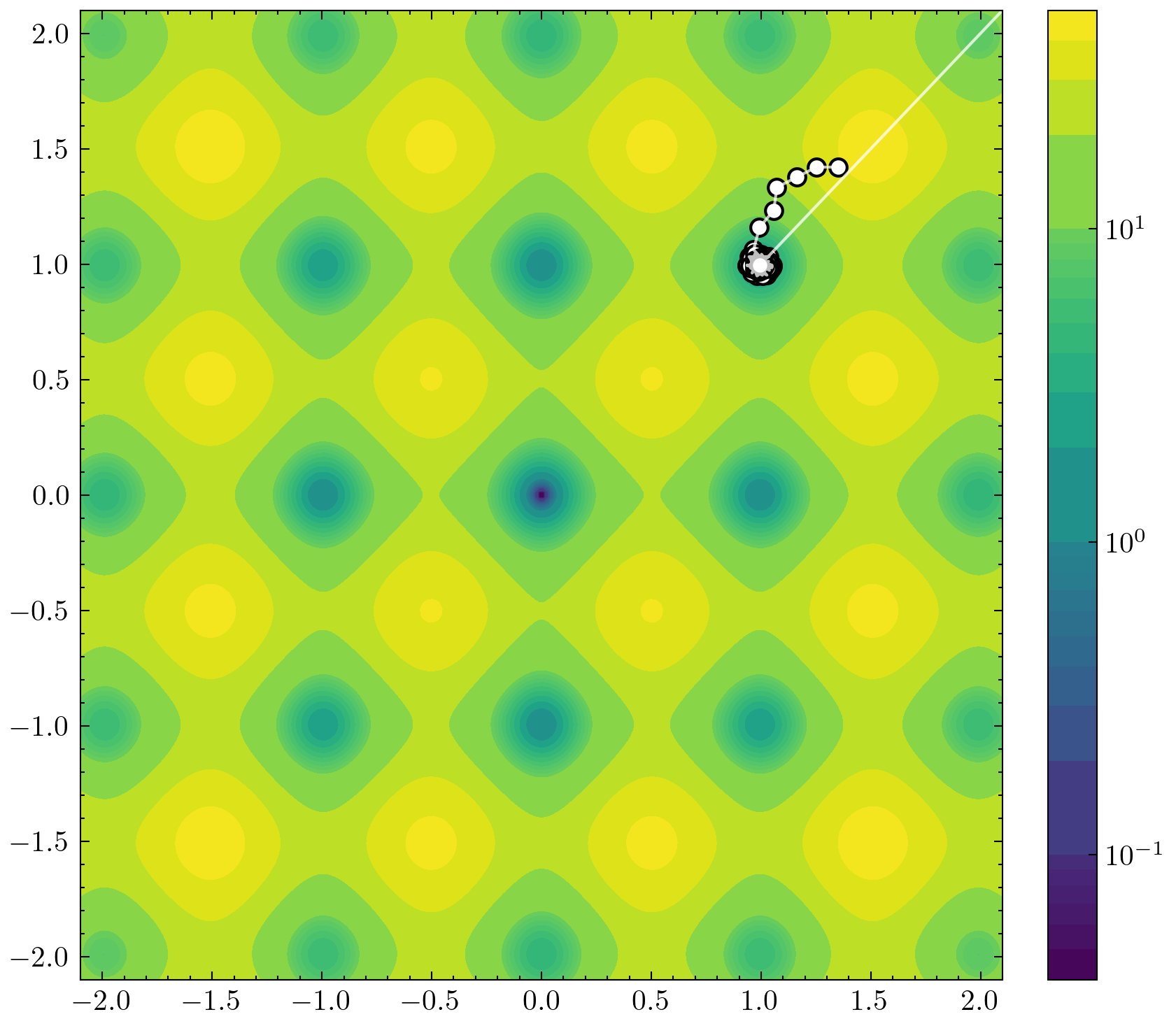}%
\caption{CH}
\end{subfigure}\hfill%
\begin{subfigure}{.3\textwidth}
\includegraphics[width=\textwidth]{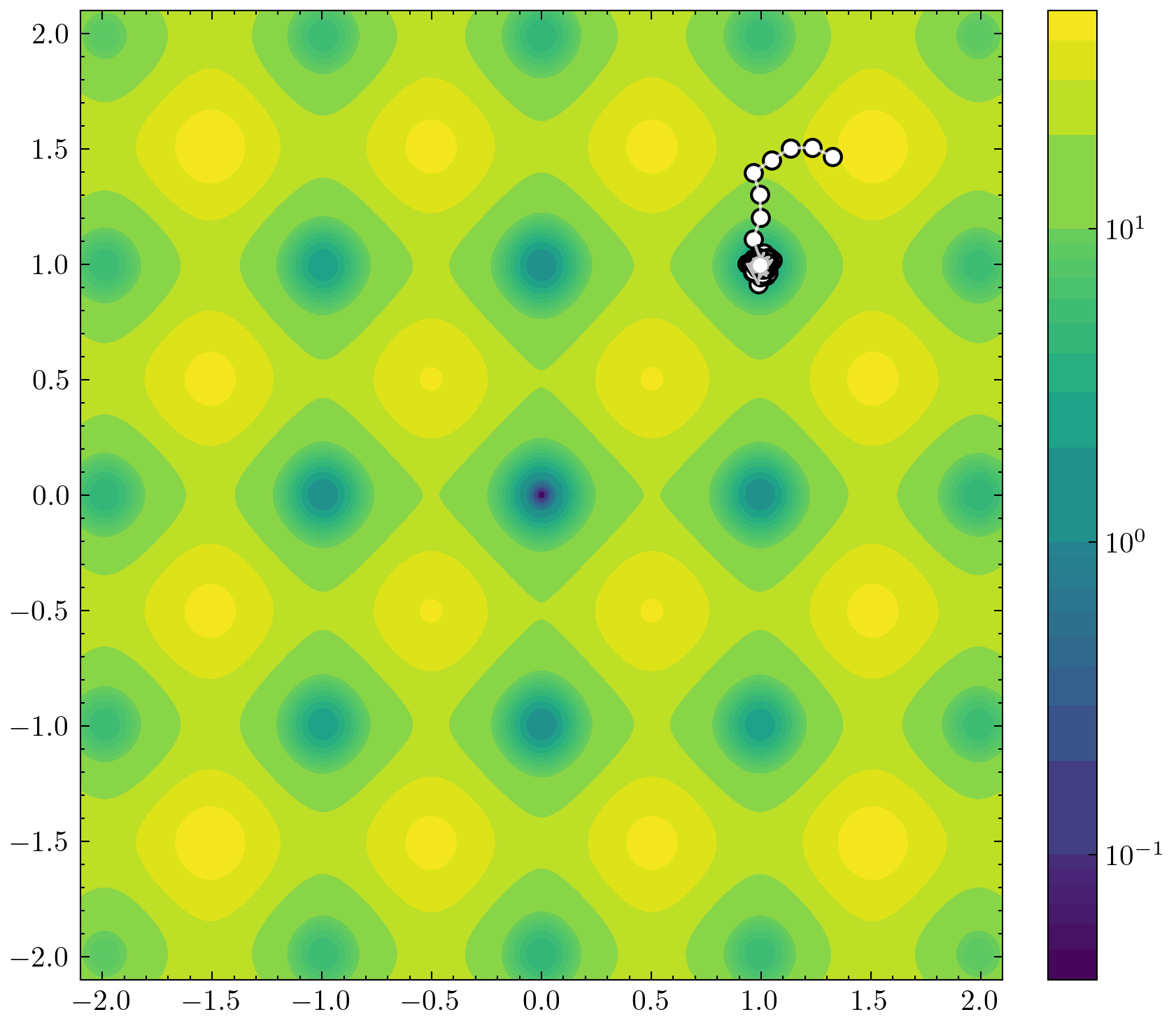}%
\caption{NES}
\end{subfigure}%
\caption{\rev{}We compare the qualitative behavior of CBO, CH and NES on the Rastrigin function.}\label{fig:Rast}
\end{figure}

\begin{figure}%
\centering%
\begin{subfigure}{.45\textwidth}
\includegraphics[width=\textwidth]{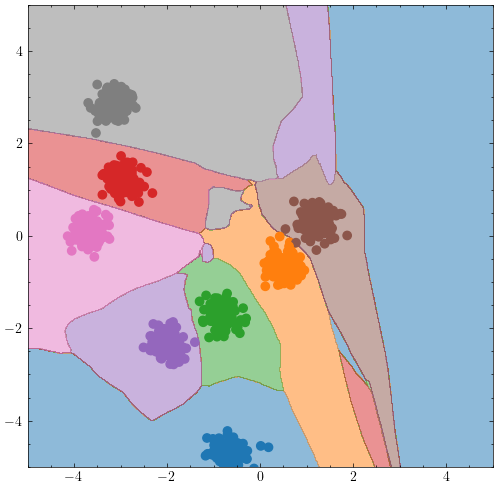}%
\caption{Data used for training and classification performance of trained network.}\label{fig:data}
\end{subfigure}\\%
\begin{subfigure}{.4\textwidth}%
	\includegraphics[width=\textwidth]{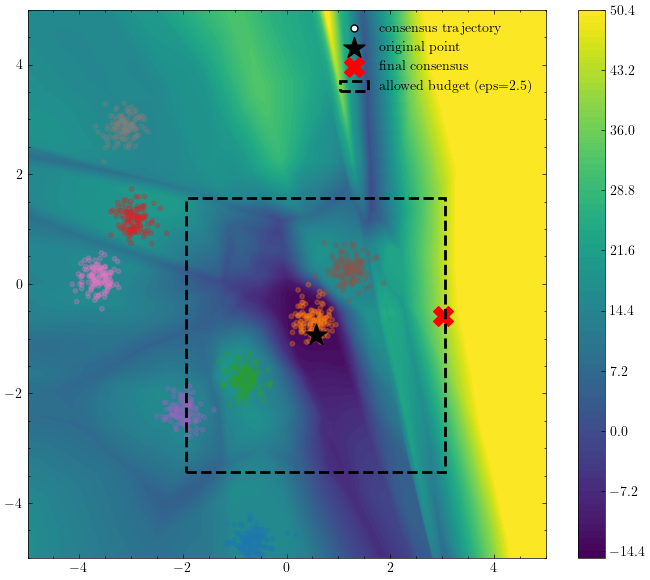}%
	\caption{CBO - untargeted}
\end{subfigure}\hfill%
\begin{subfigure}{.4\textwidth}%
	\includegraphics[width=\textwidth]{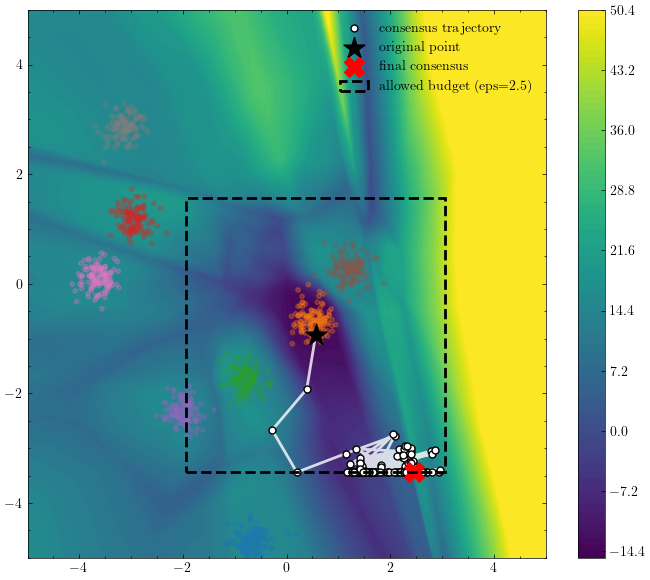}%
	\caption{NES - untargeted}
\end{subfigure}\\%
\begin{subfigure}{.4\textwidth}%
\includegraphics[width=\textwidth]{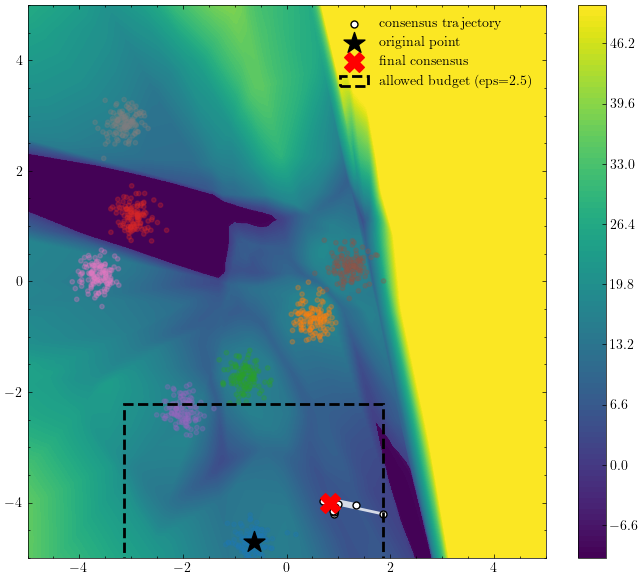}%
\caption{CBO - targeted}
\end{subfigure}\hfill%
\begin{subfigure}{.4\textwidth}%
\includegraphics[width=\textwidth]{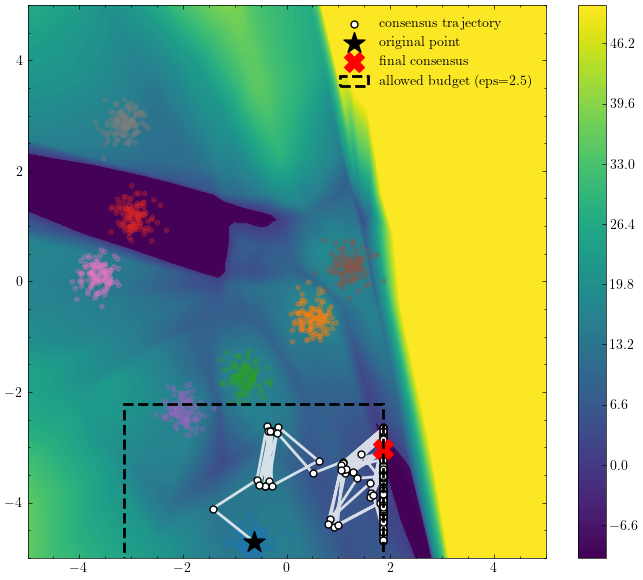}%
\caption{NES - targeted}
\end{subfigure}%
\caption{\rev{}We compare the qualitative behavior of CBO and NES for adversarial attacks on a two-dimensional classification task.}\label{fig:advtwo}
\end{figure}
\nc{}
\subsection{Further experiments on low resolution attacks}
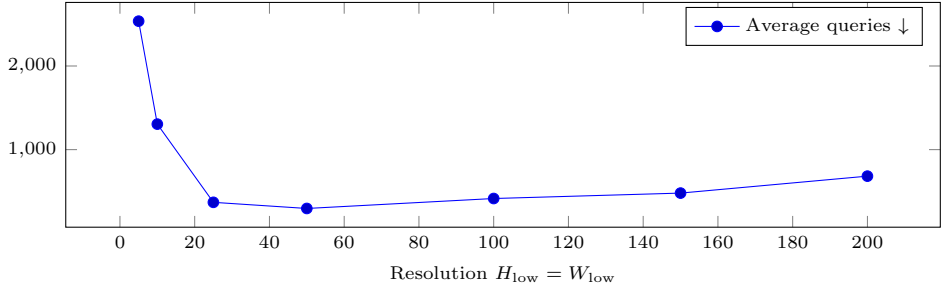
\begin{figure}
\centering%
\begin{tikzpicture}

\begin{axis}[tinyticks, width=.8\textwidth,height=.2\textheight,
xlabel = {Resolution $H_{\text{low}}=W_{\text{low}}$},
scaled y ticks=false,
yticklabel style={/pgf/number format/fixed, /pgf/number format/precision=0}
]
\addplot table [x=n, y=p, col sep=space] {numres.csv};\label{avg}
%
%
\addlegendimage{/pgfplots/refstyle=avg}\addlegendentry{Average queries $\downarrow$}
\end{axis}
\end{tikzpicture}
\caption{Average queries of CBO for unpargeted attacks on ImageNet, when varying the low resolution of the attack space in \cref{sec:lowres}.}
\label{fig:diffres}
\end{figure}
We first examine the attack performance of CBO in the low resolution scenario, when varying the dimension of the latent space $\LS = [0,1]^{C\times H_{\text{low}} \times W_{\text{low}}}$. We plot the results for different $H_{\text{low}}=W_{\text{low}}$ in \cref{fig:diffres}. For each resolution, we attack the same randomly sampled $50$ images of the ImageNet dataset, employing the (V) architecture. We choose a maximum query budget of $Q=10,000$ and report the average number of queries on all runs, i.e., each unsuccessful run contributes with $Q$ to the average.

Furthermore, in \cref{tab:lowresmore}, we extend \cref{tab:lowres} by adding more comparisons to other query-based closed-box attacks in an $\ell^\infty$ scenario with norm budget $\budget=0.05$ and query budget $Q=10,000$. Regarding the results shown therein, we remark the following:
\begin{itemize}
\item All results from \cite{meunier2019yet} were taken from the paper. We only use the results with similar failure rates as CBO. The underscore $d$ denotes the discrete attack mode described therein. This attack space was not used for CBO, therefore, the results are not directly comparable. Furthermore, one observes that the Cauchy(1+1)-ES strategy can outperform CBO in certain scenarios, however, has a higher failure rate for (I) and (R). Interestingly, this optimizer does not perform well in the targeted setting, similarly to CBO. Here, we leave a more detailed comparison between (1 + 1) schemes \cite{back2018evolutionary}  with CBO for future study. However, in particular, using Cauchy noise as proposed in \cite{yao1996fast} might be an interesting extension for consensus schemes.
\item For ZO-NGD we take the results from \cite{zhao2020towards}. The results for ZOO attack \cite{chen2017zoo} were also taken from \cite{zhao2020towards} which re-run the original implementation in our desired setting.
\item The results for Bandits \cite{ilyas2018prior} were taken from the original paper.
\item The results for the parsimonious attack \cite{moon2019parsimonious} were taken from the original paper.
\item The results for SignHunter \cite{al2019there} were taken from the original paper.
\item The results for the square attack \cite{ACFH2020square} were taken from the original paper. So far, we always made an additional distinction between different attack spaces. Here, we now mainly use the $\ell^\infty$ norm budget to compare between different optimizers. As mentioned before, the square attack space is significantly more effective than the attack spaces of the other optimizers displayed in this table and in a way not comparable to those. Nevertheless, we include it in the comparison here.
\end{itemize}
\begin{table}
\footnotesize%
\centering%
{\bfseries Untargeted $\ell^\infty$ attacks on ImageNet}\\[.5em]
\begin{tabularx}{\textwidth}{p{4.cm}*{9}{>{\centering\arraybackslash}X}}
\toprule
\footnotesize%
\multirow{2}{*}{\bfseries Attack} &
\multicolumn{3}{c}{\bfseries Failure Rate $\downarrow$} &
\multicolumn{3}{c}{\bfseries Average Queries $\downarrow$} &
\multicolumn{3}{c}{\bfseries Median Queries $\downarrow$} \\
\cmidrule(lr){2-4}\cmidrule(lr){5-7}\cmidrule(lr){8-10}
& I & R & V & I & R & V & I & R & V \\
\midrule
NES        & 1.6$\%$ & 0.2$\%$ & 0.1$\%$ &  1388.4 \tiny (1509.5) & 1044.4 \tiny (1053.8) & 728.9 \tiny (732.7) & 1072 & 715 & 154 \\
CH & 1.6$\%$ &  0.2 $\% $ & 0.1$\%$ & 1389.5 \tiny (1517.2) & 1044.3 \tiny (1053.4) & 729.6 \tiny (734.2) &  1123 & 715 & 154 \\
CBO        &  1.5$\%$  & 0.1$\%$ & 0.01$\%$ &    416.7 \tiny (560.5)  &  250.3 \tiny (259.0) &  
 139.6 \tiny (143.6) &   120  &   70 &  10 \\
\midrule
DFO$_c$ -- DiagonalCMA \cite{meunier2019yet} & 2.8$\%$ & 1.0$\%$ & 0.1$\%$ & 533 & 263 & 174 & 189 & 95 & 55 \\
DFO$_c$ -- CMA \cite{meunier2019yet}         & 0.8$\%$ & \bfseries 0.0$\%$ & 0.1$\%$ & 630 & 270 & 219 & 259 & 143 & 107 \\
DFO$_c$ -- Cauchy(1 + 1)-ES \cite{meunier2019yet}  & 2.7$\%$ & 0.4$\%$ & \bfseries 0.0$\%$ & 510 & 218 & 67  &  63 & 32 & 4 \\
DFO$_d$ -- DiagonalCMA \cite{meunier2019yet} &  2.3$\%$ & 1.2$\%$ & 0.5 $\%$  & 424 &  417 &  211 &  20&   20&   2  \\
\midrule%
ZO-NGD \cite{zhao2020towards} & 3.0$\%$ & --- & --- & 582 & --- & --- & --- & --- & --- \\
\midrule%
ZOO attack \cite{chen2017zoo} & 1.1$\%$ & --- & --- & 16,800 & --- & --- & --- & --- & --- \\
\midrule%
Bandits$_{\text{TD}}$ \cite{ilyas2018prior} & 4.6$\%$ &3.4$\%$ & 8.4$\%$ & 1117 & 722 & 370 & --- & --- & --- \\
\midrule%
SignHunter \cite{al2019there} & 1.0$\%$ &0.1$\%$ & 0.3$\%$ & 471 & 129 & 95 & 95 & 39 & 43 \\
\midrule%
Parsimonious \cite{moon2019parsimonious} & 1.5$\%$ &--- & --- & 722 & --- & --- & 237 & --- & --- \\
\midrule%
Square Attack \cite{ACFH2020square} &  \bfseries 0.3$\%$ & \bfseries 0.0 $\%$ & \bfseries 0.0 $\%$ & \bfseries 197. & \bfseries 73. & \bfseries 31. & \bfseries 24 & \bfseries 11 & \bfseries  1 \\
\bottomrule
\end{tabularx}
\caption{Performance of different optimizers for the untargeted attack problem using the low-resolution attack space in \cref{sec:lowres}. This is an extension of \cref{tab:lowres}, which compares the performance of CBO to more and also not directly related optimizers.}\label{tab:lowresmore}
\end{table}

\subsection{Further examination on the failure of CBO in the targeted ImageNet setting}\label{sec:tarfail}

In this section we analyze the targeted attack scenario in \cref{sec:lowres}, where CBO did not perform well. In \cref{fig:lowreshist} we compare the distribution of query counts that the different runs used. One observes that compared to CH and NES, CBO can achieve quick convergence more frequently. However, especially in the hard cases, CBO needs more queries, which results in an overall higher average and mean as observed in \cref{tab:low_res_tar}.
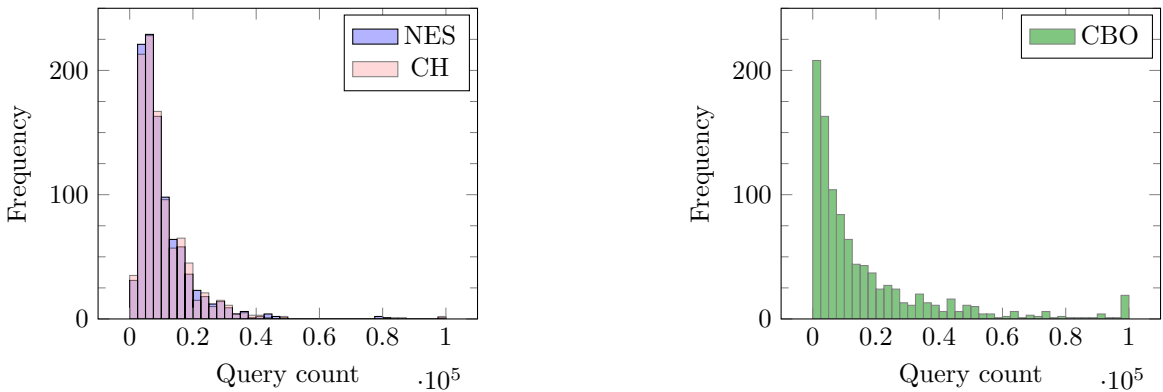
\begin{figure}
\begin{subfigure}[t]{.45\textwidth}
\centering%
\begin{tikzpicture}
\begin{axis}[
    ymin=0,
    ymax=250,
    minor y tick num = 3,
    area style,
    height = .25\textheight,
    xlabel = Query count,
    ylabel = Frequency,
    ]
\addplot+[ybar interval,mark=none, draw=black] plot coordinates {%
(0.0, 31) (2500.0, 221) (5000.0, 229) (7500.0, 163) (10000.0, 98) (12500.0, 64) (15000.0, 58) (17500.0, 36) (20000.0, 23) (22500.0, 18) (25000.0, 12) (27500.0, 14) (30000.0, 9) (32500.0, 4) (35000.0, 6) (37500.0, 1) (40000.0, 2) (42500.0, 4) (45000.0, 2) (47500.0, 1) (50000.0, 0) (52500.0, 0) (55000.0, 0) (57500.0, 0) (60000.0, 0) (62500.0, 0) (65000.0, 0) (67500.0, 0) (70000.0, 0) (72500.0, 0) (75000.0, 0) (77500.0, 2) (80000.0, 1) (82500.0, 0) (85000.0, 0) (87500.0, 0) (90000.0, 0) (92500.0, 0) (95000.0, 0) (97500.0, 1) (1e5,0)%
};
\addlegendentry{NES}
\addplot+[ybar interval,mark=none,opacity=0.5, draw=black] plot coordinates {%
(0.0, 35) (2500.0, 213) (5000.0, 228) (7500.0, 167) (10000.0, 96) (12500.0, 57) (15000.0, 65) (17500.0, 45) (20000.0, 15) (22500.0, 21) (25000.0, 10) (27500.0, 15) (30000.0, 11) (32500.0, 4) (35000.0, 5) (37500.0, 3) (40000.0, 3) (42500.0, 1) (45000.0, 0) (47500.0, 2) (50000.0, 0) (52500.0, 0) (55000.0, 0) (57500.0, 0) (60000.0, 0) (62500.0, 0) (65000.0, 0) (67500.0, 0) (70000.0, 0) (72500.0, 0) (75000.0, 0) (77500.0, 0) (80000.0, 0) (82500.0, 1) (85000.0, 1) (87500.0, 0) (90000.0, 0) (92500.0, 0) (95000.0, 0) (97500.0, 2) (1e5,0)
};
\addlegendentry{CH}
\end{axis}%
\end{tikzpicture}%
\end{subfigure}\hfill%
\begin{subfigure}[t]{.45\textwidth}
\centering%
\begin{tikzpicture}
\begin{axis}[
    ymin=0,
    ymax=250,
    minor y tick num = 3,
    area style,
    height = .25\textheight,
    xlabel = Query count,
    ylabel = Frequency
    ]
\addplot+[ybar interval,mark=none, fill=green!50!black, draw = black, opacity=0.5] plot coordinates {%
(0.0, 208) (2500.0, 163) (5000.0, 104) (7500.0, 84) (10000.0, 64) (12500.0, 44) (15000.0, 43) (17500.0, 37) (20000.0, 24) (22500.0, 27) (25000.0, 24) (27500.0, 13) (30000.0, 11) (32500.0, 20) (35000.0, 13) (37500.0, 11) (40000.0, 6) (42500.0, 16) (45000.0, 6) (47500.0, 11) (50000.0, 10) (52500.0, 4) (55000.0, 4) (57500.0, 1) (60000.0, 2) (62500.0, 6) (65000.0, 1) (67500.0, 3) (70000.0, 2) (72500.0, 6) (75000.0, 1) (77500.0, 2) (80000.0, 1) (82500.0, 1) (85000.0, 1) (87500.0, 1) (90000.0, 4) (92500.0, 1) (95000.0, 1) (97500.0, 19) (1e5,0)
};
\addlegendentry{CBO}
\end{axis}%
\end{tikzpicture}%
\end{subfigure}
\caption{Query distribution for the targeted attacks on ImageNet in the low resolution setting of \cref{tab:low_res_tar}. The histogram displays the queries obtained in the runs on the Inception architecture in \cref{tab:low_res_tar}.}\label{fig:lowreshist}
\end{figure}

Furthermore, in \cref{fig:pca} we compare the behavior of the consensus point between CH and CBO for a single attacked image. The low resolution attack space here is chosen as $\LS=[0,1]^{3 \times 20 \times 20}$. In order to visualize the behavior in this space with dimension $d=1200$, we follow the strategy used in \cite{li2018visualizing}. For each displayed run, we consider the sequence of consensus points $C=(c^{(0)}, \ldots, c^{(L)})$, which we then center around the end point, i.e., $\tilde{C} = (\tilde{c}^{(0)},\ldots,\tilde{c}^{(L)}) = C - c^{(L)}$. Furthermore, we subtract the data mean and then  consider $\bar{C} = \tilde{C} - \frac{1}{L}\sum_{k=0}^L \tilde{c}^{(k)}$. This sequence can be interpreted as a matrix $\bar{C} = (\bar{c}^{(0)},\ldots,\bar{c}^{(L)})\in\R^{d\times L}$, which allows us to perform a principled component analysis. Denoting by $V:\R^d\to\R^{\min\{d, L\}}$ the orthogonal operator that projects an element in $\R^d$ to its principal component representation, we can visualize the first two entries of $V(\bar{c}^{(k)})$ in \cref{fig:pca}. Moreover, we can also compute the error between $\tilde{c}^{(k)}$ and the projection onto the first two principal components, which is displayed as a gray value. The underlying loss landscape is obtained by evaluating the original objective function $\ell((f(\cdot), y)$ on the points $x^{i,j} = c^{(L)} + t_i v_1 + s_j v_2$, where $t_i, s_j\in\R$ are scalar coefficients and $v_1, v_2$ are the first two rows of the matrix representation of $V$. Furthermore, using the singular values $\sigma_i$ of $\bar{C}$ we can assign an importance score to each component as $\sigma_i^2/\sum_{j=1}^{\min{d,L}} \sigma_j^2$.
\begin{figure}
\centering%
\begin{subfigure}{.45\textwidth}
\includegraphics[width=\textwidth]{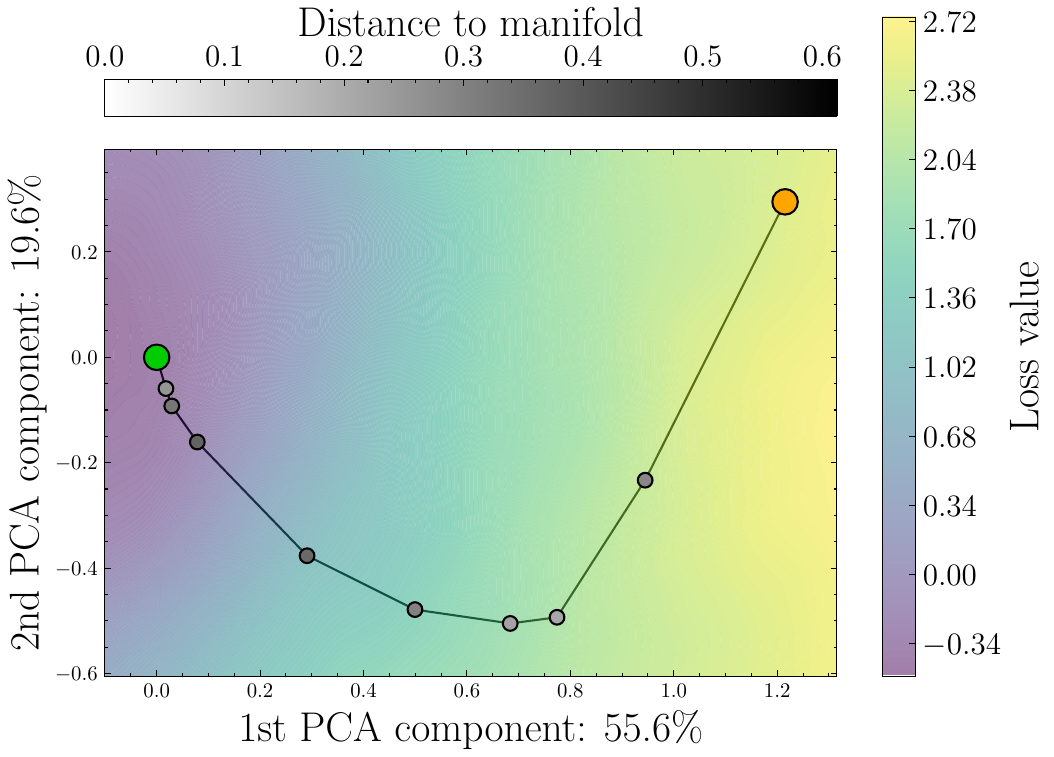}%
\caption{CBO, untargeted}
\end{subfigure}
\hfill%
\begin{subfigure}{.45\textwidth}
\includegraphics[width=\textwidth]{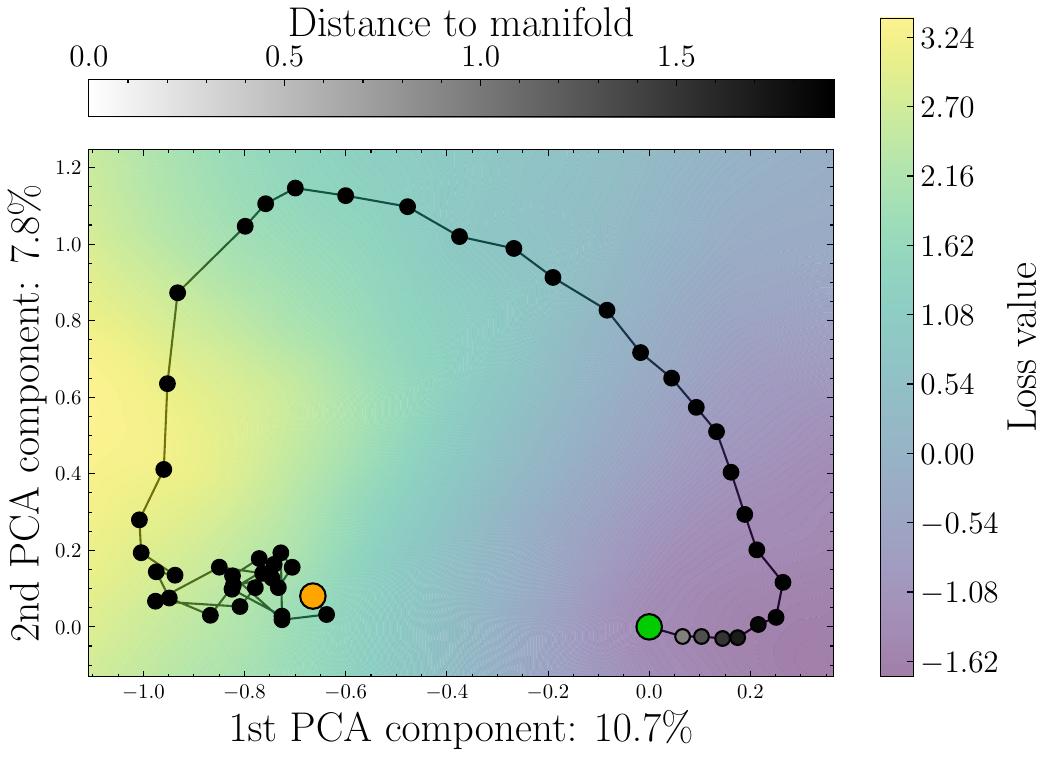}%
\caption{CH, untargeted}
\end{subfigure}
\\%
\vspace*{1cm}%
\begin{subfigure}{.45\textwidth}
\includegraphics[width=\textwidth]{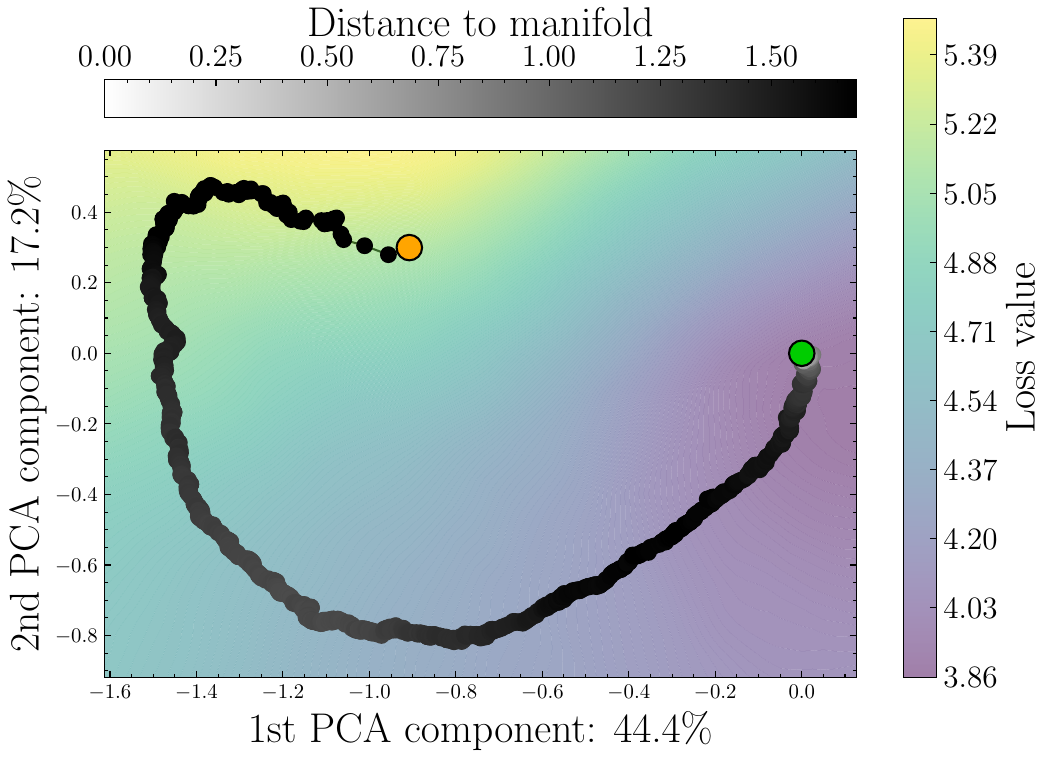}%
\caption{CBO, targeted}
\end{subfigure}
\hfill%
\begin{subfigure}{.45\textwidth}
\includegraphics[width=\textwidth]{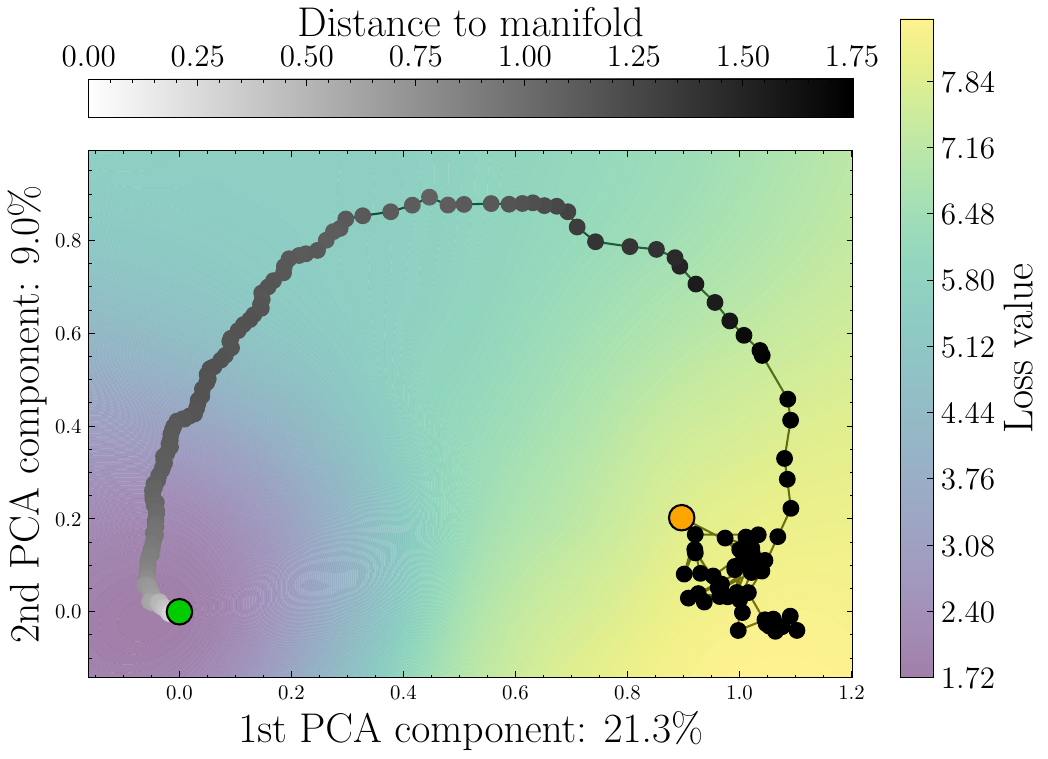}%
\caption{CH, targeted}
\end{subfigure}
\caption{Evolution of the consensus point for CBO and CH. Here we consider the untargeted and targeted low resolution attack on the image in \cref{fig:lowres-ex}. The dynamic is visualized via a PCA as described in \cref{sec:tarfail}, following the concepts of \cite{li2018visualizing}.}\label{fig:pca}
\end{figure}
In \cref{fig:pca} we first observe that for the untargeted attack CBO achieves convergence within few iterates, following a comparably smooth path. Furthermore, the first two principal components explain $75.2\%$ of the variance, which means that the consensus point mainly moves along a low-dimensional manifold. Compared to that, in CH the first two components only explain $18.5\%$ of the variance and the path is less direct. In this case, CBO is significantly more efficient by essentially moving along a low-dimensional path.

For the untargeted attack, we again see that the first two components are more important in CBO than in CH. However, the path in CBO is less direct and also the difference between consecutive iterates in the PCA basis is less. For CH, we observe an oscillatory behavior at the start of the iteration. This can be attributed to the large initial step size that is then annealed with the scheduling described in \cref{alg:CHNES}.

\subsection{Further experiments on $P$-pixel attacks}

In \cref{fig:numpix} we evaluate the attack performance when varying the number of attacked pixels. As expected, the success rate increases with the number of pixels. 

Furthermore, in \cref{fig:onepix}, we plot the optimization landscape in the coordinate variable and display a single run of CBO.
\begin{figure}
\begin{subfigure}[t]{.53\textwidth}

\centering
\begin{tikzpicture}
\begin{axis}[tinyticks, legend pos=north east,
clip bounding box=upper bound,
ytick={0, 10, 20, 30, 40, 50},
yticklabels ={0$\%$, 10$\%$,  20$\%$, 30$\%$, 40$\%$},
ymin=0, 
xmin=-2,
xmax=205,
xtick=data,
xlabel = Number of pixels $P$,
ylabel= Success rate $\uparrow$,
width=\textwidth,height=.2\textheight]
\addplot table [x=n, y=p, col sep=space] {numpix.csv};\label{rate}
\end{axis}
\end{tikzpicture}
\caption{Success rates for different numbers of $P$ for the $P$-pixel attack. We use CBO to attack $100$ images using the AlexNet architecture with a maximum budget of $Q=10,000$ queries.}
\label{fig:numpix}
\end{subfigure}%
\hfill%
\begin{subfigure}[t]{.35\textwidth}
\centering%
\includegraphics[width=.9\textwidth]{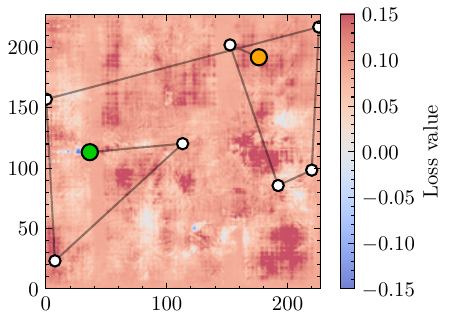}
\caption{
Loss landscape of the pixel position for a $1$-pixel attack.
}\label{fig:onepix}
\end{subfigure}
\caption{Success rate for the $P$-pixel attack for different $P$ and loss landscape of the coordinate variable. For \cref{fig:onepix}, we compute an attack with CBO and then fix the channel values of the last iterates for the visualization. On top of that, we show the evolution of the consensus point, starting at the orange point and ending at the green one.}
\end{figure}
\subsection{Further experiments on the alternative noise model}\label{sec:appsquare}

In \cref{sec:dct} and \cref{sec:square} we explore the idea of adapting the noise model in the CBO algorithm. In \cref{alg:CBO}, the standard Gaussian noise is replaced by a function motivated by the respective attack space. This is a significant modification of the original CBO model and does not directly allow for the standard convergence analysis and mean field interpretation. Beyond that, the question arises, whether this modification makes the consensus dynamic obsolete. I.e., it might be that each particle performs a random search on its own. In order to examine this hypothesis and whether the drift towards the consensus point is beneficial for the algorithm we study the effect of the drift parameter $\lambda$ in \cref{alg:CBO}, with the noise model of \cref{alg:squarenoise}. Setting $\lambda=0.$, corresponds to no interaction between the particles, i.e., a random walk. In particular, note, that in the setup of \cref{alg:squarenoise} the noise is not scaled by the distance to the consensus point. In \cref{fig:lamdadrift} we observe that choosing $\lambda>0$ (by default we have $\lambda=1.$) does in fact yield better results.

\begin{figure}

\centering
\begin{tikzpicture}
\begin{axis}[tinyticks, legend pos=north east,
clip bounding box=upper bound,
ymin=0, 
xtick=data,
xlabel = Drift parameter $\lambda$ in \cref{alg:CBO},
ylabel= Average number of queries $\downarrow$,
width=\textwidth,height=.2\textheight]
\addplot table [x=n, y=p, col sep=space] {lamda.csv};\label{rate}
\end{axis}
\end{tikzpicture}
\caption{We perform $50$ untargeted attacks on ImageNet employing (I), in the square attack scenario of \cref{sec:square}. The optimizer here is CBO with the square noise, as described in \cref{alg:squarenoise}. We vary the drift parameter $\lambda$ and show the average number of queries over all attacks, with a query limit of $Q=10,000$.}
\label{fig:lamdadrift}
\end{figure}
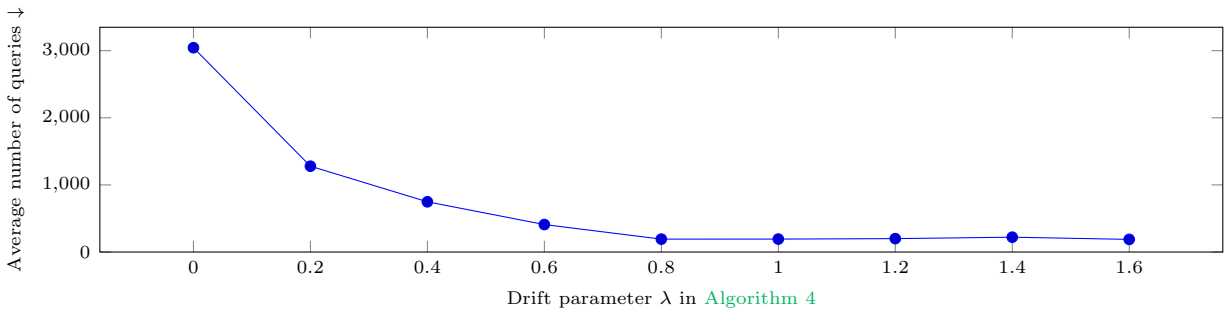%
Moreover, especially for the square attack space, one might ask if the computation of the consensus point with this noise model is even meaningful. In CBO one assumes that computing a weighted average of the existing particles can result in point that has a better objective value. But with the particles produced with, \cref{alg:squarenoise} this is not clear a priori. Compared to that the algorithm proposed in \cite{ACFH2020square} does not consider averages over different squares.

To examine this, we perform the following experiment: we initialize $5$ particles with the strategy described in \cref{sec:square}. Then we compute the consensus point $c$ for different values of $\alpha$ and plot the loss value of $\ell(f(c),y)$ in \cref{fig:squarealpha} together with the corresponding consensus point. For lower values of $\alpha$, we observe that the averaging makes the squares less visible. In particular, the squares do not use the full $\ell^\infty$-budget $\budget$. When increasing $\alpha$, the consensus point tends toward to $\argmin$ over the particles. This results in a better loss value.

\begin{figure}
\centering%
\begin{tikzpicture}
\begin{axis}[width=\textwidth,
height=.5\textwidth,grid=major,
legend pos=outer north east,
legend cell align={left}, xlabel=$\alpha$,
ylabel=Loss value $\downarrow$,
axis x line*=bottom,
axis y line*=left,
x label style={anchor=west},
y label style={anchor=south},
xmin=-0.5,
name=aa
]
%
%
\addplot[line width=1.5pt, color=blue,opacity=1., mark=*,mark options={solid}]  table[x index=0, y index=1]{squarea.csv};
\label{e1}
\tikzstyle{every pin}=[
    draw=blue,
    fill=white,
    thick
]
\tikzset{every pin edge/.style={draw=blue, thick}}
\node [
coordinate,
pin={[pin distance=2cm]below:{\includegraphics[width=1.5cm]{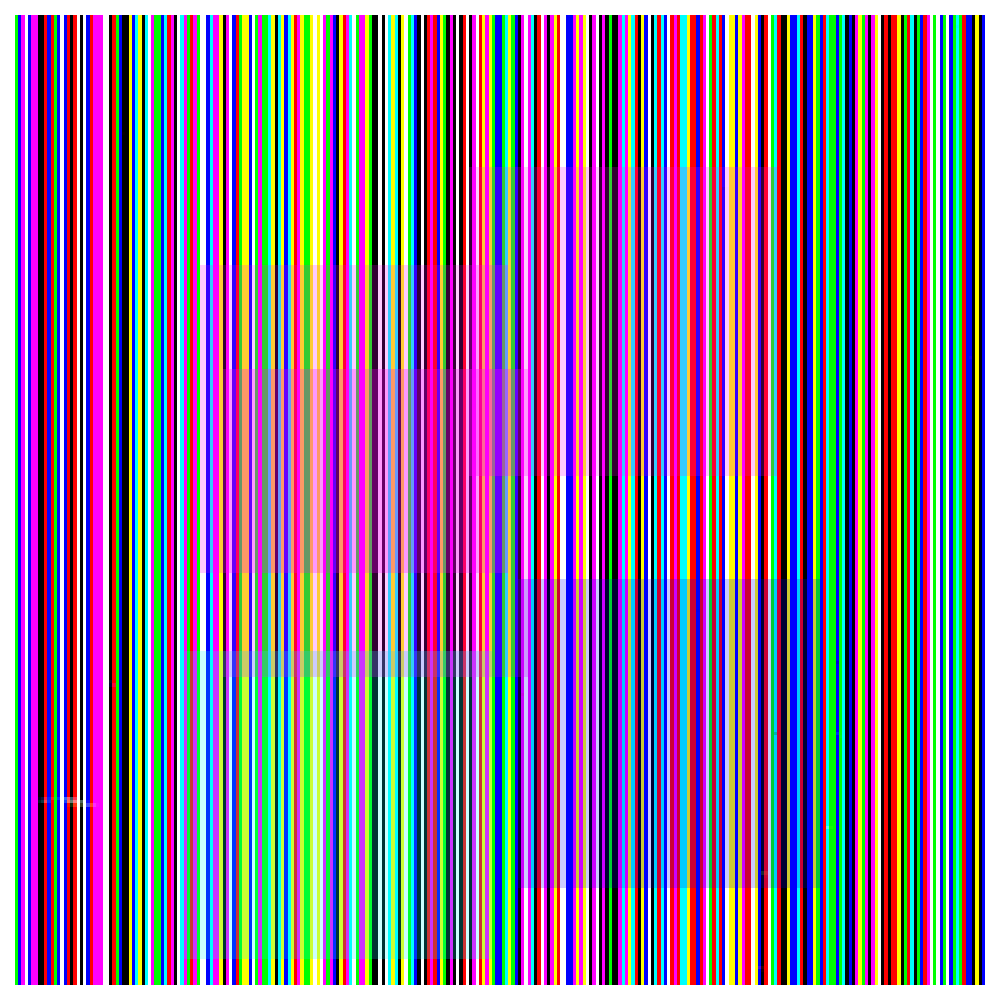}}},
] at (axis cs:0.0,2.8287720680236816)   {};
\node [
coordinate,
pin={[pin distance=2cm]above:{\includegraphics[width=1.5cm]{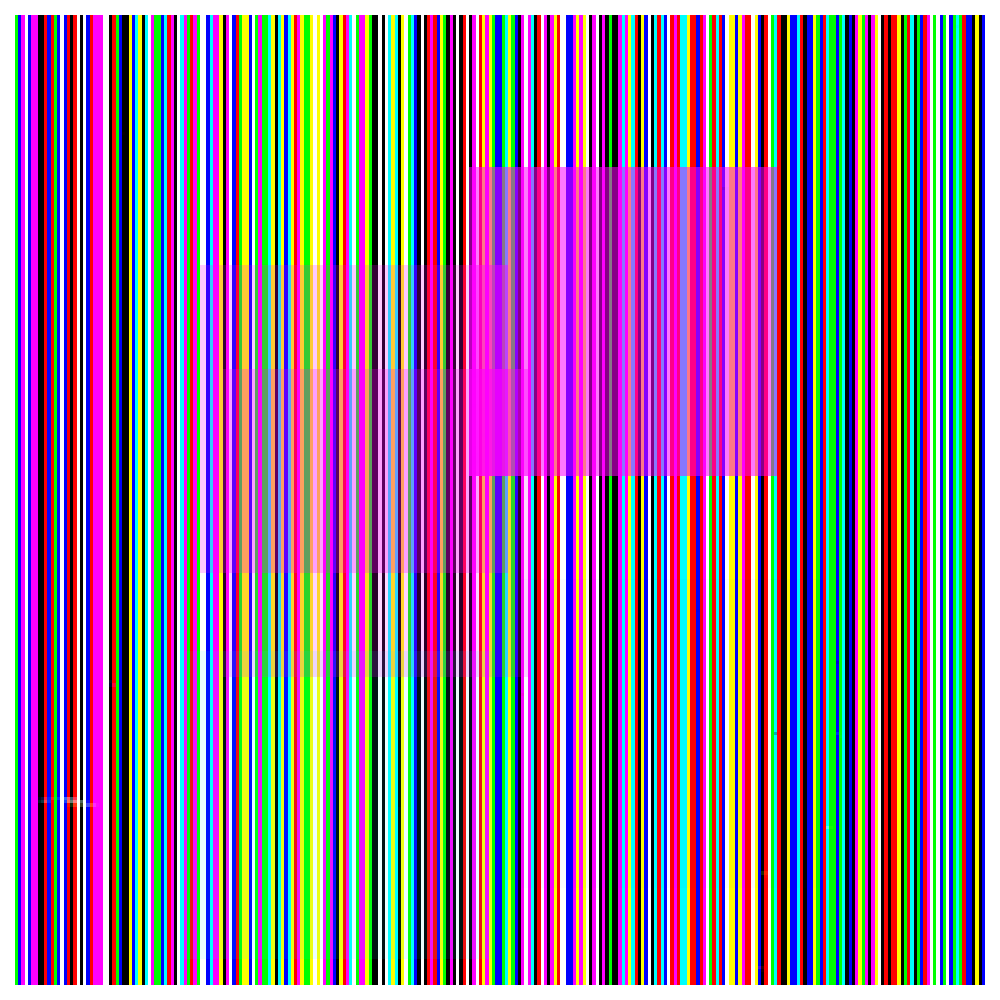}}},
] at (axis cs:0.5,2.815850257873535)   {};
\node [
coordinate,
pin={[pin distance=2cm]above:{\includegraphics[width=1.5cm]{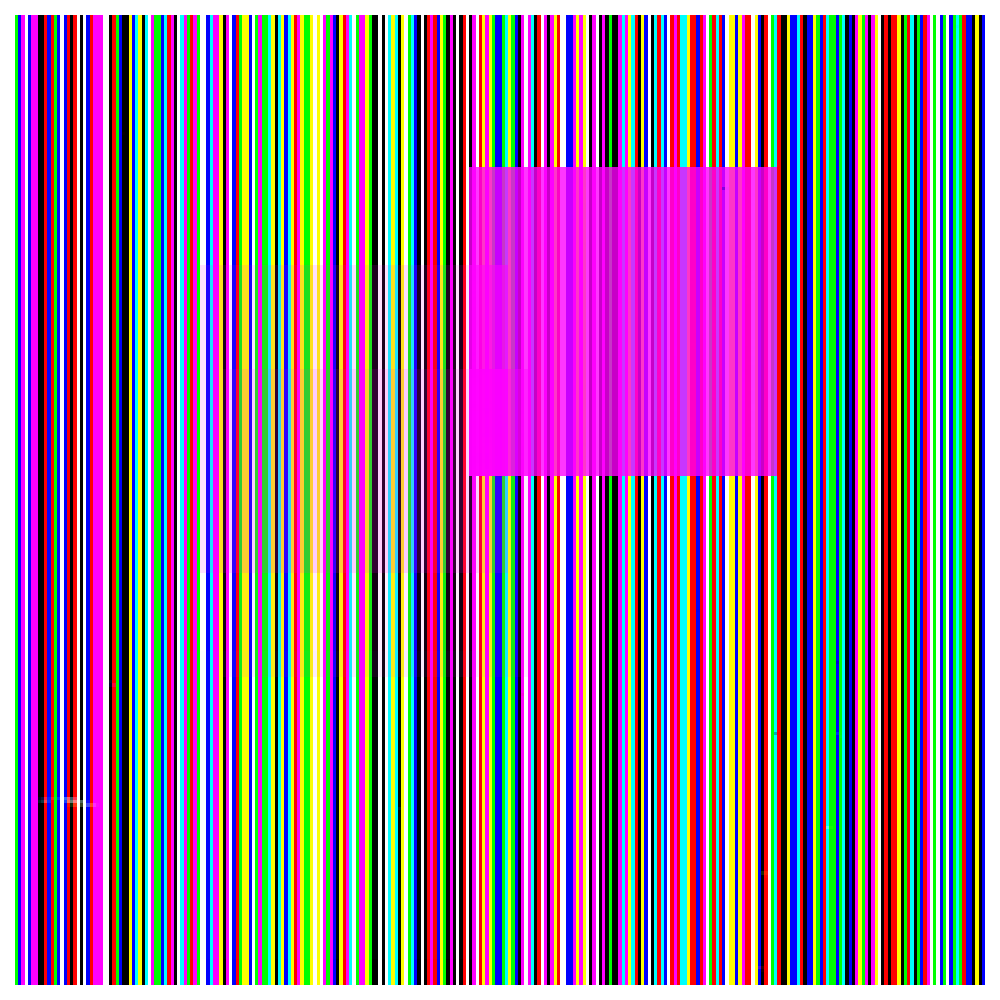}}},
] at (axis cs:1.0,2.8142032623291016)   {};
\node [
coordinate,
pin={[pin distance=2cm]above:{\includegraphics[width=1.5cm]{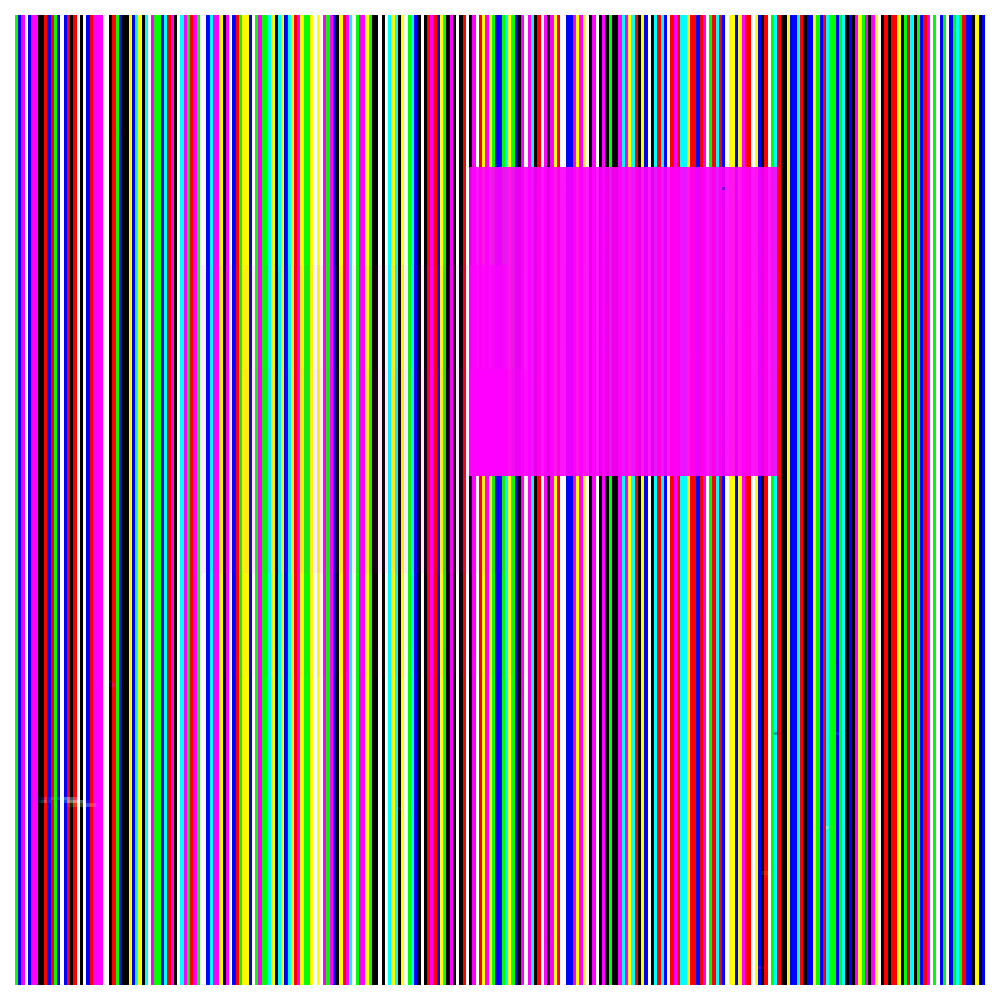}}},
] at (axis cs:1.5,2.812264919281006)   {};
\node [
coordinate,
pin={[pin distance=2cm]above:{\includegraphics[width=1.5cm]{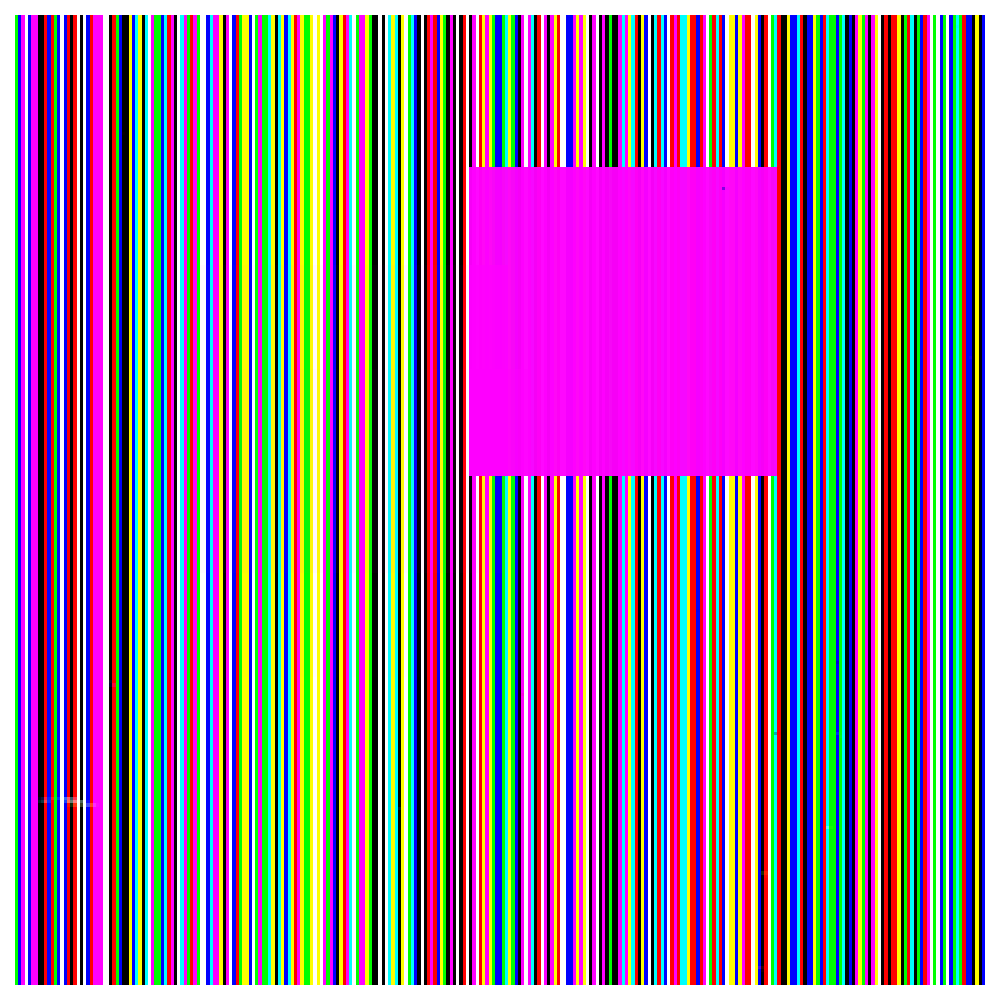}}},
] at (axis cs:2.0,2.8116331100463867)   {};
\end{axis}
\end{tikzpicture}
\caption{For the square noise as described in \cref{sec:square} a larger $\alpha$ improves the attack strength.}\label{fig:squarealpha}
\end{figure}

\end{document}